\documentclass{amsart}
\usepackage{amssymb}
\usepackage{amsfonts}
\usepackage{geometry}

\newtheorem{theorem}{Theorem}
\theoremstyle{plain}

\newtheorem{definition}{Definition}
\newtheorem{example}{Example}

\newtheorem{lemma}{Lemma}
\newtheorem{notation}{Notation}

\newtheorem{proposition}{Proposition}
\newtheorem{remark}{Remark}

\input{tcilatex}
\begin{document}
\title[Two weight boundedness]{A Good-$\lambda $ Lemma, two weight $T1$
theorems without weak boundedness, and a two weight accretive global $Tb$
theorem}
\author[E.T. Sawyer]{Eric T. Sawyer}
\address{ Department of Mathematics \& Statistics, McMaster University, 1280
Main Street West, Hamilton, Ontario, Canada L8S 4K1 }
\email{sawyer@mcmaster.ca}
\thanks{Research supported in part by NSERC}
\author[C.-Y. Shen]{Chun-Yen Shen}
\address{ Department of Mathematics \\
National Central University \\
Chungli, 32054, Taiwan }
\email{cyshen@math.ncu.edu.tw}
\thanks{C.-Y. Shen supported in part by the NSC, through grant NSC
104-2628-M-008 -003 -MY4}
\author[I. Uriarte-Tuero]{Ignacio Uriarte-Tuero}
\address{ Department of Mathematics \\
Michigan State University \\
East Lansing MI }
\email{ignacio@math.msu.edu}
\thanks{ I. Uriarte-Tuero has been partially supported by grants DMS-1056965
(US NSF), MTM2010-16232, MTM2015-65792-P (MINECO, Spain), and a Sloan
Foundation Fellowship. }
\date{September 12, 2016}

\begin{abstract}
Let $\sigma $ and $\omega $ be locally finite positive Borel measures on $%
\mathbb{R}^{n}$, let $T^{\alpha }$\ be a standard $\alpha $-fractional Calder%
\'{o}n-Zygmund operator on $\mathbb{R}^{n}$ with $0\leq \alpha <n$, and
assume as side conditions the $\mathcal{A}_{2}^{\alpha }$ conditions,
punctured $A_{2}^{\alpha }$ conditions, and certain $\alpha $\emph{-energy
conditions}. Then the weak boundedness property associated with the operator 
$T^{\alpha }$ and the weight pair $\left( \sigma ,\omega \right) $, is `good-%
$\lambda $' controlled by the testing conditions and the Muckenhoupt and
energy\ conditions. As a consequence, assuming the side conditions, we can
eliminate the weak boundedness property from Theorem 1 of \cite{SaShUr9} to
obtain that $T^{\alpha }$ is bounded from $L^{2}\left( \sigma \right) $ to $%
L^{2}\left( \omega \right) $ if and only if the testing conditions hold for $%
T^{\alpha }$\textbf{\ }and its dual. As a corollary we give a simple
derivation of a two weight accretive global $Tb$ theorem from a related $T1$
theorem. The role of two different parameterizations of the family of dyadic
grids, by scale and by translation, is highlighted in simultaneously
exploiting both goodness and NTV surgery with families of grids that are 
\emph{common} to both measures.
\end{abstract}

\maketitle
\tableofcontents

\begin{description}
\item[Dedication] This paper is dedicated to Dick Wheeden on the occasion of
his retirement from Rutgers University, and for all of his fundamental
contributions to the theory of weighted inequalities, in particular for the
beautiful paper of Hunt, Muckenhoupt and Wheeden that started it all back in
1973.
\end{description}

\section{Introduction}

The theory of weighted norm inequalities burst into the general mathematical
consciousness with the celebrated theorem of Hunt, Muckenhoupt and Wheeden 
\cite{HuMuWh} that extended boundedness of the Hilbert transform to measures
more general than Lebesgue's, namely showing that $H$ was bounded on the
weighted space $L^{2}\left( \mathbb{R}^{n};w\right) $ if and only if the $%
A_{2}$ condition of Muckenhoupt,%
\begin{equation*}
\left( \frac{1}{\left\vert Q\right\vert }\int_{Q}w\left( x\right) dx\right)
\left( \frac{1}{\left\vert Q\right\vert }\int_{Q}\frac{1}{w\left( x\right) }%
dx\right) \lesssim 1\ ,
\end{equation*}%
holds when taken uniformly over all cubes $Q$ in $\mathbb{R}^{n}$. The
ensuing thread of investigation culminated in the theorem of Coifman and
Fefferman \cite{CoFe}\ that characterized those nonnegative weights $w$ on $%
\mathbb{R}^{n}$ for which all of the `nicest' of the $L^{2}\left( \mathbb{R}%
^{n}\right) $ bounded singular integrals $T$ above are bounded on weighted
spaces $L^{2}\left( \mathbb{R}^{n};w\right) $, and does so in terms of the
above $A_{2}$ condition of Muckenhoupt.

Attention then turned to the corresponding two weight inequalities for
singular integrals, which turned out to be considerably more complicated.
For example, Cotlar and Sadosky gave a beautiful function theoretic
characterization of the weight pairs $\left( \sigma ,\omega \right) $ for
which $H$ is bounded from $L^{2}\left( \mathbb{R};\sigma \right) $ to $%
L^{2}\left( \mathbb{R};\omega \right) $, namely a two-weight extension of
the Helson-Szeg\"{o} theorem, which illuminated a deep connection between
two quite different function theoretic conditions, but failed to shed much
light on when either of them held\footnote{%
However, the testing conditions in Theorem \ref{Hilbert} are subject to the
same criticism due to the highly unstable nature of singular integrals
acting on measures.}. On the other hand, the two weight inequality for
positive fractional integrals, Poisson integrals and maximal functions were
characterized using testing conditions by one of us in \cite{Saw3} (see also 
\cite{Hyt2} for the Poisson inequality with `holes') and \cite{Saw1}, but
relying in a very strong way on the positivity of the kernel, something the
Hilbert kernel lacks. In a groundbreaking series of papers including \cite%
{NTV1},\cite{NTV2} and \cite{NTV3}, Nazarov, Treil and Volberg used weighted
Haar decompositions with random grids, introduced their `pivotal' condition,
and proved that the Hilbert transform is bounded from $L^{2}\left( \mathbb{R}%
;\sigma \right) $ to $L^{2}\left( \mathbb{R};\omega \right) $ if and only if
a variant of the $A_{2}$ condition `on steroids' held, and the norm
inequality and its dual held when tested locally over indicators of cubes -
but \textbf{only} under the side assumption that their pivotal conditions
held.

The last dozen years have seen a resurgence in the investigation of two
weight inequalities for singular integrals, beginning with the
aforementioned work of NTV, and due in part to applications of the two
weight $T1$ theorem in operator theory, such as in \cite{LaSaShUrWi}, where
embedding measures are characterized for model spaces $K_{\theta }$, where $%
\theta $ is an inner function on the disk, and where norms of composition
operators are characterized that map $K_{\theta }$ into Hardy and Bergman
spaces. A $T1$ theorem could also have implications for a number of problems
that are higher dimensional analogues of those connected to the Hilbert
transform (rank one perturbations \cite{Vol}, \cite{NiTr}); products of two
densely defined Toeplitz operators; subspaces of the Hardy space invariant
under the inverse shift operator \cite{Vol}, \cite{NaVo}; orthogonal
polynomials \cite{VoYu}, \cite{PeVoYu}; and quasiconformal theory \cite{IwMa}%
, \cite{LaSaUr}, \cite{AsGo}, \cite{AsZi}), and we refer the reader to \cite%
{SaShUr10} for more detail on these applications.

Following the groundbreaking work of Nazarov, Treil and Volberg, two of us,
Sawyer and Uriarte-Tuero, together with Lacey in \cite{LaSaUr2}, showed that
the pivotal conditions were not necessary in general, and introduced instead
a necessary `energy' condition as a substitute, along with a hybrid merging
of these two conditions that was shown to be sufficient for use as a side
condition. The resurgence was then capped along the way with a resolution -
involving the work of Nazarov, Treil and Volberg in \cite{NTV3}, the authors
and M. Lacey in the two part paper \cite{LaSaShUr3}, \cite{Lac} and T. Hyt%
\"{o}nen in \cite{Hyt2} - of the two weight Hilbert transform conjecture of
Nazarov, Treil and Volberg (\cite{Vol}):

\begin{theorem}
\label{Hilbert}The Hilbert transform is bounded from $L^{2}\left( \mathbb{R}%
;\sigma \right) $ to $L^{2}\left( \mathbb{R};\omega \right) $, i.e.%
\begin{equation}
\left\Vert H\left( f\sigma \right) \right\Vert _{L^{2}\left( \mathbb{R}%
;\omega \right) }\lesssim \left\Vert f\right\Vert _{L^{2}\left( \mathbb{R}%
;\sigma \right) },\ \ \ \ \ f\in L^{2}\left( \mathbb{R};\sigma \right) ,
\label{Hilbert'}
\end{equation}%
if and only if the two weight $A_{2}$ condition with holes holds,%
\begin{equation*}
\frac{\left\vert Q\right\vert _{\sigma }}{\left\vert Q\right\vert }\left( 
\frac{1}{\left\vert Q\right\vert }\int_{\mathbb{R}\setminus Q}\mathbf{s}%
_{Q}^{2}d\omega \left( x\right) \right) +\left( \frac{1}{\left\vert
Q\right\vert }\int_{\mathbb{R}\setminus Q}\mathbf{s}_{Q}^{2}d\sigma \left(
x\right) \right) \frac{\left\vert Q\right\vert _{\omega }}{\left\vert
Q\right\vert }\lesssim 1\ ,
\end{equation*}%
uniformly over all cubes $Q$, and the two testing conditions hold,%
\begin{eqnarray*}
\left\Vert \mathbf{1}_{Q}H\left( \mathbf{1}_{Q}\sigma \right) \right\Vert
_{L^{2}\left( \mathbb{R};\omega \right) } &\lesssim &\left\Vert \mathbf{1}%
_{Q}\right\Vert _{L^{2}\left( \mathbb{R};\sigma \right) }=\sqrt{\left\vert
Q\right\vert _{\sigma }}\ , \\
\left\Vert \mathbf{1}_{Q}H^{\ast }\left( \mathbf{1}_{Q}\omega \right)
\right\Vert _{L^{2}\left( \mathbb{R};\sigma \right) } &\lesssim &\left\Vert 
\mathbf{1}_{Q}\right\Vert _{L^{2}\left( \mathbb{R};\omega \right) }=\sqrt{%
\left\vert Q\right\vert _{\omega }}\ ,
\end{eqnarray*}%
uniformly over all cubes $Q$.
\end{theorem}

Here $Hf\left( x\right) =\int_{\mathbb{R}}\frac{f\left( y\right) }{y-x}dy$
is the Hilbert transform on the real line $\mathbb{R}$, and $\sigma $ and $%
\omega $ are locally finite positive Borel measures on $\mathbb{R}$. The two
weight $A_{2}$ condition with holes is also a testing condition in disguise,
in particular it follows from%
\begin{equation*}
\left\Vert H\left( \mathbf{s}_{Q}\sigma \right) \right\Vert _{L^{2}\left( 
\mathbb{R};\omega \right) }\lesssim \left\Vert \mathbf{s}_{Q}\right\Vert
_{L^{2}\left( \mathbb{R};\sigma \right) }\ ,
\end{equation*}%
tested over all `indicators with tails' $\mathbf{s}_{Q}\left( x\right) =%
\frac{\ell \left( Q\right) }{\ell \left( Q\right) +\left\vert
x-c_{Q}\right\vert }$ of intervals $Q$ in $\mathbb{R}$. Below we discuss the
precise interpretation of the above inequalities involving the singular
integral $H$.

At this juncture, attention naturally turned to the analogous two weight
inequalities for \emph{higher dimensional} singular integrals, as well as $%
\alpha $-fractional singular integrals such as the Cauchy transform in the
plane. A variety of results were obtained, e.g. \cite{SaShUr8}, \cite%
{LaSaShUrWi}, \cite{LaWi} and \cite{SaShUr9}, in which a $T1$ theorem was
proved under certain side conditions that implied the energy conditions.
However, in \cite{SaShUr10}, the authors have recently shown that the energy
conditions are \emph{not} in general necessary for elliptic singular
integrals.

The aforementioned higher dimensional results require refinements of the
various one-dimensional conditions associated with the norm inequalities,
namely the $A_{2}$ conditions, the testing conditions, the weak boundedness
property and energy conditions. The purpose of this paper is to prove in
higher dimensions that the weak boundedness constant $\mathcal{WBP}%
_{T^{\alpha }}\left( \sigma ,\omega \right) $ that is associated with an $%
\alpha $-fractional singular integral $T^{\alpha }$ and a weight pair $%
\left( \sigma ,\omega \right) $ in $\mathbb{R}^{n}$, is `good-$\lambda $'
controlled by the usual testing conditions $\mathfrak{T}_{T^{\alpha }}\left(
\sigma ,\omega \right) $, $\mathfrak{T}_{T^{\alpha }}^{\ast }\left( \sigma
,\omega \right) $ and two side conditions on weight pairs, namely the
Muckenhoupt conditions $\mathfrak{A}_{2}^{\alpha }\left( \sigma ,\omega
\right) $ and the energy conditions $\mathcal{E}_{\alpha }^{\limfunc{strong}%
}\left( \sigma ,\omega \right) $, $\mathcal{E}_{\alpha }^{\limfunc{strong}%
,\ast }\left( \sigma ,\omega \right) $: more precisely, for every $0<\lambda
<\frac{1}{2}$, we have the Good-$\lambda $\ Lemma:%
\begin{equation*}
\mathcal{WBP}_{T^{\alpha }}\left( \sigma ,\omega \right) \leq C_{\alpha
}\left( \frac{1}{\lambda }\sqrt{\mathfrak{A}_{2}^{\alpha }}+\mathfrak{T}%
_{T^{\alpha }}+\mathfrak{T}_{T^{\alpha }}^{\ast }+\mathcal{E}_{\alpha }^{%
\limfunc{strong}}+\mathcal{E}_{\alpha }^{\limfunc{strong},\ast }+\sqrt[4]{%
\lambda }\mathfrak{N}_{T^{\alpha }}\right) .
\end{equation*}%
The first instance of this type of conclusion appears in Lacey and Wick in 
\cite{LaWi}) - see Remark \ref{LacWic} in Subsection \ref{Sub cor}\ below.

Applications of the Good-$\lambda $\ Lemma are then given to obtain both $T1$
and $Tb$ theorems for two weights. We now turn to a description of the
higher dimensional conditions appearing in the above display. As the Good-$%
\lambda $ Lemma, along with its corollaries, hold in the more general
setting of quasicubes, we describe them first. But the reader interested
only in cubes can safely ignore this largely cosmetic generalization (but
crucial for our `measure on a curve' $T1$ theorem in \cite{SaShUr8}) by
simply deleting the prefix `quasi' wherever it appears.

\subsection{Quasicubes}

We begin by recalling the notion of quasicube used in \cite{SaShUr9} - a
special case of the classical notion used in quasiconformal theory.

\begin{definition}
We say that a homeomorphism $\Omega :\mathbb{R}^{n}\rightarrow \mathbb{R}%
^{n} $ is a globally biLipschitz map if%
\begin{equation}
\left\Vert \Omega \right\Vert _{Lip}\equiv \sup_{x,y\in \mathbb{R}^{n}}\frac{%
\left\Vert \Omega \left( x\right) -\Omega \left( y\right) \right\Vert }{%
\left\Vert x-y\right\Vert }<\infty ,  \label{rigid}
\end{equation}%
and $\left\Vert \Omega ^{-1}\right\Vert _{Lip}<\infty $.
\end{definition}

\begin{notation}
We define $\mathcal{P}^{n}$ to be the collection of half open, half closed
cubes in $\mathbb{R}^{n}$ with sides parallel to the coordinate axes. A half
open, half closed cube $Q$ in $\mathbb{R}^{n}$ has the form $Q=Q\left(
c,\ell \right) \equiv \dprod\limits_{k=1}^{n}\left[ c_{k}-\frac{\ell }{2}%
,c_{k}+\frac{\ell }{2}\right) $ for some $\ell >0$ and $c=\left(
c_{1},...,c_{n}\right) \in \mathbb{R}^{n}$. The cube $Q\left( c,\ell \right) 
$ is described as having center $c$ and sidelength $\ell $.
\end{notation}

\begin{definition}
Suppose that $\Omega :\mathbb{R}^{n}\rightarrow \mathbb{R}^{n}$ is a
globally biLipschitz map.

\begin{enumerate}
\item If $E$ is a measurable subset of $\mathbb{R}^{n}$, we define $\Omega
E\equiv \left\{ \Omega \left( x\right) :x\in E\right\} $ to be the image of $%
E$ under the homeomorphism $\Omega $.

\begin{enumerate}
\item In the special case that $E=Q$ is a cube in $\mathbb{R}^{n}$, we will
refer to $\Omega Q$ as a quasicube (or $\Omega $-quasicube if $\Omega $ is
not clear from the context).

\item We define the center $c_{\Omega Q}=c\left( \Omega Q\right) $ of the
quasicube $\Omega Q$ to be the point $\Omega c_{Q}$ where $c_{Q}=c\left(
Q\right) $ is the center of $Q$.

\item We define the side length $\ell \left( \Omega Q\right) $ of the
quasicube $\Omega Q$ to be the sidelength $\ell \left( Q\right) $ of the
cube $Q$.

\item For $r>0$ we define the `dilation' $r\Omega Q$ of a quasicube $\Omega
Q $ to be $\Omega rQ$ where $rQ$ is the usual `dilation' of a cube in $%
\mathbb{R}^{n}$ that is concentric with $Q$ and having side length $r\ell
\left( Q\right) $.
\end{enumerate}

\item If $\mathcal{K}$ is a collection of cubes in $\mathbb{R}^{n}$, we
define $\Omega \mathcal{K}\equiv \left\{ \Omega Q:Q\in \mathcal{K}\right\} $
to be the collection of quasicubes $\Omega Q$ as $Q$ ranges over $\mathcal{K}
$.

\item If $\mathcal{F}$ is a grid of cubes in $\mathbb{R}^{n}$, we define the
inherited quasigrid structure on $\Omega \mathcal{F}$ by declaring that $%
\Omega Q$ is a child of $\Omega Q^{\prime }$ in $\Omega \mathcal{F}$ if $Q$
is a child of $Q^{\prime }$ in the grid $\mathcal{F}$.
\end{enumerate}
\end{definition}

Note that if $\Omega Q$ is a quasicube, then $\left\vert \Omega Q\right\vert
^{\frac{1}{n}}\approx \left\vert Q\right\vert ^{\frac{1}{n}}=\ell \left(
Q\right) =\ell \left( \Omega Q\right) $. For a quasicube $J=\Omega Q$, we
will generally use the expression $\left\vert J\right\vert ^{\frac{1}{n}}$
in the various estimates arising in the proofs below, but will often use $%
\ell \left( J\right) $ when defining collections of quasicubes. Moreover,
there are constants $R_{big}$ and $R_{small}$ such that we have the
comparability containments%
\begin{equation*}
Q+\Omega x_{Q}\subset R_{big}\Omega Q\text{ and }R_{small}\Omega Q\subset
Q+\Omega x_{Q}\ .
\end{equation*}

\begin{example}
\label{wild}Quasicubes can be wildly shaped, as illustrated by the standard
example of a logarithmic spiral in the plane $f_{\varepsilon }\left(
z\right) =z\left\vert z\right\vert ^{2\varepsilon i}=ze^{i\varepsilon \ln
\left( z\overline{z}\right) }$. Indeed, $f_{\varepsilon }:\mathbb{%
C\rightarrow C}$ is a globally biLipschitz map with Lipschitz constant $%
1+C\varepsilon $ since $f_{\varepsilon }^{-1}\left( w\right) =w\left\vert
w\right\vert ^{-2\varepsilon i}$ and%
\begin{equation*}
\nabla f_{\varepsilon }=\left( \frac{\partial f_{\varepsilon }}{\partial z},%
\frac{\partial f_{\varepsilon }}{\partial \overline{z}}\right) =\left(
\left\vert z\right\vert ^{2\varepsilon i}+i\varepsilon \left\vert
z\right\vert ^{2\varepsilon i},i\varepsilon \frac{z}{\overline{z}}\left\vert
z\right\vert ^{2\varepsilon i}\right) .
\end{equation*}%
On the other hand, $f_{\varepsilon }$ behaves wildly at the origin since the
image of the closed unit interval on the real line under $f_{\varepsilon }$
is an infinite logarithmic spiral.
\end{example}

\subsection{Standard fractional singular integrals and the norm inequality}

Let $0\leq \alpha <n$. We define a standard $\alpha $-fractional CZ kernel $%
K^{\alpha }(x,y)$ to be a real-valued function defined on $\mathbb{R}%
^{n}\times \mathbb{R}^{n}$ satisfying the following fractional size and
smoothness conditions of order $1+\delta $ for some $\delta >0$: For $x\neq
y $,%
\begin{eqnarray}
\left\vert K^{\alpha }\left( x,y\right) \right\vert &\leq &C_{CZ}\left\vert
x-y\right\vert ^{\alpha -n}\text{ and }\left\vert \nabla K^{\alpha }\left(
x,y\right) \right\vert \leq C_{CZ}\left\vert x-y\right\vert ^{\alpha -n-1},
\label{sizeandsmoothness'} \\
\left\vert \nabla K^{\alpha }\left( x,y\right) -\nabla K^{\alpha }\left(
x^{\prime },y\right) \right\vert &\leq &C_{CZ}\left( \frac{\left\vert
x-x^{\prime }\right\vert }{\left\vert x-y\right\vert }\right) ^{\delta
}\left\vert x-y\right\vert ^{\alpha -n-1},\ \ \ \ \ \frac{\left\vert
x-x^{\prime }\right\vert }{\left\vert x-y\right\vert }\leq \frac{1}{2}, 
\notag
\end{eqnarray}%
and the last inequality also holds for the adjoint kernel in which $x$ and $%
y $ are interchanged. We note that a more general definition of kernel has
only order of smoothness $\delta >0$, rather than $1+\delta $, but the use
of the Monotonicity and Energy Lemmas in arguments below, which involve
first order Taylor approximations to the kernel functions $K^{\alpha }\left(
\cdot ,y\right) $, requires order of smoothness more than $1$ to handle
remainder terms.

\subsubsection{Defining the norm inequality}

We now turn to a precise definition of the weighted norm inequality%
\begin{equation}
\left\Vert T_{\sigma }^{\alpha }f\right\Vert _{L^{2}\left( \omega \right)
}\leq \mathfrak{N}_{T_{\sigma }^{\alpha }}\left\Vert f\right\Vert
_{L^{2}\left( \sigma \right) },\ \ \ \ \ f\in L^{2}\left( \sigma \right) .
\label{two weight'}
\end{equation}%
For this we introduce a family $\left\{ \eta _{\delta ,R}^{\alpha }\right\}
_{0<\delta <R<\infty }$ of nonnegative functions on $\left[ 0,\infty \right) 
$ so that the truncated kernels $K_{\delta ,R}^{\alpha }\left( x,y\right)
=\eta _{\delta ,R}^{\alpha }\left( \left\vert x-y\right\vert \right)
K^{\alpha }\left( x,y\right) $ are bounded with compact support for fixed $x$
or $y$. Then the truncated operators 
\begin{equation*}
T_{\sigma ,\delta ,R}^{\alpha }f\left( x\right) \equiv \int_{\mathbb{R}%
^{n}}K_{\delta ,R}^{\alpha }\left( x,y\right) f\left( y\right) d\sigma
\left( y\right) ,\ \ \ \ \ x\in \mathbb{R}^{n},
\end{equation*}%
are pointwise well-defined, and we will refer to the pair $\left( K^{\alpha
},\left\{ \eta _{\delta ,R}^{\alpha }\right\} _{0<\delta <R<\infty }\right) $
as an $\alpha $-fractional singular integral operator, which we typically
denote by $T^{\alpha }$, suppressing the dependence on the truncations.

\begin{definition}
We say that an $\alpha $-fractional singular integral operator $T^{\alpha
}=\left( K^{\alpha },\left\{ \eta _{\delta ,R}^{\alpha }\right\} _{0<\delta
<R<\infty }\right) $ satisfies the norm inequality (\ref{two weight'})
provided%
\begin{equation*}
\left\Vert T_{\sigma ,\delta ,R}^{\alpha }f\right\Vert _{L^{2}\left( \omega
\right) }\leq \mathfrak{N}_{T_{\sigma }^{\alpha }}\left\Vert f\right\Vert
_{L^{2}\left( \sigma \right) },\ \ \ \ \ f\in L^{2}\left( \sigma \right)
,0<\delta <R<\infty .
\end{equation*}
\end{definition}

It turns out that, in the presence of Muckenhoupt conditions, the norm
inequality (\ref{two weight'}) is essentially independent of the choice of
truncations used, and we now explain this in some detail. A \emph{smooth
truncation} of $T^{\alpha }$ has kernel $\eta _{\delta ,R}\left( \left\vert
x-y\right\vert \right) K^{\alpha }\left( x,y\right) $ for a smooth function $%
\eta _{\delta ,R}$ compactly supported in $\left( \delta ,R\right) $, $%
0<\delta <R<\infty $, and satisfying standard CZ estimates. A typical
example of an $\alpha $-fractional transform is the $\alpha $-fractional 
\emph{Riesz} vector of operators%
\begin{equation*}
\mathbf{R}^{\alpha ,n}=\left\{ R_{\ell }^{\alpha ,n}:1\leq \ell \leq
n\right\} ,\ \ \ \ \ 0\leq \alpha <n.
\end{equation*}%
The Riesz transforms $R_{\ell }^{n,\alpha }$ are convolution fractional
singular integrals $R_{\ell }^{n,\alpha }f\equiv K_{\ell }^{n,\alpha }\ast f$
with odd kernel defined by%
\begin{equation*}
K_{\ell }^{\alpha ,n}\left( w\right) \equiv \frac{w^{\ell }}{\left\vert
w\right\vert ^{n+1-\alpha }}\equiv \frac{\Omega _{\ell }\left( w\right) }{%
\left\vert w\right\vert ^{n-\alpha }},\ \ \ \ \ w=\left(
w^{1},...,w^{n}\right) .
\end{equation*}

However, in dealing with energy considerations, and in particular in the
Monotonicity Lemma below where first order Taylor approximations are made on
the truncated kernels, it is necessary to use the \emph{tangent line
truncation} of\emph{\ }the Riesz transform $R_{\ell }^{\alpha ,n}$ whose
kernel is defined to be $\Omega _{\ell }\left( w\right) \psi _{\delta
,R}^{\alpha }\left( \left\vert w\right\vert \right) $ where $\psi _{\delta
,R}^{\alpha }$ is continuously differentiable on an interval $\left(
0,S\right) $ with $0<\delta <R<S$, and where $\psi _{\delta ,R}^{\alpha
}\left( r\right) =r^{\alpha -n}$ if $\delta \leq r\leq R$, and has constant
derivative on both $\left( 0,\delta \right) $ and $\left( R,S\right) $ where 
$\psi _{\delta ,R}^{\alpha }\left( S\right) =0$. Here $S$ is uniquely
determined by $R$ and $\alpha $. Finally we set $\psi _{\delta ,R}^{\alpha
}\left( S\right) =0$ as well, so that the kernel vanishes on the diagonal
and common point masses do not `see' each other. Note also that the tangent
line extension of a $C^{1,\delta }$ function on the line is again $%
C^{1,\delta }$ with no increase in the $C^{1,\delta }$ norm.

It was shown in the one dimensional case with no common point masses in \cite%
{LaSaShUr3}, that boundedness of the Hilbert transform $H$ with one set of
appropriate truncations together with the $A_{2}^{\alpha }$ condition
without holes, is equivalent to boundedness of $H$ with any other set of
appropriate truncations, and this was extended to $\mathbf{R}^{\alpha ,n}$
and more general operators in higher dimensions, permitting common point
masses as well. Thus we are free to use the tangent line truncations
throughout the proofs of our results.

\subsection{Quasicube testing conditions}

The following `dual' quasicube testing conditions are necessary for the
boundedness of $T^{\alpha }$ from $L^{2}\left( \sigma \right) $ to $%
L^{2}\left( \omega \right) $,%
\begin{eqnarray*}
\mathfrak{T}_{T^{\alpha }}^{2} &\equiv &\sup_{Q\in \Omega \mathcal{P}^{n}}%
\frac{1}{\left\vert Q\right\vert _{\sigma }}\int_{Q}\left\vert T^{\alpha
}\left( \mathbf{1}_{Q}\sigma \right) \right\vert ^{2}\omega <\infty , \\
\left( \mathfrak{T}_{T^{\alpha }}^{\ast }\right) ^{2} &\equiv &\sup_{Q\in
\Omega \mathcal{P}^{n}}\frac{1}{\left\vert Q\right\vert _{\omega }}%
\int_{Q}\left\vert \left( T^{\alpha }\right) ^{\ast }\left( \mathbf{1}%
_{Q}\omega \right) \right\vert ^{2}\sigma <\infty ,
\end{eqnarray*}%
and where we interpret the right sides as holding uniformly over all tangent
line truncations of $T^{\alpha }$. Equally necessary are the following
`full' testing conditions where the integrations are taken over the entire
space $\mathbb{R}^{n}$:%
\begin{eqnarray*}
\mathfrak{FT}_{T^{\alpha }}^{2} &\equiv &\sup_{Q\in \Omega \mathcal{P}^{n}}%
\frac{1}{\left\vert Q\right\vert _{\sigma }}\int_{\mathbb{R}^{n}}\left\vert
T^{\alpha }\left( \mathbf{1}_{Q}\sigma \right) \right\vert ^{2}\omega
<\infty , \\
\left( \mathfrak{FT}_{T^{\alpha }}^{\ast }\right) ^{2} &\equiv &\sup_{Q\in
\Omega \mathcal{P}^{n}}\frac{1}{\left\vert Q\right\vert _{\omega }}\int_{%
\mathbb{R}^{n}}\left\vert \left( T^{\alpha }\right) ^{\ast }\left( \mathbf{1}%
_{Q}\omega \right) \right\vert ^{2}\sigma <\infty ,
\end{eqnarray*}

\subsection{Quasiweak boundedness and indicator/touching property}

The quasiweak boundedness property for $T^{\alpha }$ with constant $C$ is
given by 
\begin{eqnarray}
&&\left\vert \int_{Q}T^{\alpha }\left( 1_{Q^{\prime }}\sigma \right) d\omega
\right\vert \leq \mathcal{WBP}_{T^{\alpha }}\sqrt{\left\vert Q\right\vert
_{\omega }\left\vert Q^{\prime }\right\vert _{\sigma }},  \label{def WBP} \\
&&\ \ \ \ \ \text{for all quasicubes }Q,Q^{\prime }\text{ with }\frac{1}{C}%
\leq \frac{\ell \left( Q\right) }{\ell \left( Q^{\prime }\right) }\leq C, 
\notag \\
&&\ \ \ \ \ \text{and either }Q\subset 3Q^{\prime }\setminus Q^{\prime }%
\text{ or }Q^{\prime }\subset 3Q\setminus Q,  \notag
\end{eqnarray}%
and where we interpret the left side above as holding uniformly over all
tangent line trucations of $T^{\alpha }$. This condition is used in our $T1$
theorem with an energy side condition in \cite{SaShUr9}, but will be removed
in our $T1$ theorem with an energy side condition obtained here as a
corollary of the Good-$\lambda $ Lemma.

We say that two quasicubes $Q$ and $Q^{\prime }$ in $\Omega \mathcal{P}^{n}$
are \emph{touching quasicubes} if the intersection of their closures is
nonempty and contained in the boundary of the larger quasicube. Finally, let 
$\mathfrak{I}_{T^{\alpha }}=\mathfrak{I}_{T^{\alpha }}\left( \sigma ,\omega
\right) $ be the best constant in the \emph{indicator/touching }inequality%
\begin{eqnarray}
&&\left\vert \mathcal{T}^{\alpha }\left( \mathbf{1}_{Q},\mathbf{1}%
_{Q^{\prime }}\right) \right\vert \leq \mathfrak{I}_{T^{\alpha }}\left(
\sigma ,\omega \right) \left\Vert \mathbf{1}_{Q}\right\Vert _{L^{2}\left(
\sigma \right) }\left\Vert \mathbf{1}_{Q^{\prime }}\right\Vert _{L^{2}\left(
\omega \right) }\ ,  \label{Ind/touch} \\
&&\ \ \ \ \ \text{for all touching quasicubes }Q,Q^{\prime }\in \mathcal{P}%
^{n},  \notag \\
&&\ \ \ \ \ \text{with }\frac{1}{C}\leq \frac{\ell \left( Q\right) }{\ell
\left( Q^{\prime }\right) }\leq C,  \notag \\
&&\ \ \ \ \ \text{and either }Q\subset 3Q^{\prime }\setminus Q^{\prime }%
\text{ or }Q^{\prime }\subset 3Q\setminus Q.  \notag
\end{eqnarray}

\subsection{Poisson integrals and $\mathcal{A}_{2}^{\protect\alpha }$}

Let $\mu $ be a locally finite positive Borel measure on $\mathbb{R}^{n}$,
and suppose $Q$ is an $\Omega $-quasicube in $\mathbb{R}^{n}$. Recall that $%
\left\vert Q\right\vert ^{\frac{1}{n}}\approx \ell \left( Q\right) $ for a
quasicube $Q$. The two $\alpha $-fractional Poisson integrals of $\mu $ on a
quasicube $Q$ are given by:%
\begin{eqnarray*}
\mathrm{P}^{\alpha }\left( Q,\mu \right) &\equiv &\int_{\mathbb{R}^{n}}\frac{%
\left\vert Q\right\vert ^{\frac{1}{n}}}{\left( \left\vert Q\right\vert ^{%
\frac{1}{n}}+\left\vert x-x_{Q}\right\vert \right) ^{n+1-\alpha }}d\mu
\left( x\right) , \\
\mathcal{P}^{\alpha }\left( Q,\mu \right) &\equiv &\int_{\mathbb{R}%
^{n}}\left( \frac{\left\vert Q\right\vert ^{\frac{1}{n}}}{\left( \left\vert
Q\right\vert ^{\frac{1}{n}}+\left\vert x-x_{Q}\right\vert \right) ^{2}}%
\right) ^{n-\alpha }d\mu \left( x\right) ,
\end{eqnarray*}%
where we emphasize that $\left\vert x-x_{Q}\right\vert $ denotes Euclidean
distance between $x$ and $x_{Q}$ and $\left\vert Q\right\vert $ denotes the
Lebesgue measure of the quasicube $Q$. We refer to $\mathrm{P}^{\alpha }$ as
the \emph{standard} Poisson integral and to $\mathcal{P}^{\alpha }$ as the 
\emph{reproducing} Poisson integral.

We say that the pair $K,K^{\prime }$ in $\mathcal{P}^{n}$ are \emph{%
neighbours} if $K$ and $K^{\prime }$ live in a common dyadic grid and both $%
K\subset 3K^{\prime }\setminus K^{\prime }$ and $K^{\prime }\subset
3K\setminus K$, and we denote by $\mathcal{N}^{n}$ the set of pairs $\left(
K,K^{\prime }\right) $ in $\mathcal{P}^{n}\times \mathcal{P}^{n}$ that are
neighbours. Let 
\begin{equation*}
\Omega \mathcal{N}^{n}=\left\{ \left( \Omega K,\Omega K^{\prime }\right)
:\left( K,K^{\prime }\right) \in \mathcal{N}^{n}\right\}
\end{equation*}%
be the corresponding collection of neighbour pairs of quasicubes. Let $%
\sigma $ and $\omega $ be locally finite positive Borel measures on $\mathbb{%
R}^{n}$, and suppose $0\leq \alpha <n$. Then we define the classical \emph{%
offset }$A_{2}^{\alpha }$\emph{\ constants} by 
\begin{equation}
A_{2}^{\alpha }\left( \sigma ,\omega \right) \equiv \sup_{\left( Q,Q^{\prime
}\right) \in \Omega \mathcal{N}^{n}}\frac{\left\vert Q\right\vert _{\sigma }%
}{\left\vert Q\right\vert ^{1-\frac{\alpha }{n}}}\frac{\left\vert Q^{\prime
}\right\vert _{\omega }}{\left\vert Q\right\vert ^{1-\frac{\alpha }{n}}}.
\label{offset A2}
\end{equation}%
Since the cubes in $\mathcal{P}^{n}$ are products of half open, half closed
intervals $\left[ a,b\right) $, the neighbouring quasicubes $\left(
Q,Q^{\prime }\right) \in \Omega \mathcal{N}^{n}$ are disjoint, and any
common point masses of $\sigma $ and $\omega $ do not simultaneously appear
in each factor.

We now define the \emph{one-tailed} $\mathcal{A}_{2}^{\alpha }$ constant
using $\mathcal{P}^{\alpha }$. The energy constants $\mathcal{E}_{\alpha }^{%
\limfunc{strong}}$ introduced below will use the standard Poisson integral $%
\mathrm{P}^{\alpha }$.

\begin{definition}
The one-tailed constants $\mathcal{A}_{2}^{\alpha }$ and $\mathcal{A}%
_{2}^{\alpha ,\ast }$ for the weight pair $\left( \sigma ,\omega \right) $
are given by%
\begin{eqnarray*}
\mathcal{A}_{2}^{\alpha } &\equiv &\sup_{Q\in \Omega \mathcal{P}^{n}}%
\mathcal{P}^{\alpha }\left( Q,\mathbf{1}_{Q^{c}}\sigma \right) \frac{%
\left\vert Q\right\vert _{\omega }}{\left\vert Q\right\vert ^{1-\frac{\alpha 
}{n}}}<\infty , \\
\mathcal{A}_{2}^{\alpha ,\ast } &\equiv &\sup_{Q\in \Omega \mathcal{P}^{n}}%
\mathcal{P}^{\alpha }\left( Q,\mathbf{1}_{Q^{c}}\omega \right) \frac{%
\left\vert Q\right\vert _{\sigma }}{\left\vert Q\right\vert ^{1-\frac{\alpha 
}{n}}}<\infty .
\end{eqnarray*}
\end{definition}

Note that these definitions are the analogues of the corresponding
conditions with `holes' introduced by Hyt\"{o}nen \cite{Hyt2} in dimension $%
n=1$ - the supports of the measures $\mathbf{1}_{Q^{c}}\sigma $ and $\mathbf{%
1}_{Q}\omega $ in the definition of $\mathcal{A}_{2}^{\alpha }$ are
disjoint, and so the common point masses of $\sigma $ and $\omega $ do not
appear simultaneously in each factor. Note also that, unlike in \cite%
{SaShUr5}, where common point masses were not permitted, we can no longer
assert the equivalence of $\mathcal{A}_{2}^{\alpha }$ with holes taken over 
\emph{quasicubes} with $\mathcal{A}_{2}^{\alpha }$ with holes taken over 
\emph{cubes}.

\subsubsection{Punctured $A_{2}^{\protect\alpha }$ conditions}

The \emph{classical} $A_{2}^{\alpha }$ characteristic $\sup_{Q\in \Omega 
\mathcal{Q}^{n}}\frac{\left\vert Q\right\vert _{\omega }}{\left\vert
Q\right\vert ^{1-\frac{\alpha }{n}}}\frac{\left\vert Q\right\vert _{\sigma }%
}{\left\vert Q\right\vert ^{1-\frac{\alpha }{n}}}$ fails to be finite when
the measures $\sigma $ and $\omega $ have a common point mass - simply let $%
Q $ in the $\sup $ above shrink to a common mass point. But there is a
substitute that is quite similar in character that is motivated by the fact
that for large quasicubes $Q$, the $\sup $ above is problematic only if just 
\emph{one} of the measures is \emph{mostly} a point mass when restricted to $%
Q$.

Given an at most countable set $\mathfrak{P}=\left\{ p_{k}\right\}
_{k=1}^{\infty }$ in $\mathbb{R}^{n}$, a quasicube $Q\in \Omega \mathcal{P}%
^{n}$, and a locally finite positive Borel measure $\mu $, define as in \cite%
{SaShUr9}, 
\begin{equation*}
\mu \left( Q,\mathfrak{P}\right) \equiv \left\vert Q\right\vert _{\mu }-\sup
\left\{ \mu \left( p_{k}\right) :p_{k}\in Q\cap \mathfrak{P}\right\} ,
\end{equation*}%
where the supremum is actually achieved since $\sum_{p_{k}\in Q\cap 
\mathfrak{P}}\mu \left( p_{k}\right) <\infty $ as $\mu $ is locally finite.
The quantity $\mu \left( Q,\mathfrak{P}\right) $ is simply the $\widetilde{%
\mu }$ measure of $Q$ where $\widetilde{\mu }$ is the measure $\mu $ with
its largest point mass from $\mathfrak{P}$ in $Q$ removed. Given a locally
finite measure pair $\left( \sigma ,\omega \right) $, let $\mathfrak{P}%
_{\left( \sigma ,\omega \right) }=\left\{ p_{k}\right\} _{k=1}^{\infty }$ be
the at most countable set of common point masses of $\sigma $ and $\omega $.
Then the weighted norm inequality (\ref{two weight'}) typically implies
finiteness of the following \emph{punctured} Muckenhoupt conditions (see 
\cite{SaShUr9}):%
\begin{eqnarray*}
A_{2}^{\alpha ,\limfunc{punct}}\left( \sigma ,\omega \right) &\equiv
&\sup_{Q\in \Omega \mathcal{P}^{n}}\frac{\omega \left( Q,\mathfrak{P}%
_{\left( \sigma ,\omega \right) }\right) }{\left\vert Q\right\vert ^{1-\frac{%
\alpha }{n}}}\frac{\left\vert Q\right\vert _{\sigma }}{\left\vert
Q\right\vert ^{1-\frac{\alpha }{n}}}, \\
A_{2}^{\alpha ,\ast ,\limfunc{punct}}\left( \sigma ,\omega \right) &\equiv
&\sup_{Q\in \Omega \mathcal{P}^{n}}\frac{\left\vert Q\right\vert _{\omega }}{%
\left\vert Q\right\vert ^{1-\frac{\alpha }{n}}}\frac{\sigma \left( Q,%
\mathfrak{P}_{\left( \sigma ,\omega \right) }\right) }{\left\vert
Q\right\vert ^{1-\frac{\alpha }{n}}}.
\end{eqnarray*}

Now we turn to the definition of a quasiHaar basis of $L^{2}\left( \mu
\right) $.

\subsection{A weighted quasiHaar basis}

We will use a construction of a quasiHaar basis in $\mathbb{R}^{n}$ that is
adapted to a measure $\mu $ (c.f. \cite{NTV2} for the nonquasi case). Given
a dyadic quasicube $Q\in \Omega \mathcal{D}$, where $\mathcal{D}$ is a
dyadic grid of cubes from $\mathcal{P}^{n}$, let $\bigtriangleup _{Q}^{\mu }$
denote orthogonal projection onto the finite dimensional subspace $%
L_{Q}^{2}\left( \mu \right) $ of $L^{2}\left( \mu \right) $ that consists of
linear combinations of the indicators of\ the children $\mathfrak{C}\left(
Q\right) $ of $Q$ that have $\mu $-mean zero over $Q$:%
\begin{equation*}
L_{Q}^{2}\left( \mu \right) \equiv \left\{ f=\dsum\limits_{Q^{\prime }\in 
\mathfrak{C}\left( Q\right) }a_{Q^{\prime }}\mathbf{1}_{Q^{\prime
}}:a_{Q^{\prime }}\in \mathbb{R},\int_{Q}fd\mu =0\right\} .
\end{equation*}%
Then we have the important telescoping property for dyadic quasicubes $%
Q_{1}\subset Q_{2}$ that arises from the martingale differences associated
with the projections $\bigtriangleup _{Q}^{\mu }$:%
\begin{equation}
\mathbf{1}_{Q_{0}}\left( x\right) \left( \dsum\limits_{Q\in \left[
Q_{1},Q_{2}\right] }\bigtriangleup _{Q}^{\mu }f\left( x\right) \right) =%
\mathbf{1}_{Q_{0}}\left( x\right) \left( \mathbb{E}_{Q_{0}}^{\mu }f-\mathbb{E%
}_{Q_{2}}^{\mu }f\right) ,\ \ \ \ \ Q_{0}\in \mathfrak{C}\left( Q_{1}\right)
,\ f\in L^{2}\left( \mu \right) .  \label{telescope}
\end{equation}%
We will at times find it convenient to use a fixed orthonormal basis $%
\left\{ h_{Q}^{\mu ,a}\right\} _{a\in \Gamma _{n}}$ of $L_{Q}^{2}\left( \mu
\right) $ where $\Gamma _{n}\equiv \left\{ 0,1\right\} ^{n}\setminus \left\{ 
\mathbf{1}\right\} $ is a convenient index set with $\mathbf{1}=\left(
1,1,...,1\right) $. Then $\left\{ h_{Q}^{\mu ,a}\right\} _{a\in \Gamma _{n}%
\text{ and }Q\in \Omega \mathcal{D}}$ is an orthonormal basis for $%
L^{2}\left( \mu \right) $, with the understanding that we add the constant
function $\mathbf{1}$ if $\mu $ is a finite measure. In particular we have%
\begin{equation*}
\left\Vert f\right\Vert _{L^{2}\left( \mu \right) }^{2}=\sum_{Q\in \Omega 
\mathcal{D}}\left\Vert \bigtriangleup _{Q}^{\mu }f\right\Vert _{L^{2}\left(
\mu \right) }^{2}=\sum_{Q\in \Omega \mathcal{D}}\sum_{a\in \Gamma
_{n}}\left\vert \widehat{f}\left( Q\right) \right\vert ^{2},\ \ \ \ \
\left\vert \widehat{f}\left( Q\right) \right\vert ^{2}\equiv \sum_{a\in
\Gamma _{n}}\left\vert \left\langle f,h_{Q}^{\mu ,a}\right\rangle _{\mu
}\right\vert ^{2},
\end{equation*}%
where the measure is suppressed in the notation $\widehat{f}$. Indeed, this
follows from (\ref{telescope}) and Lebesgue's differentiation theorem for
quasicubes. We also record the following useful estimate. If $I^{\prime }$
is any of the $2^{n}$ $\Omega \mathcal{D}$-children of $I$, and $a\in \Gamma
_{n}$, then 
\begin{equation}
\left\vert \mathbb{E}_{I^{\prime }}^{\mu }h_{I}^{\mu ,a}\right\vert \leq 
\sqrt{\mathbb{E}_{I^{\prime }}^{\mu }\left( h_{I}^{\mu ,a}\right) ^{2}}\leq 
\frac{1}{\sqrt{\left\vert I^{\prime }\right\vert _{\mu }}}.
\label{useful Haar}
\end{equation}

\subsection{The strong quasienergy conditions}

Given a dyadic quasicube $K\in \Omega \mathcal{D}$ and a positive measure $%
\mu $ we define the quasiHaar projection $\mathsf{P}_{K}^{\mu }\equiv
\sum_{_{J\in \Omega \mathcal{D}:\ J\subset K}}\bigtriangleup _{J}^{\mu }$ on 
$K$ by 
\begin{equation*}
\mathsf{P}_{K}^{\mu }f=\sum_{_{J\in \Omega \mathcal{D}:\ J\subset
K}}\dsum\limits_{a\in \Gamma _{n}}\left\langle f,h_{J}^{\mu ,a}\right\rangle
_{\mu }h_{J}^{\mu ,a}\text{ so that }\left\Vert \mathsf{P}_{K}^{\mu
}f\right\Vert _{L^{2}\left( \mu \right) }^{2}=\sum_{_{J\in \Omega \mathcal{D}%
:\ J\subset K}}\dsum\limits_{a\in \Gamma _{n}}\left\vert \left\langle
f,h_{J}^{\mu ,a}\right\rangle _{\mu }\right\vert ^{2},
\end{equation*}%
and where a quasiHaar basis $\left\{ h_{J}^{\mu ,a}\right\} _{a\in \Gamma
_{n}\text{ and }J\in \Omega \mathcal{D}\Omega }$ adapted to the measure $\mu 
$ was defined in the subsection on a weighted quasiHaar basis above.

Now we define various notions for quasicubes which are inherited from the
same notions for cubes. The main objective here is to use the familiar
notation that one uses for cubes, but now extended to $\Omega $-quasicubes.
We have already introduced the notions of quasigrids $\Omega \mathcal{D}$,
and center, sidelength and dyadic associated to quasicubes $Q\in \Omega 
\mathcal{D}$, as well as quasiHaar functions, and we will continue to extend
to quasicubes the additional familiar notions related to cubes as we come
across them. We begin with the notion of \emph{deeply embedded}. Fix a
quasigrid $\Omega \mathcal{D}$. We say that a dyadic quasicube $J$ is $%
\left( \mathbf{r},\varepsilon \right) $-\emph{deeply embedded} in a (not
necessarily dyadic) quasicube $K$, which we write as $J\Subset _{\mathbf{r}%
,\varepsilon }K$, when $J\subset K$ and both 
\begin{eqnarray}
\ell \left( J\right) &\leq &2^{-\mathbf{r}}\ell \left( K\right) ,
\label{def deep embed} \\
\limfunc{qdist}\left( J,\partial K\right) &\geq &\frac{1}{2}\ell \left(
J\right) ^{\varepsilon }\ell \left( K\right) ^{1-\varepsilon },  \notag
\end{eqnarray}%
where we define the quasidistance $\limfunc{qdist}\left( E,F\right) $
between two sets $E$ and $F$ to be the Euclidean distance $\limfunc{dist}%
\left( \Omega ^{-1}E,\Omega ^{-1}F\right) $ between the preimages $\Omega
^{-1}E$ and $\Omega ^{-1}F$ of $E$ and $F$ under the map $\Omega $, and
where we recall that $\ell \left( J\right) \approx \left\vert J\right\vert ^{%
\frac{1}{n}}$. For the most part we will consider $J\Subset _{\mathbf{r}%
,\varepsilon }K$ when $J$ and $K$ belong to a common quasigrid $\Omega 
\mathcal{D}$, but an exception is made when defining the strong energy
constants below.

Recall that in dimension $n=1$, and for $\alpha =0$, the energy condition
constant was defined by%
\begin{equation*}
\mathcal{E}^{2}\equiv \sup_{I=\dot{\cup}I_{r}}\frac{1}{\left\vert
I\right\vert _{\sigma }}\sum_{r=1}^{\infty }\left( \frac{\mathrm{P}^{\alpha
}\left( I_{r},\mathbf{1}_{I}\sigma \right) }{\left\vert I_{r}\right\vert }%
\right) ^{2}\left\Vert \mathsf{P}_{I_{r}}^{\omega }\mathbf{x}\right\Vert
_{L^{2}\left( \omega \right) }^{2}\ ,
\end{equation*}%
where $I$, $I_{r}$ and $J$ are intervals in the real line. The extension to
higher dimensions we use here is that of `strong quasienergy condition'
defined in \cite{SaShUr9} and recalled below.

We define a quasicube $K$ (not necessarily in $\Omega \mathcal{D}$) to be an 
\emph{alternate} $\Omega \mathcal{D}$-quasicube if it is a union of $2^{n}$ $%
\Omega \mathcal{D}$-quasicubes $K^{\prime }$ with side length $\ell \left(
K^{\prime }\right) =\frac{1}{2}\ell \left( K\right) $ (such quasicubes were
called shifted in \cite{SaShUr5}, but that terminology conflicts with the
more familiar notion of shifted quasigrid). Thus for any $\Omega \mathcal{D}$%
-quasicube $L$ there are exactly $2^{n}$ alternate $\Omega \mathcal{D}$%
-quasicubes of twice the side length that contain $L$, and one of them is of
course the $\Omega \mathcal{D}$-parent of $L$. We denote the collection of
alternate $\Omega \mathcal{D}$-quasicubes by $\mathcal{A}\Omega \mathcal{D}$.

The extension of the energy conditions to higher dimensions in \cite{SaShUr5}
used the collection 
\begin{equation*}
\mathcal{M}_{\mathbf{r},\varepsilon -\limfunc{deep}}\left( K\right) \equiv
\left\{ \text{maximal }J\Subset _{\mathbf{r},\varepsilon }K\right\}
\end{equation*}%
of \emph{maximal} $\left( \mathbf{r},\varepsilon \right) $-deeply embedded
dyadic subquasicubes of a quasicube $K$ (a subquasicube $J$ of $K$ is a 
\emph{dyadic} subquasicube of $K$ if $J\in \Omega \mathcal{D}$ when $\Omega 
\mathcal{D}$ is a dyadic quasigrid containing $K$). This collection of
dyadic subquasicubes of $K$ is of course a pairwise disjoint decomposition
of $K$. We also defined there a refinement and extension of the collection $%
\mathcal{M}_{\left( \mathbf{r},\varepsilon \right) -\limfunc{deep}}\left(
K\right) $ for certain $K$ and each $\ell \geq 1$. For an alternate
quasicube $K\in \mathcal{A}\Omega \mathcal{D}$, define $\mathcal{M}_{\left( 
\mathbf{r},\varepsilon \right) -\limfunc{deep},\Omega \mathcal{D}}\left(
K\right) $ to consist of the \emph{maximal} $\mathbf{r}$-deeply embedded $%
\Omega \mathcal{D}$-dyadic subquasicubes $J$ of $K$. (In the special case
that $K$ itself belongs to $\Omega \mathcal{D}$, then $\mathcal{M}_{\left( 
\mathbf{r},\varepsilon \right) -\limfunc{deep},\Omega \mathcal{D}}\left(
K\right) =\mathcal{M}_{\left( \mathbf{r},\varepsilon \right) -\limfunc{deep}%
}\left( K\right) $.) Then in \cite{SaShUr5} for $\ell \geq 1$ we defined the
refinement%
\begin{eqnarray*}
\mathcal{M}_{\left( \mathbf{r},\varepsilon \right) -\limfunc{deep},\Omega 
\mathcal{D}}^{\ell }\left( K\right) &\equiv &\left\{ J\in \mathcal{M}%
_{\left( \mathbf{r},\varepsilon \right) -\limfunc{deep},\Omega \mathcal{D}%
}\left( \pi ^{\ell }K^{\prime }\right) \text{ for some }K^{\prime }\in 
\mathfrak{C}_{\Omega \mathcal{D}}\left( K\right) :\right. \\
&&\ \ \ \ \ \ \ \ \ \ \ \ \ \ \ \ \ \ \ \ \ \ \ \ \ \ \ \ \ \ \left.
J\subset L\text{ for some }L\in \mathcal{M}_{\left( \mathbf{r},\varepsilon
\right) -\limfunc{deep}}\left( K\right) \right\} ,
\end{eqnarray*}%
where $\mathfrak{C}_{\Omega \mathcal{D}}\left( K\right) $ is the obvious
extension to alternate quasicubes of the set of $\Omega \mathcal{D}$-dyadic
children. Thus $\mathcal{M}_{\left( \mathbf{r},\varepsilon \right) -\limfunc{%
deep},\Omega \mathcal{D}}^{\ell }\left( K\right) $ is the union, over all
quasichildren $K^{\prime }$ of $K$, of those quasicubes in $\mathcal{M}%
_{\left( \mathbf{r},\varepsilon \right) -\limfunc{deep}}\left( \pi ^{\ell
}K^{\prime }\right) $ that happen to be contained in some $L\in \mathcal{M}%
_{\left( \mathbf{r},\varepsilon \right) -\limfunc{deep},\Omega \mathcal{D}%
}\left( K\right) $. We then define the \emph{strong} quasienergy condition
as follows.

\begin{definition}
\label{def strong quasienergy}Let $0\leq \alpha <n$ and fix `goodness'
parameters $\left( \mathbf{r},\varepsilon \right) $. Suppose $\sigma $ and $%
\omega $ are locally finite positive Borel measures on $\mathbb{R}^{n}$.
Then the \emph{strong} quasienergy constant $\mathcal{E}_{\alpha }^{\limfunc{%
strong}}$ is defined by 
\begin{eqnarray*}
\left( \mathcal{E}_{\alpha }^{\limfunc{strong}}\right) ^{2} &\equiv &\sup_{I=%
\dot{\cup}I_{r}}\frac{1}{\left\vert I\right\vert _{\sigma }}%
\sum_{r=1}^{\infty }\sum_{J\in \mathcal{M}_{\mathbf{r},\varepsilon -\limfunc{%
deep}}\left( I_{r}\right) }\left( \frac{\mathrm{P}^{\alpha }\left( J,\mathbf{%
1}_{I}\sigma \right) }{\left\vert J\right\vert ^{\frac{1}{n}}}\right)
^{2}\left\Vert \mathsf{P}_{J}^{\omega }\mathbf{x}\right\Vert _{L^{2}\left(
\omega \right) }^{2} \\
&&+\sup_{\Omega \mathcal{D}}\sup_{I\in \mathcal{A}\Omega \mathcal{D}%
}\sup_{\ell \geq 0}\frac{1}{\left\vert I\right\vert _{\sigma }}\sum_{J\in 
\mathcal{M}_{\left( \mathbf{r},\varepsilon \right) -\limfunc{deep},\Omega 
\mathcal{D}}^{\ell }\left( I\right) }\left( \frac{\mathrm{P}^{\alpha }\left(
J,\mathbf{1}_{I}\sigma \right) }{\left\vert J\right\vert ^{\frac{1}{n}}}%
\right) ^{2}\left\Vert \mathsf{P}_{J}^{\omega }\mathbf{x}\right\Vert
_{L^{2}\left( \omega \right) }^{2}\ .
\end{eqnarray*}
\end{definition}

Similarly we have a dual version of $\mathcal{E}_{\alpha }^{\limfunc{strong}%
} $ denoted $\mathcal{E}_{\alpha }^{\limfunc{strong},\ast }$, and both
depend on $\mathbf{r}$ and $\varepsilon $ as well as on $n$ and $\alpha $.
An important point in this definition is that the quasicube $I$ in the
second line is permitted to lie \emph{outside} the quasigrid $\Omega 
\mathcal{D}$, but only as an alternate dyadic quasicube $I\in \mathcal{A}%
\Omega \mathcal{D} $. In the setting of quasicubes we continue to use the
linear function $\mathbf{x}$ in the final factor $\left\Vert \mathsf{P}%
_{J}^{\omega }\mathbf{x}\right\Vert _{L^{2}\left( \omega \right) }^{2}$ of
each line, and not the pushforward of $\mathbf{x}$ by $\Omega $. The reason
of course is that this condition is used to capture the first order
information in the Taylor expansion of a singular kernel.

\section{The Good-$\protect\lambda $ Lemma}

The basic new result of this paper is the following `Good-$\lambda $ Lemma'
whose utility will become evident when we pursue its corollaries below. Set 
\emph{fraktur} $A_{2}^{\alpha }$ to be the sum of the four $A_{2}^{\alpha }$
conditions:%
\begin{equation*}
\mathfrak{A}_{2}^{\alpha }=\mathcal{A}_{2}^{\alpha }+\mathcal{A}_{2}^{\alpha
,\ast }+A_{2}^{\alpha ,\limfunc{punct}}+A_{2}^{\alpha ,\ast ,\limfunc{punct}%
}.
\end{equation*}

\begin{lemma}[The Good-$\protect\lambda $ Lemma]
Suppose that $T^{\alpha }$ is a standard $\alpha $-fractional singular
integral in $\mathbb{R}^{n}$, and that $\sigma $ and $\omega $ are locally
finite positive Borel measures on $\mathbb{R}^{n}$. For every $\lambda \in
\left( 0,\frac{1}{2}\right) $, we have%
\begin{eqnarray}
&&\mathcal{WBP}_{T^{\alpha }}\left( \sigma ,\omega \right)  \label{g lambda}
\\
&\leq &C_{\alpha }\left( \frac{1}{\lambda }\sqrt{\mathfrak{A}_{2}^{\alpha
}\left( \sigma ,\omega \right) }+\left( \mathfrak{T}_{T^{\alpha }}+\mathfrak{%
T}_{T^{\alpha }}^{\ast }\right) \left( \sigma ,\omega \right) +\left( 
\mathcal{E}_{\alpha }^{\limfunc{strong}}+\mathcal{E}_{\alpha }^{\limfunc{%
strong},\ast }\right) \left( \sigma ,\omega \right) +\sqrt[4]{\lambda }%
\mathfrak{N}_{T^{\alpha }}\left( \sigma ,\omega \right) \right) .  \notag
\end{eqnarray}
\end{lemma}

Thus the effect of the Good-$\lambda $ Lemma is to `good-$\lambda $ replace'
the quasiweak boundedness property with just the usual testing conditions in
the presence of the side conditions of Muckenhoupt and energy on the weight
pair. However, in dimension $n=1$ a much stronger inequality can be proved
(see e.g. \cite{NTV3} and \cite{LaSaUr2}):%
\begin{equation*}
\mathcal{WBP}_{T^{\alpha }}\leq C_{\alpha }\left( \sqrt{\mathfrak{A}%
_{2}^{\alpha }}+\mathfrak{T}_{T^{\alpha }}+\mathfrak{T}_{T^{\alpha }}^{\ast
}\right) .
\end{equation*}

\subsection{Corollaries\label{Sub cor}}

Now we come to the corollaries of the Good-$\lambda $ Lemma. We first remove
the hypothesis of the quasiweak boundedness property from the conclusion of
part (1) of Theorem 1 in \cite{SaShUr9}.

\begin{remark}
\label{LacWic}In \cite{LaWi}, Lacey and Wick have removed the weak
boundedness property from their $T1$ theorem by using NTV surgery with two
independent grids, one for each function $f$ and $g$ in $\left\langle
T_{\sigma }^{\alpha }f,g\right\rangle $, in the course of their argument.
The use of independent grids for each of $f$ and $g$ greatly simplifies the
NTV surgery, but does not accommodate our control of functional energy by
Muckenhoupt and energy conditions.
\end{remark}

\begin{theorem}
\label{T1 theorem'}Suppose $0\leq \alpha <n$, that $T^{\alpha }$ is a
standard $\alpha $-fractional singular integral operator on $\mathbb{R}^{n}$%
, and that $\omega $ and $\sigma $ are locally finite positive Borel
measures on $\mathbb{R}^{n}$. Set $T_{\sigma }^{\alpha }f=T^{\alpha }\left(
f\sigma \right) $ for any smooth truncation of $T_{\sigma }^{\alpha }$. Let $%
\Omega :\mathbb{R}^{n}\rightarrow \mathbb{R}^{n}$ be a globally biLipschitz
map. Then the operator $T_{\sigma }^{\alpha }$ is bounded from $L^{2}\left(
\sigma \right) $ to $L^{2}\left( \omega \right) $, i.e. 
\begin{equation*}
\left\Vert T_{\sigma }^{\alpha }f\right\Vert _{L^{2}\left( \omega \right)
}\leq \mathfrak{N}_{T_{\sigma }^{\alpha }}\left\Vert f\right\Vert
_{L^{2}\left( \sigma \right) },
\end{equation*}%
uniformly in smooth truncations of $T^{\alpha }$, and moreover%
\begin{equation*}
\mathfrak{N}_{T_{\sigma }^{\alpha }}\leq C_{\alpha }\left( \sqrt{\mathfrak{A}%
_{2}^{\alpha }}+\mathfrak{T}_{T^{\alpha }}+\mathfrak{T}_{T^{\alpha }}^{\ast
}+\mathcal{E}_{\alpha }^{\limfunc{strong}}+\mathcal{E}_{\alpha }^{\limfunc{%
strong},\ast }\right) ,
\end{equation*}%
provided that the two dual $\mathcal{A}_{2}^{\alpha }$ conditions and the
two dual punctured Muckenhoupt conditions all hold, and the two dual
quasitesting conditions for $T^{\alpha }$ hold, and provided that the two
dual strong quasienergy conditions hold uniformly over all dyadic quasigrids 
$\Omega \mathcal{D}\subset \Omega \mathcal{P}^{n}$, i.e. $\mathcal{E}%
_{\alpha }^{\limfunc{strong}}+\mathcal{E}_{\alpha }^{\limfunc{strong},\ast
}<\infty $, and where the goodness parameters $\mathbf{r}$ and $\varepsilon $
implicit in the definition of the collections $\mathcal{M}_{\left( \mathbf{r}%
,\varepsilon \right) -\limfunc{deep}}\left( K\right) $ and $\mathcal{M}%
_{\left( \mathbf{r},\varepsilon \right) -\limfunc{deep},\Omega \mathcal{D}%
}^{\ell }\left( K\right) $ appearing in the strong energy conditions, are
fixed sufficiently large and small respectively depending only on $n$ and $%
\alpha $.
\end{theorem}

\begin{proof}
Let $T_{\delta ,R}^{\alpha }$ be a tangent line approximation to $T^{\alpha
} $ as introduced above. Then $\mathfrak{N}_{T_{\delta ,R}^{\alpha }}<\infty 
$, indeed $\mathfrak{N}_{T_{\delta ,R}^{\alpha }}\leq C_{n,\alpha ,\delta ,R}%
\sqrt{\mathfrak{A}_{2}^{\alpha }}$ by an easy argument, and by part (1) of
Theorem 1 in \cite{SaShUr9} applied to the $\alpha $-fractional singular
integral $T_{\delta ,R}^{\alpha }$ we have%
\begin{equation*}
\mathfrak{N}_{T_{\delta ,R}^{\alpha }}\leq C_{\alpha }\left( \sqrt{\mathfrak{%
A}_{2}^{\alpha }}+\mathfrak{T}_{T_{\delta ,R}^{\alpha }}+\mathfrak{T}%
_{T_{\delta ,R}^{\alpha }}^{\ast }+\mathcal{E}_{\alpha }^{\limfunc{strong}}+%
\mathcal{E}_{\alpha }^{\limfunc{strong},\ast }+\mathcal{WBP}_{T_{\delta
,R}^{\alpha }}\right) ,
\end{equation*}%
with $C_{\alpha }$ independent of $\delta $ and $R$. We obtain from the Good-%
$\lambda $ Lemma applied to $T_{\sigma ,\delta ,R}^{\alpha }$ in place of $%
T^{\alpha }$,%
\begin{equation*}
\mathcal{WBP}_{T_{\delta ,R}^{\alpha }}\leq C_{\alpha }\left( \frac{1}{%
\lambda }\sqrt{\mathfrak{A}_{2}^{\alpha }}+\mathfrak{T}_{T_{\delta
,R}^{\alpha }}+\mathfrak{T}_{T_{\delta ,R}^{\alpha }}^{\ast }+\mathcal{E}%
_{\alpha }^{\limfunc{strong}}+\mathcal{E}_{\alpha }^{\limfunc{strong},\ast }+%
\sqrt[4]{\lambda }\mathfrak{N}_{T_{\delta ,R}^{\alpha }}\right) ,
\end{equation*}%
and then combining inequalities gives%
\begin{equation*}
\mathfrak{N}_{T_{\delta ,R}^{\alpha }}\leq C_{\alpha }^{\prime }\left( \frac{%
1}{\lambda }\sqrt{\mathfrak{A}_{2}^{\alpha }}+\mathfrak{T}_{T_{\delta
,R}^{\alpha }}+\mathfrak{T}_{T_{\delta ,R}^{\alpha }}^{\ast }+\mathcal{E}%
_{\alpha }^{\limfunc{strong}}+\mathcal{E}_{\alpha }^{\limfunc{strong},\ast }+%
\sqrt[4]{\lambda }\mathfrak{N}_{T_{\delta ,R}^{\alpha }}\right) ,
\end{equation*}%
with $C_{\alpha }^{\prime }$ independent of $\delta $ and $R$. Since $%
\mathfrak{N}_{T_{\delta ,R}^{\alpha }}<\infty $, we can absorb the term $%
C_{\alpha }^{\prime }\sqrt{\lambda }\mathfrak{N}_{T_{\delta ,R}^{\alpha }}$
on the right hand side above into the left hand side for $\lambda >0$
sufficiently small. Since $T_{\delta ,R}^{\alpha }$ is an arbitrary tangent
line approximation to $T^{\alpha }$, the proof of Theorem \ref{T1 theorem'}
is complete.
\end{proof}

The first case of the following $T1$ theorem was proved in \cite{SaShUr8},
and the second case is a corollary of Theorem \ref{T1 theorem'} above and
Theorem 2 in \cite{SaShUr9}.

\begin{theorem}
\label{special T1}Suppose $0\leq \alpha <n$, that $T^{\alpha }$ is a
standard $\alpha $-fractional singular integral operator on $\mathbb{R}^{n}$%
, and that $\omega $ and $\sigma $ are locally finite positive Borel
measures on $\mathbb{R}^{n}$. Set $T_{\sigma }^{\alpha }f=T^{\alpha }\left(
f\sigma \right) $ for any smooth truncation of $T_{\sigma }^{\alpha }$. Let $%
\Omega :\mathbb{R}^{n}\rightarrow \mathbb{R}^{n}$ be a globally biLipschitz
map. Then%
\begin{equation*}
\mathfrak{N}_{T_{\sigma }^{\alpha }}\approx \sqrt{\mathfrak{A}_{2}^{\alpha }}%
+\mathfrak{T}_{T^{\alpha }}+\mathfrak{T}_{T^{\alpha }}^{\ast }\ ,
\end{equation*}%
in the following two cases:
\end{theorem}

\begin{enumerate}
\item when $T^{\alpha }$ is a strongly elliptic standard $\alpha $%
-fractional singular integral operator on $\mathbb{R}^{n}$, and one of the
weights $\sigma $ or $\omega $ is supported on a compact $C^{1,\delta }$
curve in $\mathbb{R}^{n}$,

\item when $T^{\alpha }=\mathbf{R}^{\alpha }$ is the vector of $\alpha $%
-fractional Riesz transforms, and both weights $\sigma $ and $\omega $ are $%
k $-energy\emph{\ }dispersed where $0\leq k\leq n-1$ satisfies%
\begin{equation*}
\left\{ 
\begin{array}{ccc}
n-k<\alpha <n,\ \alpha \neq n-1 & \text{ if } & 1\leq k\leq n-2 \\ 
0\leq \alpha <n,\ \alpha \neq 1,n-1 & \text{ if } & k=n-1%
\end{array}%
\right. .
\end{equation*}
\end{enumerate}

There is a further corollary that can be easily obtained, namely a \textbf{%
two weight} accretive global $Tb$ theorem whenever a two weight $T1$ theorem
holds for strictly comparable weight pairs. We say that two weight pairs $%
\left( \sigma ,\omega \right) $ and $\left( \widetilde{\sigma },\widetilde{%
\omega }\right) $ are \emph{strictly comparable} if $\widetilde{\sigma }%
=h_{1}\sigma $ and $\widetilde{\omega }=h_{2}\omega $ where each $h_{i}$ is
a function bounded between two positive constants. The simple proof of the
following accretive global $Tb$ theorem uses only the \emph{statement} of a
related $T1$ theorem. We say that a complex-valued function $b$ is \emph{%
accretive} on $\mathbb{R}^{n}$ if 
\begin{equation*}
0<c_{b}\leq \func{Re}b\left( x\right) \leq \left\vert b\left( x\right)
\right\vert \leq C_{b}<\infty ,\ \ \ \ \ x\in \mathbb{R}^{n}\ .
\end{equation*}

\begin{theorem}
\label{global Tb}Suppose $0\leq \alpha <n$, that $T^{\alpha }$ is a standard 
$\alpha $-fractional singular integral operator on $\mathbb{R}^{n}$, and
that $\omega $ and $\sigma $ are locally finite positive Borel measures on $%
\mathbb{R}^{n}$ for which we have the `$T1$ theorem' for strictly comparable
weight pairs, i.e.%
\begin{equation}
\mathfrak{N}_{T_{\sigma }^{\alpha }}\left( \widetilde{\sigma },\widetilde{%
\omega }\right) \approx \sqrt{\mathfrak{A}_{2}^{\alpha }\left( \widetilde{%
\sigma },\widetilde{\omega }\right) }+\mathfrak{T}_{T^{\alpha }}\left( 
\widetilde{\sigma },\widetilde{\omega }\right) +\mathfrak{T}_{T^{\alpha
}}^{\ast }\left( \widetilde{\sigma },\widetilde{\omega }\right) ,
\label{T1 comp}
\end{equation}%
whenever $\left( \sigma ,\omega \right) $ and $\left( \widetilde{\sigma },%
\widetilde{\omega }\right) $ are strictly comparable. Finally, let $b$ and $%
b^{\ast }$ be two accretive functions on $\mathbb{R}^{n}$. Then the best
constant $\mathfrak{N}_{T_{\sigma }^{\alpha }}=\mathfrak{N}_{T_{\sigma
}^{\alpha }}\left( \sigma ,\omega \right) $ in the two weight norm
inequality 
\begin{equation*}
\left\Vert T_{\sigma }^{\alpha }f\right\Vert _{L^{2}\left( \omega \right)
}\leq \mathfrak{N}_{T_{\sigma }^{\alpha }}\left\Vert f\right\Vert
_{L^{2}\left( \sigma \right) },
\end{equation*}%
taken uniformly over tangent line truncations of $T^{\alpha }$, satisfies%
\begin{equation}
\mathfrak{N}_{T_{\sigma }^{\alpha }}\approx \sqrt{\mathfrak{A}_{2}^{\alpha }}%
+\mathfrak{T}_{T^{\alpha }}^{b}+\mathfrak{T}_{T^{\alpha }}^{b^{\ast },\ast },
\label{Tb}
\end{equation}%
where the two dual $b$-testing conditions for $T^{\alpha }$ are given by%
\begin{eqnarray*}
\int_{Q}\left\vert T_{\sigma }^{\alpha }\left( \mathbf{1}_{Q}b\right)
\right\vert ^{2}d\omega &\leq &\mathfrak{T}_{T^{\alpha }}^{b}\left\vert
Q\right\vert _{\sigma }\ ,\ \ \ \ \ \text{for all cubes }Q, \\
\int_{Q}\left\vert T_{\omega }^{\alpha ,\ast }\left( \mathbf{1}_{Q}b^{\ast
}\right) \right\vert ^{2}d\sigma &\leq &\mathfrak{T}_{T^{\alpha }}^{b^{\ast
},\ast }\left\vert Q\right\vert _{\omega }\ ,\ \ \ \ \ \text{for all cubes }%
Q,
\end{eqnarray*}%
and where we interpret the left sides above as holding uniformly over all
tangent line truncations of $T^{\alpha }$.
\end{theorem}

Note that Theorem \ref{global Tb}\ applies in particular to both cases (1)
and (2) of Theorem \ref{special T1}.

\begin{proof}
We first note that since the kernel $K^{\alpha }$ is real-valued,%
\begin{eqnarray*}
\int_{Q}\left\vert T_{\sigma }^{\alpha }\left( \mathbf{1}_{Q}\func{Re}%
b\right) \right\vert ^{2}d\omega &=&\int_{Q}\left\vert \func{Re}T_{\sigma
}^{\alpha }\left( \mathbf{1}_{Q}b\right) \right\vert ^{2}d\omega \leq
\int_{Q}\left\vert T_{\sigma }^{\alpha }\left( \mathbf{1}_{Q}b\right)
\right\vert ^{2}d\omega \leq \mathfrak{T}_{T^{\alpha }}^{b}\left\vert
Q\right\vert _{\sigma }\ , \\
\int_{Q}\left\vert T_{\omega }^{\alpha ,\ast }\left( \mathbf{1}_{Q}\func{Re}%
b^{\ast }\right) \right\vert ^{2}d\sigma &=&\int_{Q}\left\vert \func{Re}%
T_{\omega }^{\alpha ,\ast }\left( \mathbf{1}_{Q}b^{\ast }\right) \right\vert
^{2}d\sigma \leq \int_{Q}\left\vert T_{\omega }^{\alpha ,\ast }\left( 
\mathbf{1}_{Q}b^{\ast }\right) \right\vert ^{2}d\sigma \leq \mathfrak{T}%
_{T^{\alpha }}^{b^{\ast },\ast }\left\vert Q\right\vert _{\omega }\ ,
\end{eqnarray*}%
and if we now define measures 
\begin{equation*}
\widetilde{\omega }\equiv \left( \func{Re}b^{\ast }\right) \omega \text{ and 
}\widetilde{\sigma }\equiv \left( \func{Re}b\right) \sigma \ ,
\end{equation*}%
we see that the operator $T^{\alpha }$ and the weight pair $\left( 
\widetilde{\sigma },\widetilde{\omega }\right) $ satisfy (\ref{T1 comp}).
But it follows that $\mathfrak{T}_{T^{\alpha }}\left( \widetilde{\sigma },%
\widetilde{\omega }\right) \approx \mathfrak{T}_{T^{\alpha }}^{b}\left(
\sigma ,\omega \right) $ and $\mathfrak{T}_{T^{\alpha }}^{\ast }\left( 
\widetilde{\sigma },\widetilde{\omega }\right) \approx \mathfrak{T}%
_{T^{\alpha }}^{b^{\ast },\ast }\left( \sigma ,\omega \right) $, and\ since
the Muckenhoupt $A_{2}$ conditions are clearly comparable for strictly
comparable weight pairs, we have the equivalence%
\begin{equation*}
\mathfrak{N}_{T_{\sigma }^{\alpha }}\left( \widetilde{\sigma },\widetilde{%
\omega }\right) \approx \sqrt{\mathfrak{A}_{2}^{\alpha }\left( \sigma
,\omega \right) }+\mathfrak{T}_{T^{\alpha }}^{b}\left( \sigma ,\omega
\right) +\mathfrak{T}_{T^{\alpha }}^{b^{\ast },\ast }\left( \sigma ,\omega
\right) .
\end{equation*}%
Finally, since $0<c\leq \func{Re}b,\func{Re}b^{\ast }\leq C$, we see that $%
\mathfrak{N}_{T_{\sigma }^{\alpha }}\left( \widetilde{\sigma },\widetilde{%
\omega }\right) \approx \mathfrak{N}_{T_{\sigma }^{\alpha }}\left( \sigma
,\omega \right) $, and this completes the proof of (\ref{Tb}).
\end{proof}

Note that the presence of a $\left( b,b^{\ast }\right) $-variant of the weak
boundedness property here would complicate matters, since in general, 
\begin{equation*}
\func{Re}\int_{Q}T^{\alpha }\left( 1_{Q^{\prime }}b\sigma \right) b^{\ast
}d\omega \neq \int_{Q}T^{\alpha }\left( 1_{Q^{\prime }}\func{Re}b\sigma
\right) \func{Re}b^{\ast }d\omega .
\end{equation*}%
To remind the reader of the versatility of even a \emph{global} $Tb$
theorem, we reproduce a proof of the boundedness of the Cauchy integral on $%
C^{1,\delta }$ curves.

\subsubsection{Boundedness of the Cauchy integral on $C^{1,\protect\delta }$
curves}

Here we point out how the above $Tb$ theorem can apply to obtain the
boundedness of the Cauchy integral on $C^{1,\delta }$ curves in the plane
(which can be obtained in many other easy ways as well, see e.g. \cite[%
Section 4 of Chapter VII]{Ste}). Recall that the problem reduces to
boundedness on $L^{2}\left( \mathbb{R}\right) $ of the singular integral
operator $C_{A}$ with kernel%
\begin{equation*}
K_{A}\left( x,y\right) \equiv \frac{1}{x-y+i\left( A\left( x\right) -A\left(
y\right) \right) },
\end{equation*}%
where the curve has graph $\left\{ x+iA\left( x\right) :x\in \mathbb{R}%
\right\} $. Now $b\left( x\right) \equiv 1+iA^{\prime }\left( x\right) $ is
accretive and we have the $b$-testing condition%
\begin{equation*}
\int_{I}\left\vert C_{A}\left( \mathbf{1}_{I}b\right) \left( x\right)
\right\vert ^{2}dx\leq \mathfrak{T}_{H}^{b}\left\vert I\right\vert ,
\end{equation*}%
and its dual. Indeed, if $I=\left[ \alpha ,\beta \right] $, then%
\begin{eqnarray*}
C_{A}\left( \mathbf{1}_{I}b\right) \left( x\right) &=&\int_{\alpha }^{\beta }%
\frac{1+iA^{\prime }\left( y\right) }{x-y+i\left( A\left( x\right) -A\left(
y\right) \right) }dy \\
&=&-\log \left( x-y+i\left( A\left( x\right) -A\left( y\right) \right)
\right) \mid _{\alpha }^{\beta } \\
&=&\log \left( \frac{x-\alpha +i\left( A\left( x\right) -A\left( \alpha
\right) \right) }{x-\beta +i\left( A\left( x\right) -A\left( \beta \right)
\right) }\right) ,
\end{eqnarray*}%
gives 
\begin{equation*}
\left\vert C_{A}\left( \mathbf{1}_{I}b\right) \left( x\right) \right\vert
^{2}\approx \ln \frac{x-\alpha }{\beta -x},\ \ \ \ \ x\in I=\left[ \alpha
,\beta \right] ,
\end{equation*}%
and it follows that%
\begin{equation*}
\int_{I}\left\vert C_{A}\left( \mathbf{1}_{I}b\right) \left( x\right)
\right\vert ^{2}dx\approx \int_{I}\left\vert \ln \frac{x-\alpha }{\beta -x}%
\right\vert ^{2}dx\approx \int_{0}^{\beta -\alpha }\left\vert \ln \frac{x}{%
\beta -\alpha }\right\vert ^{2}dx=\left( \beta -\alpha \right)
\int_{0}^{1}\left\vert \ln w\right\vert ^{2}dw=C\left\vert I\right\vert .
\end{equation*}%
Since the kernel $K_{A}$ is $C^{1,\delta }$, the $Tb$ theorem above applies
with $T=C_{A}$ and $\sigma =\omega =dx$ Lebesgue measure, to show that $%
C_{A} $ is bounded on $L^{2}\left( \mathbb{R}\right) $. Of course this proof
just misses the case of Lipschitz curves since our two weight $Tb$ theorem
does not apply to kernels that fail to be $C^{1,\delta }$.

\section{Proof of the Good-$\protect\lambda $ Lemma}

We will prove the Good-$\lambda $ Lemma by first replacing the quasiweak
boundedness constant on the left hand side of (\ref{g lambda}) with the
indicator/touching constant introduced in (\ref{Ind/touch}) above. To
control the indicator/touching constant, we will need to tweak the usual
good/bad technology of NTV a bit in the following subsection.

\subsection{Good/bad technology\label{g/b tech}}

First we recall the good/bad cube technology of Nazarov, Treil and Volberg 
\cite{Vol} as in \cite{SaShUr7}, but with a small simplification introduced
in the real line by Hyt\"{o}nen in \cite{Hyt2}. This simplification does not
impact the validity of the arguments in \cite{SaShUr6}, but will facilitate
the use of NTV surgery in later subsections.

Following \cite{Hyt2}, we momentarily fix a large positive integer $M\in 
\mathbb{N}$, and consider the tiling of $\mathbb{R}^{n}$ by the family of
cubes $\mathbb{D}_{M}\equiv \left\{ I_{\alpha }^{M}\right\} _{\alpha \in
Z^{n}}$ having side length $2^{-M}$ and given by $I_{\alpha }^{M}\equiv
I_{0}^{M}+2^{-M}\alpha $ where $I_{0}^{M}=\left[ 0,2^{-M}\right) ^{n}$. A 
\emph{dyadic grid} $\mathcal{D}$ built on $\mathbb{D}_{M}$ is\ defined to be
a family of cubes $\mathcal{D}$ satisfying:

\begin{enumerate}
\item Each $I\in \mathcal{D}$ has side length $2^{-\ell }$ for some $\ell
\in \mathbb{Z}$ with $\ell \leq M$, and $I$ is a union of $2^{n\left( M-\ell
\right) }$ cubes from the tiling $\mathbb{D}_{M}$,

\item For $\ell \leq M$, the collection $\mathcal{D}_{\ell }$ of cubes in $%
\mathcal{D}$ having side length $2^{-\ell }$ forms a pairwise disjoint
decomposition of the space $\mathbb{R}^{n}$,

\item Given $I\in \mathcal{D}_{i}$ and $J\in \mathcal{D}_{j}$ with $j\leq
i\leq M$, it is the case that either $I\cap J=\emptyset $ or $I\subset J$.
\end{enumerate}

We now momentarily fix a \emph{negative} integer $N\in -\mathbb{N}$, and
restrict the above grids to cubes of side length at most $2^{-N}$:%
\begin{equation*}
\mathcal{D}^{N}\equiv \left\{ I\in \mathcal{D}:\text{side length of }I\text{
is at most }2^{-N}\right\} \text{.}
\end{equation*}%
We refer to such grids $\mathcal{D}^{N}$ as a (truncated) dyadic grid $%
\mathcal{D}$ built on $\mathbb{D}_{M}$ of size $2^{-N}$. There are now two
traditional means of constructing probability measures on collections of
such dyadic grids.

\textbf{Construction \#1}: Consider first the special case of dimension $n=1$%
. Then for any 
\begin{equation*}
\beta =\{\beta _{i}\}_{i\in _{M}^{N}}\in \omega _{M}^{N}\equiv \left\{
0,1\right\} ^{\mathbb{Z}_{M}^{N}},
\end{equation*}%
where $\mathbb{Z}_{M}^{N}\equiv \left\{ \ell \in \mathbb{Z}:N\leq \ell \leq
M\right\} $, define the dyadic grid $\mathcal{D}_{\beta }$ built on $\mathbb{%
D}_{M}$ of size $2^{-N}$ by 
\begin{equation*}
\mathcal{D}_{\beta }=\left\{ 2^{-\ell }\left( [0,1)+k+\sum_{i:\ \ell <i\leq
M}2^{-i+\ell }\beta _{i}\right) \right\} _{N\leq \ell \leq M,\,k\in {\mathbb{%
Z}}}\ .
\end{equation*}%
Place the uniform probability measure $\rho _{M}^{N}$ on the finite index
space $\omega _{M}^{N}=\left\{ 0,1\right\} ^{\mathbb{Z}_{M}^{N}}$, namely
that which charges each $\beta \in \omega _{M}^{N}$ equally. This
construction is then extended to Euclidean space $\mathbb{R}^{n}$ by taking
products in the usual way and using the product index space $\Omega
_{M}^{N}\equiv \left( \omega _{M}^{N}\right) ^{n}$ and the uniform product
probability measure $\mu _{M}^{N}=\rho _{M}^{N}\times ...\times \rho
_{M}^{N} $.

\textbf{Construction \#2}: Momentarily fix a (truncated) dyadic grid $%
\mathcal{D}$ built on $\mathbb{D}_{M}$ of size $2^{-N}$. For any 
\begin{equation*}
\gamma =\left( \gamma _{1},...,\gamma _{n}\right) \in \Gamma _{M}^{N}\equiv
\left\{ 2^{-M}\mathbb{Z}_{+}^{n}:\left\vert \gamma _{i}\right\vert
<2^{-N}\right\} ,
\end{equation*}%
where $\mathbb{Z}_{+}=\mathbb{N}\cup \left\{ 0\right\} $, define the dyadic
grid $\mathcal{D}^{\gamma }$ built on $\mathbb{D}_{M}$ of size $2^{-N}$ by%
\begin{equation*}
\mathcal{D}^{\gamma }\equiv \mathcal{D}+\gamma .
\end{equation*}%
Place the uniform probability measure $\nu _{M}^{N}$ on the finite index set 
$\Gamma _{M}^{N}$, namely that which charges each multiindex $\gamma $ in $%
\Gamma _{M}^{N}$ equally.

The two probability spaces $\left( \left\{ \mathcal{D}_{\beta }\right\}
_{\beta \in \Omega _{M}^{N}},\mu _{M}^{N}\right) $ and $\left( \left\{ 
\mathcal{D}^{\gamma }\right\} _{\gamma \in \Gamma _{M}^{N}},\nu
_{M}^{N}\right) $ are isomorphic since both collections $\left\{ \mathcal{D}%
_{\beta }\right\} _{\beta \in \Omega _{M}^{N}}$ and $\left\{ \mathcal{D}%
^{\gamma }\right\} _{\gamma \in \Gamma _{M}^{N}}$ describe the set $%
\boldsymbol{A}_{M}^{N}$ of \textbf{all} (truncated) dyadic grids $\mathcal{D}%
^{\gamma }$ built on $\mathbb{D}_{M}$ of size $2^{-N}$, and since both
measures $\mu _{M}^{N}$ and $\nu _{M}^{N}$ are the uniform measure on this
space. Indeed, it suffices to verify this in the case $n=1$. The first
construction may be thought of as being \emph{parameterized by scales} -
each component $\beta _{i}$ in $\beta =\{\beta _{i}\}_{i\in _{M}^{N}}\in
\omega _{M}^{N}$ amounting to a choice of the two possible tilings at level $%
i$ that respect the choice of tiling at the level below - and since any grid
in $\boldsymbol{A}_{M}^{N}$ is determined by a choice of scales , we see
that $\left\{ \mathcal{D}_{\beta }\right\} _{\beta \in \Omega _{M}^{N}}=%
\boldsymbol{A}_{M}^{N}$. The second construction may be thought of as being 
\emph{parameterized by translation} - each $\gamma \in \Gamma _{M}^{N}$
amounting to a choice of translation of the grid $\mathcal{D}$ fixed in
construction \#2\ - and since any grid in $\boldsymbol{A}_{M}^{N}$ is
determined by any of the cubes at the top level, i.e. with side length $%
2^{-N}$, we see that $\left\{ \mathcal{D}^{\gamma }\right\} _{\gamma \in
\Gamma _{M}^{N}}=\boldsymbol{A}_{M}^{N}$ as well, since every cube at the
top level in $\boldsymbol{A}_{M}^{N}$ has the form $Q+\gamma $ for some $%
\gamma \in \Gamma _{M}^{N}$ and $Q\in \mathcal{D}$ at the top level in $%
\boldsymbol{A}_{M}^{N}$ (i.e. every cube at the top level in $\boldsymbol{A}%
_{M}^{N}$ is a union of small cubes in $\mathbb{D}_{M}$, and so must be a
translate of some $Q\in \mathcal{D}$ by an amount $2^{-M}$ times an element
of $\mathbb{Z}_{+}^{n}$). Note also that in all dimensions, $\#\Omega
_{M}^{N}=\#\Gamma _{M}^{N}=2^{n\left( M-N\right) }$. We will use $\mathbb{E}%
_{\Omega _{M}^{N}}$ to denote expectation with respect to this common
probability measure on $\boldsymbol{A}_{M}^{N}$.

The usual NTV probabilistic reduction to `good' cubes will be implemented
below for each positive integer $M$ and each negative integer $N$ assuming
that the functions $f$ and $g$ are supported in a large cube $L$ with $%
\int_{L}fd\sigma =0=\int_{L}gd\omega $, and moreover assuming that $-N$ is
sufficiently large compared to $\ell \left( L\right) $ that the small
probability estimates claimed below hold ($-N>\ell \left( L\right) +\mathbf{r%
}$ will work where $\mathbf{r}$ is the goodness constant), and finally
assuming that $f$ and $g$ are constant on each cube $Q$ in the tiling $%
\mathbb{D}_{M}$. Recall that we can always reduce to the case $%
\int_{L}fd\sigma =0=\int_{L}gd\omega $ by simply subtracting off averages
and controlling the resulting error terms by the testing conditions (see
e.g. \cite{Vol}).

\begin{notation}
For purposes of notation and clarity, we often suppress all reference to $M$
and $N$ in our families of grids, and in the notations $\Omega $ and $\Gamma 
$ for the parameter sets, and we will use $\mathbb{P}_{\Omega }$ and $%
\mathbb{E}_{\Omega }$ to denote probability and expectation, and instead
proceed as if all grids considered are unrestricted. The careful reader can
supply the modifications necessary to handle the assumptions made above on
the grids $\mathcal{D}$ and the functions $f$ and $g$ regarding $M$ and $N$.
In fact, we will exploit the integers $M$ and $N$ explicitly in the
subsubsections on NTV surgery below.
\end{notation}

In the case of one independent family of grids, as is the case here, the
main result is the following \emph{conditional} probability estimate: for
every $I\in \mathcal{P}^{n}$,%
\begin{equation}
\mathbb{P}_{\Omega }\left\{ \mathcal{D}:I\text{ is a \emph{bad} cube in }%
\mathcal{D}\mid I\in \mathcal{D}\right\} \leq C2^{-\varepsilon \mathbf{r}}.
\label{cond prob}
\end{equation}%
Provided we obtain estimates independent of $M$ and $N$, this will be
sufficient for our proof - this follows the procedure with \emph{two}
independent grids initiated by Hyt\"{o}nen for the Hilbert transform
inequality in \cite{Hyt2}. The key point of introducing the two different
parameterizations above of the same probability space, is that construction
\#1 is well-adapted to the reduction to good cubes in a \emph{single}
independent family of grids, as used in the proof of the main theorem in 
\cite{SaShUr6}, which is in turn needed below, while construction \#2
facilitates the use of NTV surgery below when combined with the construction
of $Q$-good \emph{grids,} to which we next turn.

\subsubsection{$Q$-good quasicubes and $Q$-good quasigrids}

We first introduce these notions for usual cubes, and later pass to
quasicubes. Let $Q\in \mathcal{P}^{n}$ be an arbitrary cube in $\mathbb{R}%
^{n}$ with sides parallel to the coordinate axes. For technical reasons
associated to our application below, we also want to consider the `siblings'
of $Q$, i.e. the `triadic children' of $3Q$.

\begin{definition}
\label{cube Qgood}We say that a cube $I\in \mathcal{P}^{n}$ is $Q$\emph{-good%
} if either $\ell \left( I\right) >2^{-\mathbf{\rho }}\ell \left( Q\right) $%
, or for every sibling $Q^{\prime }$ of $Q$, we have 
\begin{equation*}
\limfunc{dist}\left( I,\partial Q^{\prime }\right) \geq \frac{1}{2}\ell
\left( I\right) ^{\varepsilon }\ell \left( Q^{\prime }\right)
^{1-\varepsilon }
\end{equation*}%
when $\ell \left( I\right) \leq 2^{-\mathbf{\rho }}\ell \left( Q\right) $.
We say $I\in \mathcal{P}^{n}$ is $Q$\emph{-bad} if $I$ is \textbf{not} $Q$%
-good.
\end{definition}

Note that for a fixed cube $Q\in \mathcal{P}^{n}$, we do \textbf{not} have a
conditional probability estimate $\mathbb{P}_{\Omega }\left\{ \mathcal{D}%
:I\in \mathcal{D}\text{ and }I\text{ is }Q\text{-bad}\right\} \leq
C2^{-\varepsilon \mathbf{r}}$ since the property of a cube $I$ being $Q$-bad
is independent of which grids $\mathcal{D}$ it belongs to. To rectify this
complication we will introduce below a \emph{second independent} family of
grids - but this second family will also be used to simultaneously
Haar-decompose both $f\in L^{2}\left( \sigma \right) $ and $g\in L^{2}\left(
\omega \right) $\footnote{%
Traditionally, two independent grids are applied to $f$ and $g$ separately,
something we \emph{avoid} since the treatment of functional energy in the
arguments of \cite{SaShUr9}, \cite{SaShUr6} (which we use here) relies on
using a \emph{common} grid for $f$ and $g$.}.

We next wish to capture the idea of a grid $\mathcal{D}$ being `$Q$-\emph{%
good}' with respect to this fixed cube $Q$, and the idea will be to require
that $Q$ is $I$-good for all sufficiently larger cubes $I$ in the grid $%
\mathcal{D}$. Here we \emph{will} obtain a `goodness' estimate in Lemma \ref%
{bad grids} below.

\begin{definition}
\label{crit}Let $\mathbf{r}$ and $\varepsilon $ be goodness constants as in 
\cite{SaShUr7}. For $Q\in \mathcal{P}^{n}$ we declare a grid $\mathcal{D}$
to be $Q$\emph{-good} if for every sibling $Q^{\prime }$ of $Q$ and for
every $I\in \mathcal{D}$ with $\ell \left( I\right) \geq 2^{\mathbf{r}}\ell
\left( Q\right) $, the following holds: the distance from the cube $%
Q^{\prime }$ to the boundary of the cube $I$ satisfies the `deeply embedded'
inequality,%
\begin{equation*}
\limfunc{dist}\left( Q^{\prime },\partial I\right) \geq \frac{1}{2}\ell
\left( Q^{\prime }\right) ^{\varepsilon }\ell \left( I\right)
^{1-\varepsilon }.
\end{equation*}%
We say the grid $\mathcal{D}$ is $Q$\emph{-bad} if it is not $Q$-good.
\end{definition}

Note that $Q$ is fixed in this definition and it is easy to see, using the
translation parameterization in construction \#2 above, that the collection
of grids $\mathcal{D}$ that are $Q$-bad occur with small probability.
Indeed, if $I\supset Q$ has side length at least $2^{\mathbf{r}}$ times that
of $Q$, then the translates of $I$ satisfy $Q\Subset _{\mathbf{r}}I$ with
probability near $1$.

\begin{lemma}
\label{bad grids}Fix a cube $Q\in \mathcal{P}^{n}$. Then $\mathbb{P}_{\Omega
}\left\{ \mathcal{D}:\mathcal{D}\text{ is }Q\text{-bad}\right\} \leq
C2^{-\varepsilon \mathbf{r}}$.
\end{lemma}

The following is our tweaking of the good/bad technology of NTV \cite{Vol}.
Fix a cube $Q\in \mathcal{P}^{n}$ and let $\mathcal{D}$ be randomly
selected. Define linear operators (depending on the grid $\mathcal{D}$), 
\begin{eqnarray*}
\mathsf{P}_{Q;\QTR{up}{\limfunc{good}}}^{\sigma }f &\equiv &\left\{ 
\begin{array}{ccc}
\sum_{I\in \mathcal{D}:\ I\text{ is }\mathbf{r}\text{-good in }\mathcal{D}%
}\bigtriangleup _{I}^{\sigma }f & \text{ if } & \mathcal{D}\text{ is }Q\text{%
-good} \\ 
0 & \text{ if } & \mathcal{D}\text{ is }Q\text{-bad}%
\end{array}%
\right. \,, \\
\mathsf{P}_{Q;\QTR{up}{\limfunc{bad}}}^{\sigma }f &\equiv &f-\mathsf{P}_{Q;%
\QTR{up}{\limfunc{good}}}^{\sigma }f\ ,
\end{eqnarray*}%
and likewise for $\mathsf{P}_{Q;\QTR{up}{\limfunc{good}}}^{\omega }g$ and $%
\mathsf{P}_{Q;\QTR{up}{\limfunc{bad}}}^{\omega }g$.

\begin{proposition}
\label{p.Pgood}Fix a cube $Q\in \mathcal{P}^{n}$. Then we have the estimates 
\begin{eqnarray*}
\mathbb{E}_{\Omega }\left\Vert \mathsf{P}_{Q;\QTR{up}{\limfunc{bad}}%
}^{\sigma }f\right\Vert _{L^{2}(\sigma )} &\leq &C2^{-\frac{\varepsilon 
\mathbf{r}}{2}}\left\Vert f\right\Vert _{L^{2}(\sigma )}, \\
\mathbb{E}_{\Omega }\left\Vert \mathsf{P}_{Q;\QTR{up}{\limfunc{bad}}%
}^{\omega }g\right\Vert _{L^{2}(\omega )} &\leq &C2^{-\frac{\varepsilon 
\mathbf{r}}{2}}\left\Vert g\right\Vert _{L^{2}(\omega )}.
\end{eqnarray*}
\end{proposition}

\begin{proof}
We have from (\ref{cond prob}) and Lemma \ref{bad grids} that 
\begin{align*}
\mathbb{E}_{\Omega }\left\Vert \mathsf{P}_{\QTR{up}{\limfunc{bad}}}^{\sigma
}f\right\Vert _{L^{2}(\sigma )}^{2}& =\mathbb{E}_{\Omega }\left( \mathbf{1}%
_{\left\{ \mathcal{D}\text{ is }Q\text{-good}\right\} }\sum_{I\in \mathcal{D}%
\text{ is bad}}\left\Vert \bigtriangleup _{I}^{\sigma }f\right\Vert
_{L^{2}\left( \sigma \right) }^{2}\right) +\mathbb{E}_{\Omega }\left( 
\mathbf{1}_{\left\{ \mathcal{D}\text{ is }Q\text{-bad}\right\} }\sum_{I\in 
\mathcal{D}}\left\Vert \bigtriangleup _{I}^{\sigma }f\right\Vert
_{L^{2}\left( \sigma \right) }^{2}\right) \\
& \leq C2^{-\varepsilon \mathbf{r}}\sum_{I\in \mathcal{D}}\left\Vert
\bigtriangleup _{I}^{\sigma }f\right\Vert _{L^{2}\left( \sigma \right) }^{2}+%
\mathbb{E}_{\Omega }\left( \mathbf{1}_{\left\{ \mathcal{D}\text{ is }Q\text{%
-bad}\right\} }\right) \sum_{I\in \mathcal{D}}\left\Vert \bigtriangleup
_{I}^{\sigma }f\right\Vert _{L^{2}\left( \sigma \right) }^{2}\lesssim
C2^{-\varepsilon \mathbf{r}}\left\Vert f\right\Vert _{L^{2}(\sigma )}^{2}\,.
\end{align*}
\end{proof}

From this we conclude that there is an absolute choice of $\mathbf{r}$
depending on $0<\varepsilon <1$ so that the following holds. Let $%
T\;:\;L^{2}(\sigma )\rightarrow L^{2}(\omega )$ be a bounded linear
operator, and let $Q\in \mathcal{P}^{n}$ be a fixed cube. We then have 
\begin{equation}
\left\Vert T\right\Vert _{L^{2}(\sigma )\rightarrow L^{2}(\omega )}\leq
2\sup_{\left\Vert f\right\Vert _{L^{2}(\sigma )}=1}\sup_{\left\Vert
g\right\Vert _{L^{2}(\omega )}=1}\mathbb{E}_{\Omega }\lvert \left\langle T%
\mathsf{P}_{Q;\QTR{up}{\limfunc{good}}}^{\sigma }f,\mathsf{P}_{Q;\QTR{up}{%
\limfunc{good}}}^{\omega }g\right\rangle _{\omega }\rvert \,.  \label{Tgood}
\end{equation}%
Indeed, we can choose $f\in L^{2}(\sigma )$ of norm one, and $g\in
L^{2}(\omega )$ of norm one so that 
\begin{align*}
\left\Vert T\right\Vert _{L^{2}(\sigma )\rightarrow L^{2}(\omega )}&
=\left\langle Tf,g\right\rangle _{\omega } \\
& \leq \mathbb{E}_{\Omega }\lvert \left\langle T\mathsf{P}_{Q;\QTR{up}{%
\limfunc{good}}}^{\sigma }f,\mathsf{P}_{Q;\QTR{up}{\limfunc{good}}}^{\omega
}g\right\rangle _{\omega }\rvert +\mathbb{E}_{\Omega }\lvert \left\langle T%
\mathsf{P}_{Q;\QTR{up}{\limfunc{bad}}}^{\sigma }f,\mathsf{P}_{Q;\QTR{up}{%
\limfunc{good}}}^{\omega }g\right\rangle _{\omega }\rvert \\
& \quad +\mathbb{E}_{\Omega }\lvert \left\langle T\mathsf{P}_{Q;\QTR{up}{%
\limfunc{good}}}^{\sigma }f,\mathsf{P}_{Q;\QTR{up}{\limfunc{bad}}}^{\omega
}g\right\rangle _{\omega }\rvert +\mathbb{E}_{\Omega }\lvert \left\langle T%
\mathsf{P}_{Q;\QTR{up}{\limfunc{bad}}}^{\sigma }f,\mathsf{P}_{Q;\QTR{up}{%
\limfunc{bad}}}^{\omega }g\right\rangle _{\omega }\rvert \\
& \leq \mathbb{E}_{\Omega }\lvert \left\langle T\mathsf{P}_{Q;\QTR{up}{%
\limfunc{good}}}^{\sigma }f,\mathsf{P}_{Q;\QTR{up}{\limfunc{good}}}^{\omega
}g\right\rangle _{\omega }\rvert +3C\cdot 2^{-\frac{\mathbf{r}\varepsilon }{%
16}}\left\Vert T\right\Vert _{L^{2}(\sigma )\rightarrow L^{2}(\omega )}\ ,
\end{align*}%
And this proves (\ref{Tgood}) for $\mathbf{r}$ sufficiently large depending
on $\varepsilon >0$.

Clearly, all of this extends automatically to the quasiworld.

\begin{description}
\item[Implication] \emph{Given a quasicube }$Q\in \Omega \mathcal{P}^{n}$%
\emph{, }$i$\emph{t suffices to consider only }$Q$\emph{-good quasigrids and 
}$Q$\emph{-good} \emph{quasicubes in these quasigrids, and to prove an
estimate for $\left\Vert T_{\sigma }\right\Vert _{L^{2}(\sigma )\rightarrow
L^{2}(\omega )}$ that is independent of these assumptions.}
\end{description}

\subsection{Control of the indicator/touching property}

Recall the indicator/touching constant $\mathfrak{I}_{T^{\alpha }}$ defined
in (\ref{Ind/touch}) above. Here we will prove that%
\begin{equation}
\mathfrak{I}_{T^{\alpha }}\leq C_{\alpha }\left( \frac{1}{\lambda }\sqrt{%
\mathfrak{A}_{2}^{\alpha }}+\mathfrak{T}_{T^{\alpha }}+\mathfrak{T}%
_{T^{\alpha }}^{\ast }+\mathcal{E}_{\alpha }^{\limfunc{strong}}+\mathcal{E}%
_{\alpha }^{\limfunc{strong},\ast }+\sqrt[4]{\lambda }\mathfrak{N}%
_{T^{\alpha }}\right) ,  \label{I/T}
\end{equation}%
from which it easily follows that we have the same inequality for the weak
boundedness property constant $\mathcal{WBP}_{T^{\alpha }}$ defined in (\ref%
{def WBP}) above,%
\begin{equation}
\mathcal{WBP}_{T^{\alpha }}\leq C_{\alpha }\left( \frac{1}{\lambda }\sqrt{%
\mathfrak{A}_{2}^{\alpha }}+\mathfrak{T}_{T^{\alpha }}+\mathfrak{T}%
_{T^{\alpha }}^{\ast }+\mathcal{E}_{\alpha }^{\limfunc{strong}}+\mathcal{E}%
_{\alpha }^{\limfunc{strong},\ast }+\sqrt[4]{\lambda }\mathfrak{N}%
_{T^{\alpha }}\right) .  \label{WBP}
\end{equation}%
Indeed an elementary argument shows that $\mathcal{WBP}_{T^{\alpha
}}\lesssim \mathfrak{I}_{T^{\alpha }}+\sqrt{\mathfrak{A}_{2}^{\alpha }}+%
\mathfrak{T}_{T^{\alpha }}$. For the proof of (\ref{I/T}) we assume the
reader is already familiar with the proof of the main theorem in \cite%
{SaShUr6} or \cite{SaShUr9}, and we now review the parts of this proof that
are pertinent here.

We first recall the basic setup in \cite{SaShUr6}. Let $\Omega \mathcal{D}%
^{\sigma }=\Omega \mathcal{D}^{\omega }$ be a quasigrid on $\mathbb{R}^{n}$,
and let $\left\{ h_{I}^{\sigma ,a}\right\} _{I\in \Omega \mathcal{D}^{\sigma
},\ a\in \Gamma _{n}}$ and $\left\{ h_{J}^{\omega ,b}\right\} _{J\in \Omega 
\mathcal{D}^{\omega },\ b\in \Gamma _{n}}$ be corresponding quasiHaar bases,
so that $f\in L^{2}\left( \sigma \right) $ and $g\in L^{2}\left( \omega
\right) $ can be written $f=f_{\limfunc{good}}+f_{\limfunc{bad}}$ and $g=g_{%
\limfunc{good}}+g_{\limfunc{bad}}$ where 
\begin{eqnarray*}
f &=&\sum_{I\in \Omega \mathcal{D}^{\sigma }}\bigtriangleup _{I}^{\sigma }f%
\text{ and }g=\sum_{J\in \Omega \mathcal{D}^{\omega }\text{ }}\bigtriangleup
_{J}^{\omega }g\ , \\
f_{\limfunc{good}} &=&\sum_{I\in \Omega \mathcal{D}_{\limfunc{good}}^{\sigma
}}\bigtriangleup _{I}^{\sigma }f\text{ and }g_{\limfunc{good}}=\sum_{J\in
\Omega \mathcal{D}_{\limfunc{good}}^{\omega }\text{ }}\bigtriangleup
_{J}^{\omega }g\ ,
\end{eqnarray*}%
and where $\Omega \mathcal{D}_{\limfunc{good}}^{\sigma }=\Omega \mathcal{D}_{%
\limfunc{good}}^{\omega }$ is the $\left( \mathbf{r},\varepsilon \right) $%
-good subgrid, and where the quasiHaar projections $\bigtriangleup
_{I}^{\sigma }f_{\limfunc{good}}$ and $\bigtriangleup _{J}^{\omega }g_{%
\limfunc{good}}$ \emph{vanish} if the quasicubes $I$ and $J$ are not good in 
$\Omega \mathcal{D}^{\sigma }=\Omega \mathcal{D}^{\omega }$. Note that we
use a \emph{single} independent family of grids $\Omega \mathcal{D}^{\sigma
}=\Omega \mathcal{D}^{\omega }$ and only include the different superscripts $%
\sigma $ and $\omega $ to emphasize which measure the grid is being used
with in a given situation.

\begin{remark}
In \cite{SaShUr9} and \cite{SaShUr6}, the quasiHaar projections $%
\bigtriangleup _{I}^{\sigma }f_{\limfunc{good}}$ and $\bigtriangleup
_{J}^{\omega }g_{\limfunc{good}}$ are required to vanish if the quasicubes $I
$ and $J$ are not $\mathbf{\tau }$-good in $\Omega \mathcal{D}^{\sigma
}=\Omega \mathcal{D}^{\omega }$, where a quasicube $I$ is $\mathbf{\tau }$%
-good in a quasigrid $\Omega \mathcal{D}$ if $I$ together with its children
and its ancestors up to order $\mathbf{\tau }$ are all good. This more
restrictive condition doesn't affect what is done here.
\end{remark}

For future reference note that the argument in \cite{SaShUr6} applies just
as well to the smaller projections $\mathsf{P}_{Q;\QTR{up}{\limfunc{good}}%
}^{\sigma }f$ and $\mathsf{P}_{Q;\QTR{up}{\limfunc{good}}}^{\omega }g$ in
place of $f_{\limfunc{good}}$ and $g_{\limfunc{good}}$ respectively. We fix $%
f=f_{\limfunc{good}}$ and $g=g_{\limfunc{good}}$. For now we continue to
work with general functions $f$ and $g$ and the projections $f_{\limfunc{good%
}}$ and $g_{\limfunc{good}}$, but keeping in mind that in order to prove (%
\ref{I/T}), we will later specialize to the cases of indicator functions $f=%
\mathbf{1}_{Q}$ and $g=\mathbf{1}_{R}$, and we will then also include the
restriction to $Q$-good grids $\Omega \mathcal{D}_{Q;\limfunc{good}}$ and
projections $\mathsf{P}_{Q;\QTR{up}{\limfunc{good}}}^{\sigma }f$ and $%
\mathsf{P}_{Q;\QTR{up}{\limfunc{good}}}^{\omega }g$ for a fixed quasicube $Q$
- the quasicube $Q$ in the projection $\mathsf{P}_{Q;\QTR{up}{\limfunc{good}}%
}^{\sigma }f$ is chosen to coincide with the quasicube $Q$ in the indicator $%
\mathbf{1}_{Q}$ in order to achieve the three critical reductions in
Subsubsecion \ref{three crit} below. Continuing with \cite{SaShUr9}, \cite%
{SaShUr6}, we then proved there the bilinear inequality 
\begin{eqnarray}
\left\vert \mathcal{T}^{\alpha }\left( f,g\right) \right\vert &=&\left\vert
\sum_{I\in \Omega \mathcal{D}_{\limfunc{good}}^{\sigma }\text{ and }J\in
\Omega \mathcal{D}_{\limfunc{good}}^{\omega }}\mathcal{T}^{\alpha }\left(
\bigtriangleup _{I}^{\sigma }f,\bigtriangleup _{J}^{\omega }g\right)
\right\vert  \label{bilin} \\
&\leq &C_{\alpha }\left( \sqrt{\mathfrak{A}_{2}^{\alpha }}+\mathfrak{T}%
_{T^{\alpha }}+\mathfrak{T}_{T^{\alpha }}^{\ast }+\mathcal{E}_{\alpha }^{%
\limfunc{strong}}+\mathcal{E}_{\alpha }^{\limfunc{strong},\ast }+\mathcal{WBP%
}_{T^{\alpha }}\right) \left\Vert f\right\Vert _{L^{2}\left( \sigma \right)
}\left\Vert g\right\Vert _{L^{2}\left( \omega \right) },  \notag
\end{eqnarray}%
uniformly over grids $\mathcal{D}$, and we now discuss the salient features
of this proof for us.

As in \cite{SaShUr9}, \cite{SaShUr6} let 
\begin{eqnarray*}
\mathcal{NTV}_{\alpha } &\equiv &\sqrt{\mathfrak{A}_{2}^{\alpha }}+\mathfrak{%
T}_{T^{\alpha }}+\mathfrak{T}_{T^{\alpha }}^{\ast }+\mathcal{WBP}_{T^{\alpha
}}\ , \\
\mathfrak{A}_{2}^{\alpha } &\equiv &\mathcal{A}_{2}^{\alpha }+\mathcal{A}%
_{2}^{\alpha ,\ast }+A_{2}^{\alpha ,\limfunc{punct}}+A_{2}^{\alpha ,\ast ,%
\limfunc{punct}}\ ,
\end{eqnarray*}%
and recall the following brief schematic diagram of the decompositions
involved in the proof given in \cite{SaShUr6}, with bounds in $\fbox{}$: 
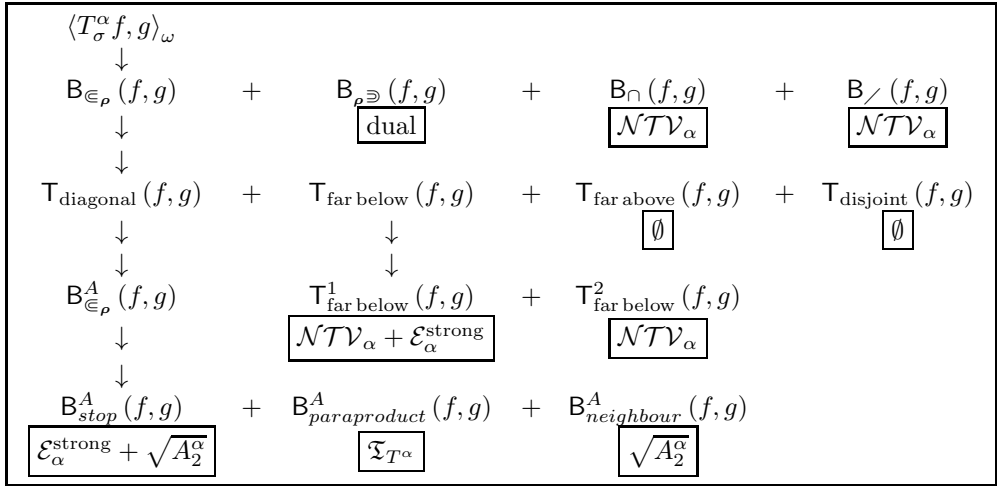
\begin{equation*}
\fbox{$%
\begin{array}{ccccccc}
\left\langle T_{\sigma }^{\alpha }f,g\right\rangle _{\omega } &  &  &  &  & 
&  \\ 
\downarrow &  &  &  &  &  &  \\ 
\mathsf{B}_{\Subset _{\mathbf{\rho }}}\left( f,g\right) & + & \mathsf{B}_{_{%
\mathbf{\rho }}\Supset }\left( f,g\right) & + & \mathsf{B}_{\cap }\left(
f,g\right) & + & \mathsf{B}_{\diagup }\left( f,g\right) \\ 
\downarrow &  & \fbox{dual} &  & \fbox{$\mathcal{NTV}_{\alpha }$} &  & \fbox{%
$\mathcal{NTV}_{\alpha }$} \\ 
\downarrow &  &  &  &  &  &  \\ 
\mathsf{T}_{\limfunc{diagonal}}\left( f,g\right) & + & \mathsf{T}_{\limfunc{%
far}\limfunc{below}}\left( f,g\right) & + & \mathsf{T}_{\limfunc{far}%
\limfunc{above}}\left( f,g\right) & + & \mathsf{T}_{\limfunc{disjoint}%
}\left( f,g\right) \\ 
\downarrow &  & \downarrow &  & \fbox{$\emptyset $} &  & \fbox{$\emptyset $}
\\ 
\downarrow &  & \downarrow &  &  &  &  \\ 
\mathsf{B}_{\Subset _{\mathbf{\rho }}}^{A}\left( f,g\right) &  & \mathsf{T}_{%
\limfunc{far}\limfunc{below}}^{1}\left( f,g\right) & + & \mathsf{T}_{%
\limfunc{far}\limfunc{below}}^{2}\left( f,g\right) &  &  \\ 
\downarrow &  & \fbox{$\mathcal{NTV}_{\alpha }+\mathcal{E}_{\alpha }^{%
\limfunc{strong}}$} &  & \fbox{$\mathcal{NTV}_{\alpha }$} &  &  \\ 
\downarrow &  &  &  &  &  &  \\ 
\mathsf{B}_{stop}^{A}\left( f,g\right) & + & \mathsf{B}_{paraproduct}^{A}%
\left( f,g\right) & + & \mathsf{B}_{neighbour}^{A}\left( f,g\right) &  &  \\ 
\fbox{$\mathcal{E}_{\alpha }^{\limfunc{strong}}+\sqrt{A_{2}^{\alpha }}$} & 
& \fbox{$\mathfrak{T}_{T^{\alpha }}$} &  & \fbox{$\sqrt{A_{2}^{\alpha }}$} & 
& 
\end{array}%
$}
\end{equation*}%
With reference to this diagram, we now make a sweeping and crucial claim.

The \textbf{only} two places in our proof of the main theorem in \cite%
{SaShUr6}\ where the \textbf{weak boundedness property} $\mathcal{WBP}%
_{T^{\alpha }}$ is used, is

\begin{enumerate}
\item in proving the estimates for terms $A_{1}$ and $A_{2}$ involving $%
\left\langle T_{\sigma }^{\alpha }\left( \bigtriangleup _{I}^{\sigma
}f\right) ,\bigtriangleup _{J}^{\omega }g\right\rangle _{\omega }$ that
arise in estimating the form $\mathsf{B}_{\diagup }\left( f,g\right) $ at
the top right of the schematic diagram, and

\item and in the estimates for the inner products $\left\langle T_{\sigma
}^{\alpha }\left( \bigtriangleup _{I}^{\sigma }f\right) ,\bigtriangleup
_{J}^{\omega }g\right\rangle _{\omega }$ in the form $\mathsf{T}_{\limfunc{%
far}\limfunc{below}}^{2}\left( f,g\right) $ for which $I$ are $J$ are close
in both scale and position,

\item and \emph{even then} in these two cases, only for certain child
quasicubes $I_{\theta }$ and $J_{\theta ^{\prime }}$ when they \emph{touch},
i.e. their interiors are disjoint but their closures intersect (even in just
a point). In all other instances where $\mathcal{NTV}_{\alpha }$ appears in
the schematic diagram, the weak boundedness property is \emph{not} used.
\end{enumerate}

In order to make the application of the quasiweak boundedness property in
these arguments clear, we reproduce the relevant portions of the arguments
from \cite{SaShUr6} that deal with the forms $\mathsf{B}_{\diagup }\left(
f,g\right) $ and $\mathsf{T}_{\limfunc{far}\limfunc{below}}^{2}\left(
f,g\right) $. Recall also that the parameters $\mathbf{\rho ,\tau ,r}$ in 
\cite[Definition 12 on page 40]{SaShUr6} were fixed to satisfy%
\begin{equation*}
\mathbf{\tau >r}\text{ and }\mathbf{\rho >\tau +r\ }.
\end{equation*}

\bigskip

\textbf{1}: Here is the beginning of the proof of (6.1) on page 28 dealing
with $\mathsf{B}_{\diagup }\left( f,g\right) $ in the statement of Lemma 9
in \cite{SaShUr6}.

\bigskip

\textbf{Extract from pages 28 and 29 of \cite{SaShUr6}:}

\textit{Note that in (6.1) we have used the parameter }$\mathbf{\rho }$%
\textit{\ in the exponent rather than }$\mathbf{r}$\textit{, and this is
possible because the arguments we use here only require that there are
finitely many levels of scale separating }$I$\textit{\ and }$J$\textit{. To
handle this term we first decompose it into}%
\begin{eqnarray*}
&&\left\{ \sum_{\substack{ \left( I,J\right) \in \Omega \mathcal{D}^{\sigma
}\times \Omega \mathcal{D}^{\omega }:\ J\subset 3I  \\ 2^{-\mathbf{\rho }%
}\ell \left( I\right) \leq \ell \left( J\right) \leq 2^{\mathbf{\rho }}\ell
\left( I\right) }}+\sum_{\substack{ \left( I,J\right) \in \Omega \mathcal{D}%
^{\sigma }\times \Omega \mathcal{D}^{\omega }:\ I\subset 3J  \\ 2^{-\mathbf{%
\rho }}\ell \left( I\right) \leq \ell \left( J\right) \leq 2^{\mathbf{\rho }%
}\ell \left( I\right) }}+\sum_{\substack{ \left( I,J\right) \in \Omega 
\mathcal{D}^{\sigma }\times \Omega \mathcal{D}^{\omega }  \\ 2^{-\mathbf{%
\rho }}\ell \left( I\right) \leq \ell \left( J\right) \leq 2^{\mathbf{\rho }%
}\ell \left( I\right)  \\ J\not\subset 3I\text{ and }I\not\subset 3J}}%
\right\} \left\vert \left\langle T_{\sigma }^{\alpha }\left( \bigtriangleup
_{I}^{\sigma }f\right) ,\bigtriangleup _{J}^{\omega }g\right\rangle _{\omega
}\right\vert \\
&&\ \ \ \ \ \ \ \ \ \ \equiv A_{1}+A_{2}+A_{3}.
\end{eqnarray*}%
\textit{The proof of the bound for term }$A_{3}$\textit{\ is similar to that
of the bound for the left side of (6.2), and so we will defer the bound for }%
$A_{3}$\textit{\ until after (6.2) has been proved.}

\textit{We now consider term }$A_{1}$\textit{\ as term }$A_{2}$\textit{\ is
symmetric. To handle this term we will write the quasiHaar functions }$%
h_{I}^{\sigma }$\textit{\ and }$h_{J}^{\omega }$\textit{\ as linear
combinations of the indicators of the children of their supporting
quasicubes, denoted }$I_{\theta }$\textit{\ and }$J_{\theta ^{\prime }}$%
\textit{\ respectively. Then we use the quasitesting condition on }$%
I_{\theta }$\textit{\ and }$J_{\theta ^{\prime }}$\textit{\ when they
overlap, i.e. their interiors intersect; we use the quasiweak boundedness
property on }$I_{\theta }$\textit{\ and }$J_{\theta ^{\prime }}$\textit{\
when they touch, i.e. their interiors are disjoint but their closures
intersect (even in just a point); and finally we use the }$A_{2}^{\alpha }$%
\textit{\ condition when }$I_{\theta }$\textit{\ and }$J_{\theta ^{\prime }}$%
\textit{\ are separated, i.e. their closures are disjoint. We will suppose
initially that the side length of }$J$\textit{\ is at most the side length }$%
I$\textit{, i.e. }$\ell \left( J\right) \leq \ell \left( I\right) $\textit{,
the proof for }$J=\pi I$\textit{\ being similar but for one point mentioned
below. So suppose that }$I_{\theta }$\textit{\ is a child of }$I$\textit{\
and that }$J_{\theta ^{\prime }}$\textit{\ is a child of }$J$\textit{. If }$%
J_{\theta ^{\prime }}\subset I_{\theta }$\textit{\ we have from (\ref{useful
Haar}) that,} 
\begin{eqnarray*}
\left\vert \left\langle T_{\sigma }^{\alpha }\left( \mathbf{1}_{I_{\theta
}}\bigtriangleup _{I}^{\sigma }f\right) ,\mathbf{1}_{J_{\theta ^{\prime
}}}\bigtriangleup _{J}^{\omega }g\right\rangle _{\omega }\right\vert
&\lesssim &\sup_{a,a^{\prime }\in \Gamma _{n}}\frac{\left\vert \left\langle
f,h_{I}^{\sigma ,a}\right\rangle _{\sigma }\right\vert }{\sqrt{\left\vert
I_{\theta }\right\vert _{\sigma }}}\left\vert \left\langle T_{\sigma
}^{\alpha }\left( \mathbf{1}_{I_{\theta }}\right) ,\mathbf{1}_{J_{\theta
^{\prime }}}\right\rangle _{\omega }\right\vert \frac{\left\vert
\left\langle g,h_{J}^{\omega ,a^{\prime }}\right\rangle _{\omega
}\right\vert }{\sqrt{\left\vert J_{\theta ^{\prime }}\right\vert _{\omega }}}
\\
&\lesssim &\sup_{a,a^{\prime }\in \Gamma _{n}}\frac{\left\vert \left\langle
f,h_{I}^{\sigma ,a}\right\rangle _{\sigma }\right\vert }{\sqrt{\left\vert
I_{\theta }\right\vert _{\sigma }}}\left( \int_{J_{\theta ^{\prime
}}}\left\vert T_{\sigma }^{\alpha }\left( \mathbf{1}_{I_{\theta }}\right)
\right\vert ^{2}d\omega \right) ^{\frac{1}{2}}\left\vert \left\langle
g,h_{J}^{\omega ,a^{\prime }}\right\rangle _{\omega }\right\vert \\
&\lesssim &\sup_{a,a^{\prime }\in \Gamma _{n}}\frac{\left\vert \left\langle
f,h_{I}^{\sigma ,a}\right\rangle _{\sigma }\right\vert }{\sqrt{\left\vert
I_{\theta }\right\vert _{\sigma }}}\mathfrak{T}_{T_{\alpha }}\left\vert
I_{\theta }\right\vert _{\sigma }^{\frac{1}{2}}\left\vert \left\langle
g,h_{J}^{\omega ,a^{\prime }}\right\rangle _{\omega }\right\vert \\
&\lesssim &\sup_{a,a^{\prime }\in \Gamma _{n}}\mathfrak{T}_{T_{\alpha
}}\left\vert \left\langle f,h_{I}^{\sigma ,a}\right\rangle _{\sigma
}\right\vert \left\vert \left\langle g,h_{J}^{\omega ,a^{\prime
}}\right\rangle _{\omega }\right\vert .
\end{eqnarray*}%
\textit{The point referred to above is that when }$J=\pi I$\textit{\ we
write }$\left\langle T_{\sigma }^{\alpha }\left( \mathbf{1}_{I_{\theta
}}\right) ,\mathbf{1}_{J_{\theta ^{\prime }}}\right\rangle _{\omega
}=\left\langle \mathbf{1}_{I_{\theta }},T_{\omega }^{\alpha ,\ast }\left( 
\mathbf{1}_{J_{\theta ^{\prime }}}\right) \right\rangle _{\sigma }$\textit{\
and get the dual quasitesting constant }$T_{T_{\alpha }}^{\ast }$\textit{.
If }$J_{\theta ^{\prime }}$\textit{\ and }$I_{\theta }$\textit{\ touch, then 
}$\ell \left( J_{\theta ^{\prime }}\right) \leq \ell \left( I_{\theta
}\right) $\textit{\ and we have }$J_{\theta ^{\prime }}\subset 3I_{\theta
}\setminus I_{\theta }$\textit{, and so}%
\begin{eqnarray}
\left\vert \left\langle T_{\sigma }^{\alpha }\left( \mathbf{1}_{I_{\theta
}}\bigtriangleup _{I}^{\sigma }f\right) ,\mathbf{1}_{J_{\theta ^{\prime
}}}\bigtriangleup _{J}^{\omega }g\right\rangle _{\omega }\right\vert
&\lesssim &\sup_{a,a^{\prime }\in \Gamma _{n}}\frac{\left\vert \left\langle
f,h_{I}^{\sigma ,a}\right\rangle _{\sigma }\right\vert }{\sqrt{\left\vert
I_{\theta }\right\vert _{\sigma }}}\left\vert \left\langle T_{\sigma
}^{\alpha }\left( \mathbf{1}_{I_{\theta }}\right) ,\mathbf{1}_{J_{\theta
^{\prime }}}\right\rangle _{\omega }\right\vert \frac{\left\vert
\left\langle g,h_{J}^{\omega ,a^{\prime }}\right\rangle _{\omega
}\right\vert }{\sqrt{\left\vert J_{\theta ^{\prime }}\right\vert _{\omega }}}
\label{final display} \\
&\lesssim &\sup_{a,a^{\prime }\in \Gamma _{n}}\frac{\left\vert \left\langle
f,h_{I}^{\sigma ,a}\right\rangle _{\sigma }\right\vert }{\sqrt{\left\vert
I_{\theta }\right\vert _{\sigma }}}\mathcal{WBP}_{T^{\alpha }}\sqrt{%
\left\vert I_{\theta }\right\vert _{\sigma }\left\vert J_{\theta ^{\prime
}}\right\vert _{\omega }}\frac{\left\vert \left\langle g,h_{J}^{\omega
,a^{\prime }}\right\rangle _{\omega }\right\vert }{\sqrt{\left\vert
J_{\theta ^{\prime }}\right\vert _{\omega }}}  \notag \\
&=&\sup_{a,a^{\prime }\in \Gamma _{n}}\mathcal{WBP}_{T^{\alpha }}\left\vert
\left\langle f,h_{I}^{\sigma ,a}\right\rangle _{\sigma }\right\vert
\left\vert \left\langle g,h_{J}^{\omega ,a^{\prime }}\right\rangle _{\omega
}\right\vert .  \notag
\end{eqnarray}

\bigskip

The only place where the quasiweak boundedness property $\mathcal{WBP}%
_{T^{\alpha }}$ was used above was in the second line of the display (\ref%
{final display}) when we invoked%
\begin{equation*}
\left\vert \left\langle T_{\sigma }^{\alpha }\left( \mathbf{1}_{I_{\theta
}}\right) ,\mathbf{1}_{J_{\theta ^{\prime }}}\right\rangle _{\omega
}\right\vert \leq \mathcal{WBP}_{T^{\alpha }}\sqrt{\left\vert I_{\theta
}\right\vert _{\sigma }\left\vert J_{\theta ^{\prime }}\right\vert _{\omega }%
}
\end{equation*}%
for quasicubes $I_{\theta }\in \mathfrak{C}\left( I\right) $ and $J_{\theta
^{\prime }}\in \mathfrak{C}\left( J\right) $ that touch.

\bigskip

\textbf{2}: Here is the beginning of the proof on page 41 that controls the
form $\mathsf{T}_{\limfunc{far}\limfunc{below}}\left( f,g\right) $ in \cite%
{SaShUr6}

\bigskip

\textbf{Extract from page 41 of \cite{SaShUr6}:}

\textit{The far below term }$T_{\limfunc{far}\limfunc{below}}\left(
f,g\right) $\textit{\ is bounded using the Intertwining Proposition and the
control of functional energy condition by the energy condition given in the
next two sections. Indeed, assuming these two results, we have from }$\tau
<\rho $\textit{\ that}%
\begin{eqnarray*}
\mathsf{T}_{\limfunc{far}\limfunc{below}}\left( f,g\right) &=&\sum 
_{\substack{ A,B\in \mathcal{A}  \\ B\subsetneqq A}}\sum_{\substack{ I\in 
\mathcal{C}_{A}\text{ and }J\in \mathcal{C}_{B}^{\mathbf{\tau }-\limfunc{%
shift}}  \\ J\Subset _{\mathbf{\rho },\varepsilon }I}}\left\langle T_{\sigma
}^{\alpha }\left( \bigtriangleup _{I}^{\sigma }f\right) ,\left(
\bigtriangleup _{J}^{\omega }g\right) \right\rangle _{\omega } \\
&=&\sum_{B\in \mathcal{A}}\sum_{A\in \mathcal{A}:\ B\subsetneqq A}\sum 
_{\substack{ I\in \mathcal{C}_{A}\text{ and }J\in \mathcal{C}_{B}^{\mathbf{%
\tau }-\limfunc{shift}}  \\ J\Subset _{\mathbf{\rho },\varepsilon }I}}%
\left\langle T_{\sigma }^{\alpha }\left( \bigtriangleup _{I}^{\sigma
}f\right) ,\left( \bigtriangleup _{J}^{\omega }g\right) \right\rangle
_{\omega } \\
&=&\sum_{B\in \mathcal{A}}\sum_{A\in \mathcal{A}:\ B\subsetneqq A}\sum_{I\in 
\mathcal{C}_{A}\text{ and }J\in \mathcal{C}_{B}^{\mathbf{\tau }-\limfunc{%
shift}}}\left\langle T_{\sigma }^{\alpha }\left( \bigtriangleup _{I}^{\sigma
}f\right) ,\left( \bigtriangleup _{J}^{\omega }g\right) \right\rangle
_{\omega } \\
&&-\sum_{B\in \mathcal{A}}\sum_{A\in \mathcal{A}:\ B\subsetneqq A}\sum 
_{\substack{ I\in \mathcal{C}_{A}\text{ and }J\in \mathcal{C}_{B}^{\mathbf{%
\tau }-\limfunc{shift}}  \\ J\not\Subset _{\mathbf{\rho },\varepsilon }I}}%
\left\langle T_{\sigma }^{\alpha }\left( \bigtriangleup _{I}^{\sigma
}f\right) ,\left( \bigtriangleup _{J}^{\omega }g\right) \right\rangle
_{\omega } \\
&=&\mathsf{T}_{\limfunc{far}\limfunc{below}}^{1}\left( f,g\right) -\mathsf{T}%
_{\limfunc{far}\limfunc{below}}^{2}\left( f,g\right) .
\end{eqnarray*}%
\textit{Now }$\mathsf{T}_{\limfunc{far}\text{\ }\limfunc{below}}^{2}\left(
f,g\right) $\textit{\ is bounded by }$NTV_{\alpha }$\textit{\ by Lemma 9
since }$J$\textit{\ is good if }$\bigtriangleup _{J}^{\omega }g\neq 0$%
\textit{.}

\bigskip

The only place where the quasiweak boundedness property $\mathcal{WBP}%
_{T^{\alpha }}$ was used above\footnote{%
On page 41 of \cite{SaShUr6}, there was a typo in that $J\Subset _{\mathbf{%
\rho },\varepsilon }I$ appeared in the fourth line of the display instead of 
$J\not\Subset _{\mathbf{\rho },\varepsilon }I$ as corrected here.} was in
bounding the inner products $\left\langle T_{\sigma }^{\alpha }\left(
\bigtriangleup _{I}^{\sigma }f\right) ,\left( \bigtriangleup _{J}^{\omega
}g\right) \right\rangle _{\omega }$ by Lemma 9 of \cite{SaShUr6} when in
addition $I$ and $J$ were close in both scale and position, and this reduces
to the previous extract from pages 28 and 29 of \cite{SaShUr6} treated above.

Thus we may split the sum in (\ref{bilin}) as follows:%
\begin{eqnarray*}
\mathcal{T}^{\alpha }\left( f,g\right) &=&\sum_{I\in \Omega \mathcal{D}_{%
\limfunc{good}}^{\sigma }\text{ and }J\in \Omega \mathcal{D}_{\limfunc{good}%
}^{\omega }}\mathcal{T}^{\alpha }\left( \bigtriangleup _{I}^{\sigma
}f,\bigtriangleup _{J}^{\omega }g\right) \\
&=&\left\{ \sum_{\substack{ \left( I,J\right) \in \Omega \mathcal{D}_{%
\limfunc{good}}^{\sigma }\times \Omega \mathcal{D}_{\limfunc{good}}^{\omega
}:\ J\subset 3I  \\ 2^{-\mathbf{\rho }}\ell \left( I\right) \leq \ell \left(
J\right) \leq 2^{\mathbf{\rho }}\ell \left( I\right) }}+\sum_{\substack{ %
\left( I,J\right) \in \Omega \mathcal{D}_{\limfunc{good}}^{\sigma }\times
\Omega \mathcal{D}_{\limfunc{good}}^{\omega }:\ J\subset 3I  \\ 2^{-\mathbf{%
\rho }}\ell \left( I\right) \leq \ell \left( J\right) \leq 2^{\mathbf{\rho }%
}\ell \left( I\right) }}\right\} \mathcal{T}^{\alpha }\left( \bigtriangleup
_{I}^{\sigma }f,\bigtriangleup _{J}^{\omega }g\right) +\mathcal{R}^{\alpha
}\left( f,g\right) \\
&\equiv &\left\{ A_{1}\left( f,g\right) +A_{2}\left( f,g\right) \right\} +%
\mathcal{R}^{\alpha }\left( f,g\right) ,
\end{eqnarray*}%
where we are including in the terms $A_{1}\left( f,g\right) +A_{2}\left(
f,g\right) $ the corresponding inner products from the form $\mathsf{T}_{%
\limfunc{far}\limfunc{below}}^{2}\left( f,g\right) $ to which Lemma 9 of 
\cite{SaShUr6} was applied. Then the remainder form $\mathcal{R}^{\alpha
}\left( f,g\right) $ satisfies the estimate%
\begin{equation}
\left\vert \mathcal{R}^{\alpha }\left( f,g\right) \right\vert \leq C_{\alpha
}\left( \sqrt{\mathfrak{A}_{2}^{\alpha }}+\mathfrak{T}_{T^{\alpha }}+%
\mathfrak{T}_{T^{\alpha }}^{\ast }+\mathcal{E}_{\alpha }^{\limfunc{strong}}+%
\mathcal{E}_{\alpha }^{\limfunc{strong},\ast }\right) \left\Vert
f\right\Vert _{L^{2}\left( \sigma \right) }\left\Vert g\right\Vert
_{L^{2}\left( \omega \right) }.  \label{remainder est}
\end{equation}%
The key point here is that the quasiweak boundedness constant $\mathcal{WBP}%
_{T^{\alpha }}$ does\textbf{\ not} appear on the right hand side of this
estimate, and this is because the arguments in \cite{SaShUr6} that are used
to bound $\mathcal{R}^{\alpha }\left( f,g\right) $ do not use the quasiweak
boundedness property at all, as a patient reader can verify. This
constitutes the deepest part of our argument to prove (\ref{I/T}).

We now turn to the `good-$\lambda $' argument that will substitute for the
use of the quasiweak boundedness property in (\ref{bilin}) in order to prove
(\ref{I/T}). First we observe that the constant $C$ in (\ref{Ind/touch}) can
be taken to be $2^{\mathbf{\rho }}$, and then an application of the
inequality 
\begin{equation*}
\left\vert \left\langle T_{\sigma }^{\alpha }\left( \mathbf{1}_{I_{\theta
}}\right) ,\mathbf{1}_{J_{\theta ^{\prime }}}\right\rangle _{\omega
}\right\vert \leq \mathfrak{I}_{T^{\alpha }}\sqrt{\left\vert I_{\theta
}\right\vert _{\sigma }\left\vert J_{\theta ^{\prime }}\right\vert _{\omega }%
},
\end{equation*}%
to the display in (\ref{final display}) above, shows that%
\begin{eqnarray*}
\left\vert \left\langle T_{\sigma }^{\alpha }\left( \mathbf{1}_{I_{\theta
}}\bigtriangleup _{I}^{\sigma }f\right) ,\mathbf{1}_{J_{\theta ^{\prime
}}}\bigtriangleup _{J}^{\omega }g\right\rangle _{\omega }\right\vert
&\lesssim &\sup_{a,a^{\prime }\in \Gamma _{n}}\frac{\left\vert \left\langle
f,h_{I}^{\sigma ,a}\right\rangle _{\sigma }\right\vert }{\sqrt{\left\vert
I_{\theta }\right\vert _{\sigma }}}\left\vert \left\langle T_{\sigma
}^{\alpha }\left( \mathbf{1}_{I_{\theta }}\right) ,\mathbf{1}_{J_{\theta
^{\prime }}}\right\rangle _{\omega }\right\vert \frac{\left\vert
\left\langle g,h_{J}^{\omega ,a^{\prime }}\right\rangle _{\omega
}\right\vert }{\sqrt{\left\vert J_{\theta ^{\prime }}\right\vert _{\omega }}}
\\
&\lesssim &\sup_{a,a^{\prime }\in \Gamma _{n}}\frac{\left\vert \left\langle
f,h_{I}^{\sigma ,a}\right\rangle _{\sigma }\right\vert }{\sqrt{\left\vert
I_{\theta }\right\vert _{\sigma }}}\mathfrak{I}_{T^{\alpha }}\sqrt{%
\left\vert I_{\theta }\right\vert _{\sigma }\left\vert J_{\theta ^{\prime
}}\right\vert _{\omega }}\frac{\left\vert \left\langle g,h_{J}^{\omega
,a^{\prime }}\right\rangle _{\omega }\right\vert }{\sqrt{\left\vert
J_{\theta ^{\prime }}\right\vert _{\omega }}} \\
&=&\sup_{a,a^{\prime }\in \Gamma _{n}}\mathfrak{I}_{T^{\alpha }}\left\vert
\left\langle f,h_{I}^{\sigma ,a}\right\rangle _{\sigma }\right\vert
\left\vert \left\langle g,h_{J}^{\omega ,a^{\prime }}\right\rangle _{\omega
}\right\vert .
\end{eqnarray*}%
From this we obtain the following \emph{crude} estimate valid for any $f\in
L^{2}\left( \sigma \right) $ and $g\in L^{2}\left( \omega \right) $:%
\begin{equation}
\left\vert A_{1}\left( f,g\right) +A_{2}\left( f,g\right) \right\vert \leq
C_{\alpha }\left( \sqrt{\mathfrak{A}_{2}^{\alpha }}+\mathfrak{T}_{T^{\alpha
}}+\mathfrak{T}_{T^{\alpha }}^{\ast }+\mathfrak{I}_{T^{\alpha }}\right)
\left\Vert f\right\Vert _{L^{2}\left( \sigma \right) }\left\Vert
g\right\Vert _{L^{2}\left( \omega \right) }.  \label{bilin'}
\end{equation}

\begin{definition}
We say that two quasicubes $K$ and $L$ have $\eta $\emph{-comparable} side
lengths, or simply that $\ell \left( K\right) $ and $\ell \left( L\right) $
are $\eta $\emph{-comparable}, if%
\begin{equation*}
2^{-\mathbf{\eta }}\ell \left( K\right) \leq \ell \left( L\right) \leq 2^{%
\mathbf{\eta }}\ell \left( K\right) .
\end{equation*}%
Furthermore, we say that $K$ and $L$ are $\eta $\emph{-close }if they have $%
\eta $\emph{-comparable} side lengths, and if they belong to a common
quasigrid $\Omega \mathcal{D}$ and are touching quasicubes that satisfy
either $K\subset 3L$ or $L\subset 3K$.
\end{definition}

Now consider the special indicator case $f=\mathbf{1}_{Q}$ and $g=\mathbf{1}%
_{R}$ where $Q$ and $R$ are $\rho $\emph{-close} in some $\Omega \mathcal{D}$%
. For this case we will be able to do much better than (\ref{bilin'}). In
fact, for each $0<\lambda <\frac{1}{2}$ we claim that the following `good-$%
\lambda $' inequality holds:%
\begin{equation}
\left\vert A_{1}\left( \mathbf{1}_{Q},\mathbf{1}_{R}\right) \right\vert
+\left\vert A_{2}\left( \mathbf{1}_{Q},\mathbf{1}_{R}\right) \right\vert
\leq C_{\alpha }\left( \frac{1}{\lambda }\sqrt{\mathfrak{A}_{2}^{\alpha }}+%
\mathfrak{T}_{T^{\alpha }}+\mathfrak{T}_{T^{\alpha }}^{\ast }+\sqrt[4]{%
\lambda }\mathfrak{N}_{T^{\alpha }}\right) \left\Vert \mathbf{1}%
_{Q}\right\Vert _{L^{2}\left( \sigma \right) }\left\Vert \mathbf{1}%
_{R}\right\Vert _{L^{2}\left( \omega \right) }.  \label{claim}
\end{equation}%
With (\ref{claim}) proved, we can use it and (\ref{remainder est}) to
complete the proof of the estimate for the indicator/touching property (\ref%
{I/T}) by taking expectations $\mathbb{E}_{\Omega }$ as usual:%
\begin{eqnarray*}
&&\mathbb{E}_{\Omega }\left\vert \sum_{I\in \Omega \mathcal{D}^{\sigma }%
\text{ and }J\in \Omega \mathcal{D}^{\omega }}\mathcal{T}_{\sigma }^{\alpha
}\left( \bigtriangleup _{I}^{\sigma }\mathbf{1}_{Q},\bigtriangleup
_{J}^{\omega }\mathbf{1}_{R}\right) \right\vert \\
&\leq &\mathbb{E}_{\Omega }\left( \left\vert A_{1}\right\vert +\left\vert
A_{2}\right\vert \right) +\mathbb{E}_{\Omega }\left\vert \mathcal{R}^{\alpha
}\left( \mathbf{1}_{Q},\mathbf{1}_{R}\right) \right\vert \\
&\leq &C_{\alpha }\left( \frac{1}{\lambda }\sqrt{\mathfrak{A}_{2}^{\alpha }}+%
\mathfrak{T}_{T^{\alpha }}+\mathfrak{T}_{T^{\alpha }}^{\ast }+\sqrt{\lambda }%
\mathfrak{N}_{T^{\alpha }}\right) \left\Vert \mathbf{1}_{Q}\right\Vert
_{L^{2}\left( \sigma \right) }\left\Vert \mathbf{1}_{R}\right\Vert
_{L^{2}\left( \omega \right) } \\
&&+C_{\alpha }\left( \sqrt{\mathfrak{A}_{2}^{\alpha }}+\mathfrak{T}%
_{T^{\alpha }}+\mathfrak{T}_{T^{\alpha }}^{\ast }+\mathcal{E}_{\alpha }^{%
\limfunc{strong}}+\mathcal{E}_{\alpha }^{\limfunc{strong},\ast }\right)
\left\Vert \mathbf{1}_{Q}\right\Vert _{L^{2}\left( \sigma \right)
}\left\Vert \mathbf{1}_{R}\right\Vert _{L^{2}\left( \omega \right) } \\
&\leq &C_{\alpha }\left( \frac{1}{\lambda }\sqrt{\mathfrak{A}_{2}^{\alpha }}+%
\mathfrak{T}_{T^{\alpha }}+\mathfrak{T}_{T^{\alpha }}^{\ast }+\mathcal{E}%
_{\alpha }^{\limfunc{strong}}+\mathcal{E}_{\alpha }^{\limfunc{strong},\ast }+%
\sqrt[4]{\lambda }\mathfrak{N}_{T^{\alpha }}\right) \left\Vert \mathbf{1}%
_{Q}\right\Vert _{L^{2}\left( \sigma \right) }\left\Vert \mathbf{1}%
_{R}\right\Vert _{L^{2}\left( \omega \right) },
\end{eqnarray*}%
which gives (\ref{I/T}) upon taking the supremum over such $Q$ and $R$ to get%
\begin{equation*}
\mathfrak{I}_{T^{\alpha }}\leq C_{\alpha }^{\prime }\left( \frac{1}{\lambda }%
\sqrt{\mathfrak{A}_{2}^{\alpha }}+\mathfrak{T}_{T^{\alpha }}+\mathfrak{T}%
_{T^{\alpha }}^{\ast }+\mathcal{E}_{\alpha }^{\limfunc{strong}}+\mathcal{E}%
_{\alpha }^{\limfunc{strong},\ast }+\sqrt[4]{\lambda }\mathfrak{N}%
_{T^{\alpha }}\right) .
\end{equation*}

\begin{notation}
The remainder of this paper is devoted to proving (\ref{claim}) for touching
and $\mathbf{\rho }$-close quasicubes $Q$ and $R$. To simplify notation and
geometric constructions, we consider only the case of \emph{ordinary} cubes
in $\mathcal{P}^{n}$, and note that the extension to the quasiworld is then
routine.
\end{notation}

To prove the claim (\ref{claim}) we use the \emph{parameterization by
translation} introduced above. Essentially this approach was used in the
averaging technique employed in \cite{Saw1}, which in turn was borrowed from
Fefferman and Stein \cite{FeSt}, later refined in \cite{Hyt2}, and further
refined here in this paper. It suffices to prove that%
\begin{eqnarray*}
\left\vert \mathcal{T}^{\alpha }\left( \left( \mathbf{1}_{Q}\right) _{%
\limfunc{good}},\left( \mathbf{1}_{R}\right) _{\limfunc{good}}\right)
\right\vert  &\leq &C_{\alpha }\left( \frac{1}{\lambda }\sqrt{\mathfrak{A}%
_{2}^{\alpha }}+\mathfrak{T}_{T^{\alpha }}+\mathfrak{T}_{T^{\alpha }}^{\ast
}+\mathcal{E}_{\alpha }^{\limfunc{strong}}+\mathcal{E}_{\alpha }^{\limfunc{%
strong},\ast }+\sqrt{\lambda }\mathfrak{N}_{T^{\alpha }}\right)  \\
&&\ \ \ \ \ \ \ \ \ \ \ \ \ \ \ \ \ \ \ \ \ \ \ \ \ \times \left\Vert 
\mathbf{1}_{Q}\right\Vert _{L^{2}\left( \sigma \right) }\left\Vert \mathbf{1}%
_{R}\right\Vert _{L^{2}\left( \omega \right) }\ ,
\end{eqnarray*}%
for all $Q,R\in \mathcal{P}^{n}$ that are $\mathbf{\rho }$-close, uniformly
over $Q$-good grids, and where%
\begin{equation*}
\mathcal{T}^{\alpha }\left( \left( \mathbf{1}_{Q}\right) _{\limfunc{good}%
},\left( \mathbf{1}_{R}\right) _{\limfunc{good}}\right) =\sum_{I\in \mathcal{%
D}_{Q;\limfunc{good}}^{\sigma }\text{ and }J\in \mathcal{D}_{Q;\limfunc{good}%
}^{\omega }}\mathcal{T}^{\alpha }\left( \bigtriangleup _{I}^{\sigma }\mathbf{%
1}_{Q},\bigtriangleup _{J}^{\omega }\mathbf{1}_{R}\right) .
\end{equation*}%
The grids $\mathcal{D}_{Q;\limfunc{good}}^{\sigma }=\mathcal{D}_{Q;\limfunc{%
good}}^{\omega }$ are those arising in the projections $\mathsf{P}_{Q;%
\QTR{up}{\limfunc{good}}}^{\sigma }f$ and $\mathsf{P}_{Q;\QTR{up}{\limfunc{%
good}}}^{\omega }g$ above. Moreover, due to the key observation above
regarding where the weak boundedness property arises in the proof of the
main theorem in \cite{SaShUr6}, it suffices to prove%
\begin{eqnarray*}
&&\mathbb{E}_{\Omega }\left\{ \sum_{\substack{ \left( I,J\right) \in 
\mathcal{D}_{Q;\limfunc{good}}^{\sigma }\times \mathcal{D}_{Q;\limfunc{good}%
}^{\omega }:\ J\subset 3I \\ 2^{-\mathbf{\rho }}\ell \left( I\right) \leq
\ell \left( J\right) \leq 2^{\mathbf{\rho }}\ell \left( I\right) }}+\sum
_{\substack{ \left( I,J\right) \in \mathcal{D}_{Q;\limfunc{good}}^{\sigma
}\times \mathcal{D}_{Q;\limfunc{good}}^{\omega }:\ I\subset 3J \\ 2^{-%
\mathbf{\rho }}\ell \left( I\right) \leq \ell \left( J\right) \leq 2^{%
\mathbf{\rho }}\ell \left( I\right) }}\right\} \left\vert \left\langle
T_{\sigma }^{\alpha }\left( \bigtriangleup _{I}^{\sigma }\mathbf{1}%
_{Q}\right) ,\bigtriangleup _{J}^{\omega }\mathbf{1}_{R}\right\rangle
_{\omega }\right\vert  \\
&\leq &C_{\alpha }\left( \frac{1}{\lambda }\sqrt{\mathfrak{A}_{2}^{\alpha }}+%
\mathfrak{T}_{T^{\alpha }}+\mathfrak{T}_{T^{\alpha }}^{\ast }+\sqrt[4]{%
\lambda }\mathfrak{N}_{T^{\alpha }}\right) \left\Vert \mathbf{1}%
_{Q}\right\Vert _{L^{2}\left( \sigma \right) }\left\Vert \mathbf{1}%
_{R}\right\Vert _{L^{2}\left( \omega \right) }\ ,
\end{eqnarray*}%
under the assumption that we sum over only $Q$-good cubes $I$ and $J$ that
belong to $Q$-good grids in the above sums, and where we recall that we may
realize the underlying probability space as translations of any fixed grid,
say the standard dyadic grid. Note that $R$ is contained in $3Q$, and this
accounts for our inclusion of siblings in Definition \ref{crit} above.

By symmetry it suffices to prove for all $0<\lambda <\frac{1}{2}$ that%
\begin{eqnarray}
&&\mathbb{E}_{\Omega }\sum_{\substack{ \left( I,J\right) \in \mathcal{D}_{Q;%
\limfunc{good}}^{\sigma }\times \mathcal{D}_{Q;\limfunc{good}}^{\omega }:\
J\subset 3I \\ 2^{-\mathbf{\rho }}\ell \left( I\right) \leq \ell \left(
J\right) \leq 2^{\mathbf{\rho }}\ell \left( I\right)  \\ I\text{ and }J\text{
touch}}}\left\vert \left\langle T_{\sigma }^{\alpha }\left( \bigtriangleup
_{I}^{\sigma }\mathbf{1}_{Q}\right) ,\bigtriangleup _{J}^{\omega }\mathbf{1}%
_{R}\right\rangle _{\omega }\right\vert   \label{suff} \\
&\leq &C_{\alpha }\left( \frac{1}{\lambda }\sqrt{\mathfrak{A}_{2}^{\alpha }}+%
\mathfrak{T}_{T^{\alpha }}+\mathfrak{T}_{T^{\alpha }}^{\ast }+\sqrt[4]{%
\lambda }\mathfrak{I}_{T^{\alpha }}\right) \left\Vert \mathbf{1}%
_{Q}\right\Vert _{L^{2}\left( \sigma \right) }\left\Vert \mathbf{1}%
_{R}\right\Vert _{L^{2}\left( \omega \right) }\ ,  \notag
\end{eqnarray}%
for all cubes $Q,R\in \mathcal{P}^{n}$ that are $\mathbf{\rho }$-close (we
are including the testing conditions here because we are including children $%
I_{\theta }$ and $J_{\theta ^{\prime }}$ in the display (\ref{final display}%
) that coincide as well).

\subsubsection{Three critical reductions\label{three crit}}

Now we make three critical reductions that permit the application of NTV
surgery, and lie at the core of the much better estimate (\ref{claim}).

\begin{enumerate}
\item We must have that $I$ `cuts across the boundary' of $Q$, i.e. $%
\left\vert I\cap Q\right\vert >0$ and $\left\vert I\cap Q^{c}\right\vert >0$
(or else $\bigtriangleup _{I}^{\sigma }\mathbf{1}_{Q}=0$),

\item We must have that $J$ `cuts across the boundary' of $R$, i.e. $%
\left\vert J\cap R\right\vert >0$ and $\left\vert J\cap R^{c}\right\vert >0$
(or else $\bigtriangleup _{J}^{\omega }\mathbf{1}_{R}=0$),

\item By the assumed `$Q$-goodness' in Definition \ref{crit}, together with
reductions (1) and (2) above, we \emph{cannot} have either $\ell \left(
I\right) \geq 2^{\mathbf{r}}\ell \left( Q\right) $ or $\ell \left( J\right)
\geq 2^{\mathbf{r}}\ell \left( R\right) $.
\end{enumerate}

From these reductions, we are left to prove%
\begin{eqnarray}
&&\mathbb{E}_{\Omega }\sum_{\substack{ \left( I,J\right) \in \mathcal{D}_{Q;%
\limfunc{good}}^{\sigma }\times \mathcal{D}_{Q;\limfunc{good}}^{\omega }:\
J\subset 3I \\ I\text{ and }J\text{ are }\mathbf{\rho }\text{-close} \\ \ell
\left( I\right) <2^{\mathbf{r}}\ell \left( Q\right) \text{ and }\ell \left(
J\right) <2^{\mathbf{r}}\ell \left( R\right) }}\left\vert \left\langle
T_{\sigma }^{\alpha }\left( \bigtriangleup _{I}^{\sigma }\mathbf{1}%
_{Q}\right) ,\bigtriangleup _{J}^{\omega }\mathbf{1}_{R}\right\rangle
_{\omega }\right\vert   \label{restricted} \\
&\leq &C_{\alpha }\left( \frac{1}{\lambda }\sqrt{\mathfrak{A}_{2}^{\alpha }}+%
\mathfrak{T}_{T^{\alpha }}+\mathfrak{T}_{T^{\alpha }}^{\ast }+\sqrt[4]{%
\lambda }\mathfrak{N}_{T^{\alpha }}\right) \left\Vert \mathbf{1}%
_{Q}\right\Vert _{L^{of2}\left( \sigma \right) }\left\Vert \mathbf{1}%
_{R}\right\Vert _{L^{2}\left( \omega \right) }\ ,  \notag
\end{eqnarray}%
for all $\mathbf{\rho }$-close $Q,R\in \mathcal{P}^{n}$.

The \emph{small} pairs of cubes $\left( I,J\right) $, i.e. those with both $%
\ell \left( I\right) <2^{-\mathbf{r}}\ell \left( Q\right) $ and $\ell \left(
J\right) <2^{\mathbf{r}}\ell \left( R\right) $, pose a difficulty and our
next task is to further reduce matters to proving the more restricted
estimate:%
\begin{eqnarray}
&&\mathbb{E}_{\Omega }\sum_{\substack{ \left( I,J\right) \in \mathcal{D}_{Q;%
\limfunc{good}}^{\sigma }\times \mathcal{D}_{Q;\limfunc{good}}^{\omega }:\
J\subset 3I \\ I\text{ and }J\text{ are }\mathbf{\rho }\text{-close} \\ \ell
\left( I\right) \text{ and }\ell \left( Q\right) \text{ are }r\text{%
-comparable} \\ \ell \left( J\right) \text{ and }\ell \left( R\right) \text{
are }r\text{-comparable}}}\left\vert \left\langle T_{\sigma }^{\alpha
}\left( \bigtriangleup _{I}^{\sigma }\mathbf{1}_{Q}\right) ,\bigtriangleup
_{J}^{\omega }\mathbf{1}_{R}\right\rangle _{\omega }\right\vert 
\label{more restricted} \\
&\leq &C_{\alpha }\left( \frac{1}{\lambda }\sqrt{\mathfrak{A}_{2}^{\alpha }}+%
\mathfrak{T}_{T^{\alpha }}+\mathfrak{T}_{T^{\alpha }}^{\ast }+\sqrt[4]{%
\lambda }\mathfrak{N}_{T^{\alpha }}\right) \left\Vert \mathbf{1}%
_{Q}\right\Vert _{L^{2}\left( \sigma \right) }\left\Vert \mathbf{1}%
_{R}\right\Vert _{L^{2}\left( \omega \right) }\ ,  \notag
\end{eqnarray}%
for all $Q,R\in \mathcal{P}^{n}$ that are $\mathbf{\rho }$-close. The
difference between (\ref{more restricted}) and (\ref{restricted}) is that in
(\ref{more restricted}), we do \emph{not} permit small pairs of $\left(
I,J\right) $, i.e. those with $\ell \left( I\right) <2^{-\mathbf{r}}\ell
\left( Q\right) $ or $\ell \left( J\right) <2^{-\mathbf{r}}\ell \left(
RQ\right) $.

\subsubsection{Elimination of small pairs}

To eliminate the small pairs from (\ref{restricted}), we apply for a second
time our proof from \cite{SaShUr6} as outlined above, but this time to each
inner product $\left\langle T_{\sigma }^{\alpha }\left( \bigtriangleup
_{I}^{\sigma }\mathbf{1}_{Q}\right) ,\bigtriangleup _{J}^{\omega }\mathbf{1}%
_{R}\right\rangle _{\omega }$ appearing in the sum in (\ref{restricted})
inside the expectation $\mathbb{E}_{\Omega }$. In other words, for fixed $I$%
, $J$, $Q$ and $R$, we take $f=\bigtriangleup _{I}^{\sigma }\mathbf{1}_{Q}$
and $g=\bigtriangleup _{J}^{\omega }\mathbf{1}_{R}$, and we obtain that%
\begin{eqnarray*}
&&\mathbb{E}_{\Omega }\mathbb{E}_{\Omega ^{\prime }}\left\vert \left\langle
T_{\sigma }^{\alpha }\left( \bigtriangleup _{I}^{\sigma }\mathbf{1}%
_{Q}\right) ,\bigtriangleup _{J}^{\omega }\mathbf{1}_{R}\right\rangle
_{\omega }\right\vert  \\
&\leq &C_{\alpha }\left( \sqrt{\mathfrak{A}_{2}^{\alpha }}+\mathfrak{T}%
_{T^{\alpha }}+\mathfrak{T}_{T^{\alpha }}^{\ast }++\mathcal{E}_{\alpha }^{%
\limfunc{strong}}+\mathcal{E}_{\alpha }^{\limfunc{strong},\ast
}+2^{-\varepsilon \mathbf{r}}\mathfrak{N}_{T^{\alpha }}\right) \left\Vert
\bigtriangleup _{I}^{\sigma }\mathbf{1}_{Q}\right\Vert _{L^{2}\left( \sigma
\right) }\left\Vert \bigtriangleup _{J}^{\sigma }\mathbf{1}_{R}\right\Vert
_{L^{2}\left( \omega \right) } \\
&&+\mathbb{E}_{\Omega }\mathbb{E}_{\Omega ^{\prime }}\sum_{\substack{ \left(
K,L\right) \in \mathcal{D}_{Q;\limfunc{good}}^{\prime }\times \mathcal{D}_{Q;%
\limfunc{good}}^{\prime }:\ L\subset 3K \\ K\text{ and }L\text{ are }\mathbf{%
\rho }\text{-close} \\ \ell \left( K\right) <2^{\mathbf{r}}\ell \left(
I\right) \text{ and }\ell \left( L\right) <2^{\mathbf{r}}\ell \left(
JR\right) }}\left\vert \left\langle T_{\sigma }^{\alpha }\bigtriangleup
_{K}^{\sigma }\left( \bigtriangleup _{I}^{\sigma }\mathbf{1}_{Q}\right)
,\bigtriangleup _{L}^{\omega }\left( \bigtriangleup _{J}^{\omega }\mathbf{1}%
_{R}\right) \right\rangle _{\omega }\right\vert \ ,
\end{eqnarray*}%
where here the expectation $\mathbb{E}_{\Omega ^{\prime }}$ is taken to be
independent of $\mathbb{E}_{\Omega }$.

But now we may further assume that the pair of grids $\left( \mathcal{D},%
\mathcal{D}^{\prime }\right) $, for which $\left( I,J\right) \in \mathcal{D}%
\times \mathcal{D}$ and $\left( K,L\right) \in \mathcal{D}^{\prime }\times 
\mathcal{D}^{\prime }$, are \emph{mutually good}\footnote{%
Both $I$ and $J$ belong to the common grid $\mathcal{D}$, while $K$ and $L$
belong to the independent common grid $\mathcal{D}^{\prime }$ - in contrast
to the traditional use of two independent grids where $I\in \mathcal{D}$ and 
$J\in \mathcal{D}^{\prime }$.}. Thus we cannot have $\ell \left( K\right)
<2^{-\mathbf{\rho }}\ell \left( I\right) $ because $K$ is $I$-good, and this
elimates the inclusion of small pairs $\left( K,L\right) $, i.e. those with $%
\ell \left( K\right) <2^{-\mathbf{\rho }}\ell \left( I\right) $. Note that
the term $2^{-\varepsilon \mathbf{r}}\mathfrak{N}_{T^{\alpha }}$ arises from
the bad Haar projections $\bigtriangleup _{K}^{\sigma }$ and $\bigtriangleup
_{L}^{\omega }$ of $\bigtriangleup _{I}^{\sigma }\mathbf{1}_{Q}$ and $%
\bigtriangleup _{J}^{\omega }\mathbf{1}_{R}$ respecitively. Finally, we note
that $f=\bigtriangleup _{I}^{\sigma }\mathbf{1}_{Q}$ is constant on the
children of $I$ and that $\left\Vert \bigtriangleup _{I}^{\sigma }\mathbf{1}%
_{Q}\right\Vert _{L^{2}\left( \sigma \right) }^{2}=\sum_{I^{\prime }\in 
\mathfrak{C}\left( I\right) }\int_{I^{\prime }}\left\vert \mathbb{E}%
_{I^{\prime }}^{\sigma }\mathbf{1}_{Q}-\mathbb{E}_{I}^{\sigma }\mathbf{1}%
_{Q}\right\vert ^{2}d\sigma $. Thus it suffices to prove the following
estimate,%
\begin{eqnarray*}
&&\mathbb{E}_{\Omega ^{\prime }}\sum_{\substack{ \left( K,L\right) \in 
\mathcal{D}_{Q;\limfunc{good}}^{\prime }\times \mathcal{D}_{Q;\limfunc{good}%
}^{\prime }:\ L\subset 3K \\ K\text{ and }L\text{ are }\mathbf{\rho }\text{%
-close} \\ \ell \left( K\right) \text{ and }\ell \left( I\right) \text{ are }%
\mathbf{r}\text{-comparable} \\ \ell \left( L\right) \text{ and }\ell \left(
J\right) \text{ are }\mathbf{r}\text{-comparable}}}\sum_{\substack{ %
I^{\prime }\in \mathfrak{C}\left( I\right)  \\ J^{\prime }\in \mathfrak{C}%
\left( J\right) }}\left\vert \left\langle T_{\sigma }^{\alpha
}\bigtriangleup _{K}^{\sigma }\left( \left[ \mathbb{E}_{I^{\prime }}^{\sigma
}\bigtriangleup _{I}^{\sigma }\mathbf{1}_{Q}\right] \mathbf{1}_{I^{\prime
}}\right) ,\bigtriangleup _{L}^{\omega }\left( \left[ \mathbb{E}_{J^{\prime
}}^{\omega }\bigtriangleup _{J}^{\omega }\mathbf{1}_{R}\right] \mathbf{1}%
_{J^{\prime }}\right) \right\rangle _{\omega }\right\vert  \\
&\leq &C_{\alpha }\left( \frac{1}{\lambda }\sqrt{\mathfrak{A}_{2}^{\alpha }}+%
\mathfrak{T}_{T^{\alpha }}+\mathfrak{T}_{T^{\alpha }}^{\ast }+\sqrt[4]{%
\lambda }\mathfrak{N}_{T^{\alpha }}\right) \sum_{\substack{ I^{\prime }\in 
\mathfrak{C}\left( I\right)  \\ J^{\prime }\in \mathfrak{C}\left( J\right) }}%
\left\vert \mathbb{E}_{I^{\prime }}^{\sigma }\bigtriangleup _{I}^{\sigma }%
\mathbf{1}_{Q}\right\vert \ \left\vert \mathbb{E}_{J^{\prime }}^{\omega
}\bigtriangleup _{J}^{\omega }\mathbf{1}_{R}\right\vert \ \left\Vert \mathbf{%
1}_{I^{\prime }}\right\Vert _{L^{2}\left( \sigma \right) }\left\Vert \mathbf{%
1}_{J^{\prime }}\right\Vert _{L^{2}\left( \omega \right) },
\end{eqnarray*}%
which we can write simply as%
\begin{eqnarray*}
&&\mathbb{E}_{\Omega ^{\prime }}\sum_{\substack{ \left( K,L\right) \in 
\mathcal{D}_{Q;\limfunc{good}}^{\prime }\times \mathcal{D}_{Q;\limfunc{good}%
}^{\prime }:\ L\subset 3K \\ K\text{ and }L\text{ are }\mathbf{\rho }\text{%
-close} \\ \ell \left( K\right) \text{ and }\ell \left( I^{\prime }\right) 
\text{ are }\mathbf{r}\text{-comparable} \\ \ell \left( L\right) \text{ and }%
\ell \left( J^{\prime }\right) \text{ are }\mathbf{r}\text{-comparable}}}%
\left\vert \left\langle T_{\sigma }^{\alpha }\bigtriangleup _{K}^{\sigma
}\left( \mathbf{1}_{I^{\prime }}\right) ,\bigtriangleup _{L}^{\omega }\left( 
\mathbf{1}_{J^{\prime }}\right) \right\rangle _{\omega }\right\vert  \\
&\leq &C_{\alpha }\left( \frac{1}{\lambda }\sqrt{\mathfrak{A}_{2}^{\alpha }}+%
\mathfrak{T}_{T^{\alpha }}+\mathfrak{T}_{T^{\alpha }}^{\ast }+\sqrt[4]{%
\lambda }\mathfrak{N}_{T^{\alpha }}\right) \ \left\Vert \mathbf{1}%
_{I^{\prime }}\right\Vert _{L^{2}\left( \sigma \right) }\left\Vert \mathbf{1}%
_{J^{\prime }}\right\Vert _{L^{2}\left( \omega \right) }
\end{eqnarray*}%
for each $I^{\prime }\in \mathfrak{C}\left( I\right) $ and $J^{\prime }\in 
\mathfrak{C}\left( J\right) $. Now relabel $I^{\prime }$ and $J^{\prime }$
as $Q$ and $R$ respectively (and then also $K$ and $L$ as $I$ and $J$
respectively) to obtain (\ref{more restricted}).

\subsubsection{NTV surgery}

Now in order to prove (\ref{more restricted}), we invoke the technique of
NTV surgery as used in \cite{NTV}, \cite{HyMa} and \cite{LaWi}. Given $%
0<\lambda <\frac{1}{2}$, define%
\begin{equation*}
J_{\lambda }\equiv \left\{ x\in J:\limfunc{dist}\left( x,\partial J\right)
>\lambda \ell \left( J\right) \right\} .
\end{equation*}%
Then we write%
\begin{eqnarray*}
\left\vert \left\langle T_{\sigma }^{\alpha }\left( \bigtriangleup
_{I}^{\sigma }\mathbf{1}_{Q}\right) ,\bigtriangleup _{J}^{\omega }\mathbf{1}%
_{R}\right\rangle _{\omega }\right\vert &\leq &\left\vert \left\langle
T_{\sigma }^{\alpha }\left( \bigtriangleup _{I}^{\sigma }\mathbf{1}%
_{Q}\right) ,\mathbf{1}_{J_{\lambda }}\bigtriangleup _{J}^{\omega }\mathbf{1}%
_{R}\right\rangle _{\omega }\right\vert +\left\vert \left\langle T_{\sigma
}^{\alpha }\left( \bigtriangleup _{I}^{\sigma }\mathbf{1}_{Q}\right) ,%
\mathbf{1}_{J\setminus J_{\lambda }}\bigtriangleup _{J}^{\omega }\mathbf{1}%
_{R}\right\rangle _{\omega }\right\vert \\
&\equiv &A_{1}+A_{2}.
\end{eqnarray*}%
Now we use first the fact that$\ I$ and $J_{\lambda }$ are separated by a
distance at least $\lambda \ell \left( J\right) >0$ in order to bound the
first term $A_{1}$ by 
\begin{eqnarray}
A_{1} &=&\left\vert \left\langle T_{\sigma }^{\alpha }\left( \mathbf{1}%
_{I^{\prime }}\bigtriangleup _{I}^{\sigma }\mathbf{1}_{Q}\right) ,\mathbf{1}%
_{J_{\lambda }}\bigtriangleup _{J}^{\omega }\mathbf{1}_{R}\right\rangle
_{\omega }\right\vert  \label{sep} \\
&\lesssim &\frac{1}{\lambda }\sqrt{\mathfrak{A}_{2}^{\alpha }}\left\Vert
\bigtriangleup _{I}^{\sigma }\mathbf{1}_{Q}\right\Vert _{L^{2}\left( \sigma
\right) }\left\Vert \bigtriangleup _{J}^{\omega }\mathbf{1}_{R}\right\Vert
_{L^{2}\left( \omega \right) }\leq \frac{1}{\lambda }\sqrt{\mathfrak{A}%
_{2}^{\alpha }}\left\Vert \mathbf{1}_{Q}\right\Vert _{L^{2}\left( \sigma
\right) }\left\Vert \mathbf{1}_{R}\right\Vert _{L^{2}\left( \omega \right)
}\ .  \notag
\end{eqnarray}%
We further dominate the square of the second term $A_{2}$ by%
\begin{eqnarray}
A_{2}^{2} &=&\left\vert \left\langle T_{\sigma }^{\alpha }\left(
\bigtriangleup _{I}^{\sigma }\mathbf{1}_{Q}\right) ,\mathbf{1}_{J\setminus
J_{\lambda }}\bigtriangleup _{J}^{\omega }\mathbf{1}_{R}\right\rangle
_{\omega }\right\vert ^{2}  \label{B2} \\
&=&\left\vert \left\langle T_{\sigma }^{\alpha }\left( \sum_{I^{\prime }\in 
\mathfrak{C}\left( I\right) }\mathbf{1}_{I^{\prime }}\bigtriangleup
_{I}^{\sigma }\mathbf{1}_{Q}\right) ,\mathbf{1}_{J\setminus J_{\delta
}}\sum_{J^{\prime }\in \mathfrak{C}\left( J\right) }\mathbf{1}_{J^{\prime
}}\bigtriangleup _{J}^{\omega }\mathbf{1}_{R}\right\rangle _{\omega
}\right\vert ^{2}  \notag \\
&\lesssim &\sum_{I^{\prime }\in \mathfrak{C}\left( I\right) }\sum_{J^{\prime
}\in \mathfrak{C}\left( J\right) }\left\vert \left\langle T_{\sigma
}^{\alpha }\left( \mathbf{1}_{I^{\prime }}\bigtriangleup _{I}^{\sigma }%
\mathbf{1}_{Q}\right) ,\mathbf{1}_{J^{\prime }\setminus J_{\lambda
}}\bigtriangleup _{J}^{\omega }\mathbf{1}_{R}\right\rangle _{\omega
}\right\vert ^{2}  \notag \\
&\lesssim &\sum_{I^{\prime }\in \mathfrak{C}\left( I\right) }\sum_{J^{\prime
}\in \mathfrak{C}\left( J\right) }\mathfrak{N}_{T^{\alpha }}^{2}\left\Vert 
\mathbf{1}_{I^{\prime }}\bigtriangleup _{I}^{\sigma }\mathbf{1}%
_{Q}\right\Vert _{L^{2}\left( \sigma \right) }^{2}\left\Vert \mathbf{1}%
_{J^{\prime }\setminus J_{\lambda }}\bigtriangleup _{J}^{\omega }\mathbf{1}%
_{R}\right\Vert _{L^{2}\left( \omega \right) }^{2}  \notag \\
&\lesssim &\mathfrak{N}_{T^{\alpha }}^{2}\left\Vert \mathbf{1}%
_{Q}\right\Vert _{L^{2}\left( \sigma \right) }^{2}\sum_{J^{\prime }\in 
\mathfrak{C}\left( J\right) }\left\Vert \mathbf{1}_{J^{\prime }\setminus
J_{\lambda }}\bigtriangleup _{J}^{\omega }\mathbf{1}_{R}\right\Vert
_{L^{2}\left( \omega \right) }^{2}=\mathfrak{N}_{T^{\alpha }}^{2}\left\Vert 
\mathbf{1}_{Q}\right\Vert _{L^{2}\left( \sigma \right) }^{2}\int_{J^{\prime
}\setminus J_{\lambda }}\left\vert \bigtriangleup _{J}^{\omega }\mathbf{1}%
_{R}\right\vert ^{2}d\omega \ .  \notag
\end{eqnarray}

Then we note the fact that, using the \emph{translation parameterization} of 
$\Omega $ indexed by $\gamma \in \Gamma $, we have 
\begin{equation}
\mathbb{E}_{\Omega }\left\vert R\cap \left[ \left( J+\gamma \right) ^{\prime
}\setminus \left( J+\gamma \right) _{\lambda }\right] \right\vert _{\omega
}\leq C_{\alpha }\lambda \left\vert R\right\vert _{\omega },  \label{hand}
\end{equation}%
which follows upon taking the average over certain translates $\mathcal{D}%
_{0}+\gamma $ where $\mathcal{D}_{0}$ is a fixed grid containing $J$. This
is of course equivalent to taking instead the average over\ the same
translates $\omega +\gamma $ of the measure $\omega $, and it is in this
latter form that (\ref{hand}) is evident.

Now we will apply (\ref{hand}), together with an argument to resolve the
difficulty associated with the appearance of $J$ in \emph{both} $J^{\prime
}\setminus J_{\lambda }$ and $\bigtriangleup _{J}^{\omega }\mathbf{1}_{R}$,
to obtain the following key estimate for every $0<\lambda <\frac{1}{2}$:%
\begin{equation}
\mathbb{E}_{\Omega }\int_{J^{\prime }\setminus J_{\lambda }}\left\vert
\bigtriangleup _{J}^{\omega }\mathbf{1}_{R}\right\vert ^{2}d\omega \leq
C_{\alpha }\sqrt{\lambda }\left\vert R\right\vert _{\omega },
\label{follow est}
\end{equation}%
for the expected value of the final integral on the right hand side of (\ref%
{B2}). With (\ref{follow est}) and (\ref{sep}) in hand, we will obtain that 
\begin{eqnarray*}
&&\mathbb{E}_{\Omega }\left\vert \left\langle T_{\sigma }^{\alpha }\left(
\bigtriangleup _{I}^{\sigma }\mathbf{1}_{Q}\right) ,\bigtriangleup
_{J}^{\omega }\mathbf{1}_{R}\right\rangle _{\omega }\right\vert ^{2} \\
&\lesssim &\mathbb{E}_{\Omega }\left\vert \left\langle T_{\sigma }^{\alpha
}\left( \bigtriangleup _{I}^{\sigma }\mathbf{1}_{Q}\right) ,\mathbf{1}%
_{J_{\lambda }}\bigtriangleup _{J}^{\omega }\mathbf{1}_{R}\right\rangle
_{\omega }\right\vert ^{2}+\mathbb{E}_{\Omega }\sum_{I^{\prime }\in 
\mathfrak{C}\left( I\right) }\sum_{J^{\prime }\in \mathfrak{C}\left(
J\right) }\left\vert \left\langle T_{\sigma }^{\alpha }\left( \mathbf{1}%
_{I^{\prime }}\bigtriangleup _{I}^{\sigma }\mathbf{1}_{Q}\right) ,\mathbf{1}%
_{J^{\prime }\setminus J_{\lambda }}\bigtriangleup _{J}^{\omega }\mathbf{1}%
_{R}\right\rangle _{\omega }\right\vert ^{2} \\
&\leq &C_{\alpha }^{2}\frac{1}{\lambda ^{2}}\mathfrak{A}_{2}^{\alpha
}\left\Vert \mathbf{1}_{Q}\right\Vert _{L^{2}\left( \sigma \right)
}^{2}\left\Vert \mathbf{1}_{R}\right\Vert _{L^{2}\left( \omega \right) }^{2}+%
\mathbb{E}_{\Omega }\sum_{I^{\prime }\in \mathfrak{C}\left( I\right)
}\sum_{J^{\prime }\in \mathfrak{C}\left( J\right) }\mathfrak{N}_{T^{\alpha
}}^{2}\left\Vert \mathbf{1}_{I^{\prime }}\bigtriangleup _{I}^{\sigma }%
\mathbf{1}_{Q}\right\Vert _{L^{2}\left( \sigma \right) }^{2}\left\Vert 
\mathbf{1}_{J^{\prime }\setminus J_{\lambda }}\bigtriangleup _{J}^{\omega }%
\mathbf{1}_{R}\right\Vert _{L^{2}\left( \omega \right) }^{2} \\
&\leq &C_{\alpha }^{2}\frac{1}{\lambda ^{2}}\mathfrak{A}_{2}^{\alpha
}\left\Vert \mathbf{1}_{Q}\right\Vert _{L^{2}\left( \sigma \right)
}^{2}\left\Vert \mathbf{1}_{R}\right\Vert _{L^{2}\left( \omega \right) }^{2}+%
\sqrt{\lambda }\mathfrak{N}_{T^{\alpha }}^{2}\left\Vert \mathbf{1}%
_{Q}\right\Vert _{L^{2}\left( \sigma \right) }^{2}\left\Vert \mathbf{1}%
_{R}\right\Vert _{L^{2}\left( \omega \right) }^{2}\ ,
\end{eqnarray*}%
as required. Thus the proof of (\ref{I/T}), and hence also that of the Good-$%
\lambda $ Lemma, will be complete once we have proved the estimate (\ref%
{follow est}), to which we now turn.

\begin{remark}
In the third line above we have used the norm inequality $\left\vert
\left\langle T_{\sigma }^{\alpha }f,g\right\rangle _{\omega }\right\vert
\leq \mathfrak{N}_{T^{\alpha }}\left\Vert f\right\Vert _{L^{2}\left( \sigma
\right) }\left\Vert g\right\Vert _{L^{2}\left( \omega \right) }$ with $f=%
\mathbf{1}_{I^{\prime }}\bigtriangleup _{I}^{\sigma }\mathbf{1}_{Q}$ and $g=%
\mathbf{1}_{J^{\prime }\setminus J_{\lambda }}\bigtriangleup _{J}^{\omega }%
\mathbf{1}_{R}$, and where $g$ is a constant multiple of an indicator of a
`rectangle' $J^{\prime }\setminus J_{\lambda }$. This prevents us from using
the smaller bound $\lambda \mathfrak{I}_{T^{\alpha }}^{2}$ in place of $%
\lambda \mathfrak{N}_{T^{\alpha }}^{2}$.
\end{remark}

In order to illuminate the main ideas in the proof of (\ref{follow est}), we
first prove the simplest case of dimension $n=1$. So let 
\begin{equation*}
J\setminus J_{\lambda }=J_{\lambda }^{\limfunc{left}}\cup J_{\lambda }^{%
\limfunc{right}},
\end{equation*}%
where $J_{\lambda }^{\limfunc{left}}=J_{-}\setminus J_{\lambda }$ and $%
J_{\lambda }^{\limfunc{right}}=J_{+}\setminus J_{\lambda }$, and write%
\begin{equation}
\mathbb{E}_{\Omega }\int_{J^{\prime }\setminus J_{\lambda }}\left\vert
\bigtriangleup _{J}^{\omega }\mathbf{1}_{R}\right\vert ^{2}d\omega =\mathbb{E%
}_{\Omega }\int_{J_{\lambda }^{\limfunc{left}}}\left\vert \bigtriangleup
_{J}^{\omega }\mathbf{1}_{R}\right\vert ^{2}d\omega +\mathbb{E}_{\Omega
}\int_{J_{\lambda }^{\limfunc{right}}}\left\vert \bigtriangleup _{J}^{\omega
}\mathbf{1}_{R}\right\vert ^{2}d\omega =Left+Right.  \label{expectations}
\end{equation}

Now we recall the parameterization of the expectation by translations $%
\gamma \in \Gamma _{M}^{N}$ of step size $2^{-M}$, and let $\eta =\lambda
2^{M}$ where $\lambda $ is the side length of the interval $J^{\prime
}\setminus J_{\lambda }$. Then, by using the `average of an average'
principle, we can rewrite the expectation in terms of the larger step size $%
\eta 2^{-M}$. We continue to use $\gamma $ to denote the new step size $\eta
2^{-M}$. Then we further decompose the expectation $Left$ in (\ref%
{expectations}) as%
\begin{eqnarray*}
Left &=&\mathbb{E}_{\Omega }\int_{J_{\lambda }^{\limfunc{left}}}\left\vert
\bigtriangleup _{J}^{\omega }\mathbf{1}_{R}\right\vert ^{2}d\omega =\mathbb{E%
}_{\Omega }\int_{\left( J+\gamma \right) _{\lambda }^{\limfunc{left}%
}}\left\vert \bigtriangleup _{J+\gamma }^{\omega }\mathbf{1}_{R}\right\vert
^{2}d\omega \\
&=&\mathbb{E}_{\Omega }\mathbf{1}_{\left\{ \gamma :\left( J+\gamma \right)
_{\lambda }^{\limfunc{left}}\subset R\right\} }\int_{\left( J+\gamma \right)
_{\lambda }^{\limfunc{left}}}\left\vert \bigtriangleup _{J+\gamma }^{\omega }%
\mathbf{1}_{R}\right\vert ^{2}d\omega \\
&&+\mathbb{E}_{\Omega }\mathbf{1}_{\left\{ \gamma :\left( J+\gamma \right)
_{\lambda }^{\limfunc{left}}\text{ lies to the left of }R\right\}
}\int_{\left( J+\gamma \right) _{\lambda }^{\limfunc{left}}}\left\vert
\bigtriangleup _{J+\gamma }^{\omega }\mathbf{1}_{R}\right\vert ^{2}d\omega \\
&\equiv &A_{3}+A_{4}\ ,
\end{eqnarray*}%
where because of our change of step size, we have that $\left\{ \left(
J+\gamma \right) _{\lambda }^{\limfunc{left}}\right\} _{\gamma }$ is a
pairwise disjoint covering of the top interval containing $J$ that has side
length $2^{-N}$ (see the beginning of Subsection \ref{g/b tech}\ above).

For term $A_{3}$ we use the elementary estimate 
\begin{equation*}
\left\vert \bigtriangleup _{J+\gamma }^{\omega }\mathbf{1}_{R}\right\vert
=\left\vert \mathbb{E}_{\left( J+\gamma \right) _{-}}\mathbf{1}_{R}-\mathbb{E%
}_{\left( J+\gamma \right) }\mathbf{1}_{R}\right\vert \leq 1
\end{equation*}%
together with the estimate in (\ref{hand}), to obtain%
\begin{eqnarray*}
A_{3} &=&\mathbb{E}_{\Omega }\mathbf{1}_{\left\{ \gamma :\left( J+\gamma
\right) _{\lambda }^{\limfunc{left}}\subset R\right\} }\int_{\left( J+\gamma
\right) _{\lambda }^{\limfunc{left}}}\left\vert \bigtriangleup _{J+\gamma
}^{\omega }\mathbf{1}_{R}\right\vert ^{2}d\omega \\
&\leq &\mathbb{E}_{\Omega }\left\vert R\cap \left( J+\gamma \right)
_{\lambda }^{\limfunc{left}}\right\vert _{\omega }\leq C_{\alpha }\lambda
\left\vert R\right\vert _{\omega }\ .
\end{eqnarray*}

For term $A_{4}$ we proceed as follows. We suppose that $\left( J+\gamma
\right) _{\lambda }^{\limfunc{left}}$ lies to the left of $R$, since the
case when $\left( J+\gamma \right) _{\lambda }^{\limfunc{right}}$ lies to
the right of $R$ is similar. We have%
\begin{eqnarray*}
\int_{\left( J+\gamma \right) _{\lambda }^{\limfunc{left}}}\left\vert
\bigtriangleup _{J+\gamma }^{\omega }\mathbf{1}_{R}\right\vert ^{2}d\omega
&=&\int_{\left( J+\gamma \right) _{\lambda }^{\limfunc{left}}}\left\vert 
\mathbb{E}_{\left( J+\gamma \right) _{-}}\mathbf{1}_{R}-\mathbb{E}_{\left(
J+\gamma \right) }\mathbf{1}_{R}\right\vert ^{2}d\omega \\
&=&\int_{\left( J+\gamma \right) _{\lambda }^{\limfunc{left}}}\left\vert 
\frac{\left\vert R\cap \left( J+\gamma \right) _{-}\right\vert _{\omega }}{%
\left\vert \left( J+\gamma \right) _{-}\right\vert _{\omega }}-\frac{%
\left\vert R\cap \left( J+\gamma \right) \right\vert _{\omega }}{\left\vert
J+\gamma \right\vert _{\omega }}\right\vert ^{2}d\omega \\
&\leq &2\left\vert \left( J+\gamma \right) _{\lambda }^{\limfunc{left}%
}\right\vert _{\omega }\left( \frac{\left\vert R\cap \left( J+\gamma \right)
_{-}\right\vert _{\omega }}{\left\vert \left( J+\gamma \right)
_{-}\right\vert _{\omega }}\right) ^{2}+2\left\vert \left( J+\gamma \right)
_{\lambda }^{\limfunc{left}}\right\vert _{\omega }\left( \frac{\left\vert
R\cap \left( J+\gamma \right) \right\vert _{\omega }}{\left\vert J+\gamma
\right\vert _{\omega }}\right) ^{2}\ .
\end{eqnarray*}%
We now estimate the sum of the first terms above since the sum of the second
terms can be estimated with the same argument.

For the sum of the first terms we write%
\begin{eqnarray*}
&&\sum_{\gamma :\ \left( J+\gamma \right) _{\lambda }^{\limfunc{left}}\text{
is left of }R}\left\vert \left( J+\gamma \right) _{\lambda }^{\limfunc{left}%
}\right\vert _{\omega }\left( \frac{\left\vert R\cap \left( J+\gamma \right)
_{-}\right\vert _{\omega }}{\left\vert \left( J+\gamma \right)
_{-}\right\vert _{\omega }}\right) ^{2} \\
&\leq &\left( \sum_{\gamma :\ \left( J+\gamma \right) _{\lambda }^{\limfunc{%
left}}\text{ is left of }R}\frac{\left\vert \left( J+\gamma \right)
_{\lambda }^{\limfunc{left}}\right\vert _{\omega }}{\left\vert \left(
J+\gamma \right) _{-}\right\vert _{\omega }}\frac{\left\vert R\cap \left(
J+\gamma \right) _{-}\right\vert _{\omega }}{\left\vert \left( J+\gamma
\right) _{-}\right\vert _{\omega }}\right) \left\vert R\right\vert _{\omega
}\ ,
\end{eqnarray*}%
and let $J+\gamma _{1}$ be the leftmost translate of $J$ such that%
\begin{equation}
\frac{\left\vert \left( J+\gamma \right) _{\lambda }^{\limfunc{left}%
}\right\vert _{\omega }}{\left\vert \left( J+\gamma \right) _{-}\right\vert
_{\omega }}\frac{\left\vert R\cap \left( J+\gamma \right) _{-}\right\vert
_{\omega }}{\left\vert \left( J+\gamma \right) _{-}\right\vert _{\omega }}%
>\delta ,  \label{bigger than delta}
\end{equation}%
where $\delta >0$ will be chosen later to be $\sqrt{\lambda }$. We suppose
the translations $\gamma $ are ordered to be increasing. Note that we have
both%
\begin{equation*}
1\geq \frac{\left\vert R\cap \left( J+\gamma _{1}\right) _{-}\right\vert
_{\omega }}{\left\vert \left( J+\gamma _{1}\right) _{-}\right\vert _{\omega }%
}>\delta
\end{equation*}%
and 
\begin{eqnarray*}
&&\left( J+\gamma \right) _{\lambda }^{\limfunc{left}}\subset \left(
J+\gamma _{1}\right) _{-}\ , \\
\text{if both }\gamma &>&\gamma _{1}\text{ and }\left( J+\gamma \right)
_{\lambda }^{\limfunc{left}}\text{ is left of }R.
\end{eqnarray*}%
Thus we compute that%
\begin{eqnarray}
&&\mathbb{E}_{\Omega }\int_{\left( J+\gamma \right) _{\lambda }^{\limfunc{%
left}}}\left\vert \bigtriangleup _{J+\gamma }^{\omega }\mathbf{1}%
_{R}\right\vert ^{2}d\omega =\frac{1}{\Lambda }\left\{ \sum_{\gamma <\gamma
_{1}}+\sum_{\gamma >\gamma _{1}:\ \left( J+\gamma \right) _{\lambda }^{%
\limfunc{left}}\text{ is left of }R}\right\} \int_{\left( J+\gamma \right)
_{\lambda }^{\limfunc{left}}}\left\vert \bigtriangleup _{J+\gamma }^{\omega }%
\mathbf{1}_{R}\right\vert ^{2}d\omega  \label{compute} \\
&\leq &\frac{1}{\Lambda }\sum_{\gamma <\gamma _{1}}\frac{\left\vert \left(
J+\gamma \right) _{\lambda }^{\limfunc{left}}\right\vert _{\omega
}\left\vert R\cap \left( J+\gamma \right) _{-}\right\vert _{\omega }}{%
\left\vert \left( J+\gamma \right) _{-}\right\vert _{\omega }^{2}}\left\vert
R\right\vert _{\omega }+\frac{1}{\Lambda }\sum_{\gamma >\gamma _{1}:\ \left(
J+\gamma \right) _{\lambda }^{\limfunc{left}}\text{ is left of }R}\left\vert
\left( J+\gamma \right) _{\lambda }^{\limfunc{left}}\right\vert _{\omega } 
\notag \\
&\leq &\frac{1}{\Lambda }\delta \ \#\left\{ \gamma <\gamma _{1}\right\}
\left\vert R\right\vert _{\omega }+\frac{1}{\Lambda }\left\vert \left(
J+\gamma _{1}\right) _{-}\right\vert _{\omega }\leq \delta \left\vert
R\right\vert _{\omega }+\frac{1}{\Lambda }\frac{1}{\delta }\left\vert R\cap
\left( J+\gamma _{1}\right) _{-}\right\vert _{\omega }  \notag \\
&\leq &\left( \delta +\frac{\lambda }{\delta }\right) \left\vert
R\right\vert _{\omega }=2\sqrt{\lambda }\left\vert R\right\vert _{\omega }, 
\notag
\end{eqnarray}%
if we choose $\delta =\sqrt{\lambda }$. This completes the proof of (\ref%
{follow est}) in dimension $n=1$.

\subsubsection{Higher dimensions}

In the case of $n>1$ dimensions we decompose the `corner-like' pieces $%
J^{\prime }\setminus J_{\lambda }$ for each child $J^{\prime }\in \mathfrak{C%
}\left( J\right) $ into faces $S+\gamma $ of width $\lambda $ (when $n=1$
there are only two such faces $S+\gamma $, namely the intervals $\left(
J+\gamma \right) _{\lambda }^{\limfunc{left}}$ and $\left( J+\gamma \right)
_{\lambda }^{\limfunc{right}}$). Then we apply the above argument for $%
\left( J+\gamma \right) _{\lambda }^{\limfunc{left}}$ to $S+\gamma $ for
each face $S$ of width $\lambda $ in $J^{\prime }\setminus J_{\lambda }$,
but using only translations perpendicular to the face $S$, and finally apply
the `average of an average' principle, to obtain (\ref{follow est}). We
illustrate the proof in the case $n=2$ since the general case $n\geq 2$ is
no different.

For a square $K$ in the plane, let $K_{-}$ denote the lower left child of $K$%
. Now fix squares $J$ and $R$ in the plane with $\mathbf{\rho }$-comparable
side lengths and such that $J\subset 3R$. For $\gamma \in \mathcal{H}%
_{\lambda }$, where $\mathcal{H}_{\lambda }$ is the set of \emph{horizontal}
translations $\gamma $ of step size $\lambda $ with $\left\vert \gamma
\right\vert \leq C\ell \left( R\right) $, denote by $\left( J+\gamma \right)
_{\lambda }^{\limfunc{lower}\limfunc{left}}$ the $L$-shaped `corner'%
\begin{equation*}
\left( J+\gamma \right) _{\lambda }^{\limfunc{lower}\limfunc{left}}\equiv
\left( J+\gamma \right) _{-}\setminus J_{\lambda }\ ,
\end{equation*}%
and by $\left( J+\gamma \right) _{\lambda }^{\limfunc{left}}$ the \emph{%
vertical} portion of the $L$-shaped set $\left( J+\gamma \right) _{\lambda
}^{\limfunc{lower}\limfunc{left}}$ (this is one of the faces $S+\gamma $
introduced above). We will show that%
\begin{equation}
\frac{1}{\#\mathcal{H}_{\lambda }}\sum_{\gamma \in \mathcal{H}_{\lambda
}}\int_{\left( J+\gamma \right) _{\lambda }^{\limfunc{left}}}\left\vert
\bigtriangleup _{J+\gamma }^{\omega }\mathbf{1}_{R}\right\vert ^{2}d\omega
\lesssim \sqrt{\lambda },  \label{horizontal}
\end{equation}%
where $\#\mathcal{H}_{\lambda }\approx \frac{C\ell \left( R\right) }{\lambda 
}$, and then by the `average of an average' principle we obtain (\ref{follow
est}). To prove (\ref{horizontal}) we will apply the one-dimensional
argument from the previous subsubsection, but with modifications to
accommodate the fact that $\left( J+\gamma \right) _{\lambda }^{\limfunc{left%
}}$ can now spill out over the top of $R$ as well as to the left of $R$
(recall that in the one-dimensional setting, $\left( J+\gamma \right)
_{\lambda }^{\limfunc{left}}$ occurred to the left of the interval $R$ if it
was not contained in $R$). As in dimension $n=1$, let $J+\gamma _{1}$ be the
leftmost horizontal translate of $J$ such that%
\begin{equation}
\frac{\left\vert \left( J+\gamma \right) _{\lambda }^{\limfunc{left}%
}\right\vert _{\omega }}{\left\vert \left( J+\gamma \right) _{-}\right\vert
_{\omega }}\frac{\left\vert R\cap \left( J+\gamma \right) _{-}\right\vert
_{\omega }}{\left\vert \left( J+\gamma \right) _{-}\right\vert _{\omega }}%
>\delta ,  \label{bigger than delta n}
\end{equation}%
so that we have%
\begin{equation*}
1\geq \frac{\left\vert R\cap \left( J+\gamma _{1}\right) _{-}\right\vert
_{\omega }}{\left\vert \left( J+\gamma _{1}\right) _{-}\right\vert _{\omega }%
}>\delta .
\end{equation*}

Then with notation analogous to the case $n=1$ we have a similar calculation
to that in (\ref{compute}):%
\begin{eqnarray*}
&&\frac{1}{\Lambda }\left\{ \sum_{\gamma <\gamma _{1}}+\sum_{\gamma >\gamma
_{1}:\ \left( J+\gamma \right) _{\lambda }^{\limfunc{left}}\subset \left(
J+\gamma _{1}\right) _{-}}\right\} \int_{\left( J+\gamma \right) _{\lambda
}^{\limfunc{left}}}\left\vert \bigtriangleup _{J+\gamma }^{\omega }\mathbf{1}%
_{R}\right\vert ^{2}d\omega \\
&\leq &\frac{1}{\Lambda }\sum_{\gamma <\gamma _{1}}\frac{\left\vert \left(
J+\gamma \right) _{\lambda }^{\limfunc{left}}\right\vert _{\omega
}\left\vert R\cap \left( J+\gamma \right) _{-}\right\vert _{\omega }}{%
\left\vert \left( J+\gamma \right) _{-}\right\vert _{\omega }^{2}}\left\vert
R\right\vert _{\omega }+\frac{1}{\Lambda }\sum_{\gamma >\gamma _{1}:\ \left(
J+\gamma \right) _{\lambda }^{\limfunc{left}}\subset \left( J+\gamma
_{1}\right) _{-}}\left\vert \left( J+\gamma \right) _{\lambda }^{\limfunc{%
left}}\right\vert _{\omega } \\
&\leq &\frac{1}{\Lambda }\delta \ \#\left\{ \gamma <\gamma _{1}\right\}
\left\vert R\right\vert _{\omega }+\frac{1}{\Lambda }\left\vert \left(
J+\gamma _{1}\right) _{-}\right\vert _{\omega }\leq \delta \left\vert
R\right\vert _{\omega }+\frac{1}{\Lambda }\frac{1}{\delta }\left\vert R\cap
\left( J+\gamma _{1}\right) _{-}\right\vert _{\omega } \\
&\leq &\left( \delta +\frac{\lambda }{\delta }\right) \left\vert
R\right\vert _{\omega }=2\sqrt{\lambda }\left\vert R\right\vert _{\omega },
\end{eqnarray*}%
if we choose $\delta =\sqrt{\lambda }$. Thus we have so far successfully
estimated the sum over translations $\gamma $ that satisfy either $\gamma
<\gamma _{1}$ or $\left( J+\gamma \right) _{\lambda }^{\limfunc{left}%
}\subset \left( J+\gamma _{1}\right) _{-}$.

Now we simply repeat the last step considering only the remaining horizontal
translations. Since the side lengths of $J$ and $R$ are comparable, there
are at most a fixed number of such steps left, and adding up the results,
and using the `average of an average' principle, then gives%
\begin{equation*}
\mathbb{E}_{\Omega }\int_{\left( J+\gamma \right) _{\lambda }^{\limfunc{left}%
}}\left\vert \bigtriangleup _{J+\gamma }^{\omega }\mathbf{1}_{R}\right\vert
^{2}d\omega \leq C_{\alpha }\sqrt{\lambda }.
\end{equation*}%
This completes the proof of (\ref{follow est}) in the case of dimension $n=2$%
, and as mentioned earlier, the above two-dimensional argument easily adapts
to the case $n\geq 3$.

\section{Appendix}

We assume notation as above. Define the bilinear form%
\begin{equation*}
\mathcal{B}\left( f,g\right) \equiv \left\langle T_{\sigma }^{\alpha
}f,g\right\rangle _{\omega }\ ,\ \ \ \ \ f\in L^{2}\left( \sigma \right)
,g\in L^{2}\left( \omega \right) ,
\end{equation*}%
restricted to functions $f$ and $g$ of compact support and mean zero. For
each dyadic grid $\mathcal{D}$ we then have%
\begin{equation*}
\mathcal{B}\left( f,g\right) =\sum_{I,J\in \mathcal{D}\text{ }}\left\langle
T_{\sigma }^{\alpha }\bigtriangleup _{I}^{\sigma }f,\bigtriangleup
_{J}^{\omega }g\right\rangle _{\omega }\ .
\end{equation*}%
Now define the bilinear forms%
\begin{equation*}
\mathcal{C}_{\mathcal{D}}\left( f,g\right) =\sum_{I,J\in \mathcal{D}:\ I%
\text{ and }J\text{ are }\mathbf{r}\text{-close }}\left\langle T_{\sigma
}^{\alpha }\bigtriangleup _{I}^{\sigma }f,\bigtriangleup _{J}^{\omega
}g\right\rangle _{\omega }\ ,\ \ \ \ \ f\in L^{2}\left( \sigma \right) ,g\in
L^{2}\left( \omega \right) .
\end{equation*}%
Thus the form $\mathcal{C}_{\mathcal{D}}\left( f,g\right) $ sums over those
pairs of cubes in the grid $\mathcal{D}$ that are close in both scale and
position, these being the only pairs where the need for a weak boundedness
property traditionally arises. We also consider the subbilinear form%
\begin{equation*}
\mathcal{S}_{\mathcal{D}}\left( f,g\right) =\sum_{I,J\in \mathcal{D}:\ I%
\text{ and }J\text{ are }\mathbf{r}\text{-close }}\left\vert \left\langle
T_{\sigma }^{\alpha }\bigtriangleup _{I}^{\sigma }f,\bigtriangleup
_{J}^{\omega }g\right\rangle _{\omega }\right\vert \ ,\ \ \ \ \ f\in
L^{2}\left( \sigma \right) ,g\in L^{2}\left( \omega \right) ,
\end{equation*}%
which dominates $\mathcal{C}_{\mathcal{D}}\left( f,g\right) $, i.e. $%
\left\vert \mathcal{C}_{\mathcal{D}}\left( f,g\right) \right\vert \leq 
\mathcal{S}_{\mathcal{D}}\left( f,g\right) $ for all $f\in L^{2}\left(
\sigma \right) ,g\in L^{2}\left( \omega \right) $. The main results above
can be organized into the following two part theorem.

\begin{theorem}
With notation as above, we have:

\begin{enumerate}
\item For $f$ and $g$ of compact support and mean zero, 
\begin{eqnarray*}
&&\mathbb{E}_{\Omega }\left\vert \mathcal{B}\left( f,g\right) -\mathcal{C}_{%
\mathcal{D}}\left( f,g\right) \right\vert \\
&\leq &C_{\alpha }\left( \sqrt{\mathfrak{A}_{2}^{\alpha }}+\mathfrak{T}%
_{T^{\alpha }}+\mathfrak{T}_{T^{\alpha }}^{\ast }+\mathcal{E}_{\alpha }^{%
\limfunc{strong}}+\mathcal{E}_{\alpha }^{\limfunc{strong},\ast
}+2^{-\varepsilon \mathbf{r}}\mathfrak{N}_{T^{\alpha }}\right) \left\Vert
f\right\Vert _{L^{2}\left( \sigma \right) }\left\Vert g\right\Vert
_{L^{2}\left( \omega \right) } \\
&&\ \ \ \ \ \ \ \ \ \ \ \ \ \ \ \ \ \ \ \ +C_{\alpha }\mathbb{E}_{\Omega }%
\mathcal{S}_{\mathcal{D}}\left( f,g\right) .
\end{eqnarray*}

\item For $f$ and $g$ of compact support and mean zero, and for $0<\lambda <%
\frac{1}{2}$,%
\begin{equation*}
\mathbb{E}_{\Omega }\mathcal{S}_{\mathcal{D}}\left( f,g\right) \leq
C_{\alpha }\left( \frac{1}{\lambda }\sqrt{\mathfrak{A}_{2}^{\alpha }}+%
\mathfrak{T}_{T^{\alpha }}+\mathfrak{T}_{T^{\alpha }}^{\ast }+\sqrt[4]{%
\lambda }\mathfrak{N}_{T^{\alpha }}\right) \left\Vert f\right\Vert
_{L^{2}\left( \sigma \right) }\left\Vert g\right\Vert _{L^{2}\left( \omega
\right) }.
\end{equation*}
\end{enumerate}
\end{theorem}

The reason for emphasizing the two estimates in this way, is that a
different proof strategy might produce a \emph{different} bound for $\mathbb{%
E}_{\Omega }\left\vert \mathcal{B}\left( f,g\right) -\mathcal{C}_{\mathcal{D}%
}\left( f,g\right) \right\vert $, which can then be combined with the bound
for $\mathbb{E}_{\Omega }\mathcal{S}_{\mathcal{D}}\left( f,g\right) $ to
control $\left\vert \mathcal{B}\left( f,g\right) \right\vert $. Note also
that the term $C_{\alpha }\mathbb{E}_{\Omega }\mathcal{S}_{\mathcal{D}%
}\left( f,g\right) $ is included in part (1) of the theorem, to allow for
some of the inner products in the definition of $\mathcal{C}_{\mathcal{D}%
}\left( f,g\right) $ to be added back into the form $\mathcal{B}\left(
f,g\right) -\mathcal{C}_{\mathcal{D}}\left( f,g\right) $ during the course
of the proof of estimate (1). Indeed, this was done when controlling the
form $\mathsf{T}_{\limfunc{far}\limfunc{below}}^{2}\left( f,g\right) $ above.


\begin{thebibliography}{LaSaShUrWi}
\bibitem[AsGo]{AsGo} \textsc{K. Astala and M. J. Gonzalez,}\textit{Chord-arc
curves and the Beurling transform}, Invent. Math. \textbf{205} (2016), no.
1, 57-81.

\bibitem[AsZi]{AsZi} \textsc{K. Astala and M. Zinsmeister,} \textit{Teichm%
\"{u}ller spaces and }$BMOA$\textit{,} Math. Ann. \textbf{289} (1991), no.
4, 613--625.

\bibitem[CoFe]{CoFe} \textsc{R. R. Coifman and C. L. Fefferman,} \textit{%
Weighted norm inequalities for maximal functions and singular integrals,}
Studia Math. \textbf{51} (1974), 241-250.

\bibitem[FeSt]{FeSt} \textsc{C. L. Fefferman and E. M. Stein,} \textit{Some
maximal inequalities, }Amer. J. Math. \textbf{93} (1971), 107-115.

\bibitem[HuMuWh]{HuMuWh} \textsc{R. Hunt, B. Muckenhoupt} \textsc{and R. L.
Wheeden,} \textit{Weighted norm inequalities for the conjugate function and
the Hilbert transform}, Trans. Amer. Math. Soc. \textbf{176} (1973), 227-251.

\bibitem[Hyt2]{Hyt2} \textsc{Hyt\"{o}nen, Tuomas, }\textit{The two weight
inequality for the Hilbert transform with general measures, \texttt{%
arXiv:1312.0843v2}.}

\bibitem[HyMa]{HyMa} \textsc{Hyt\"{o}nen, Tuomas and H. Martikainen,} 
\textit{On general local }$Tb$\textit{\ theorems}, \texttt{arXiv:1011.0642v1.%
}

\bibitem[IwMa]{IwMa} \textsc{T. Iwaniec and G. Martin,} \textit{%
Quasiconformal mappings and capacity,} Indiana Univ. Math. J. \textbf{40}
(1991), no. 1, 101--122.

\bibitem[Lac]{Lac} \textsc{Lacey, Michael T.,}\textit{\ Two weight
inequality for the Hilbert transform: A real variable characterization, II},
Duke Math. J. Volume \textbf{163}, Number 15 (2014), 2821-2840.

\bibitem[LaMa]{LaMa} \textsc{M. T. Lacey and H. Martikainen,} \textit{Local }%
$Tb$\textit{\ theorem with }$L^{2}$\textit{\ testing conditions and general
measures: Calder\'{o}n--Zygmund operators}, \texttt{arXiv:1310.08531v1.}

\bibitem[LaSaUr]{LaSaUr} \textsc{Lacey, Michael T., Sawyer, Eric T.,
Uriarte-Tuero, Ignacio,} \textit{Astala's conjecture on distortion of
Hausdorff measures under quasiconformal maps in the plane, }Acta Math. 
\textbf{204} (2010), 273-292.

\bibitem[LaSaUr2]{LaSaUr2} \textsc{Lacey, Michael T., Sawyer, Eric T.,
Uriarte-Tuero, Ignacio,} \textit{A Two Weight Inequality for the Hilbert
transform assuming an energy hypothesis, } Journal of Functional Analysis,
Volume \textbf{263} (2012), Issue 2, 305-363.

\bibitem[LaSaShUr3]{LaSaShUr3} \textsc{Lacey, Michael T., Sawyer, Eric T.,
Shen, Chun-Yen, Uriarte-Tuero, Ignacio,} \textit{Two weight inequality for
the Hilbert transform: A real variable characterization I}, Duke Math. J,
Volume \textbf{163}, Number 15 (2014), 2795-2820.

\bibitem[LaSaShUrWi]{LaSaShUrWi} \textsc{Lacey, Michael T., Sawyer, Eric T.,
Shen, Chun-Yen, Uriarte-Tuero, Ignacio, Wick, Brett D.,} \textit{Two weight
inequalities for the Cauchy transform from }$\mathbb{R}$ to $\mathbb{C}_{+}$%
, \textit{\texttt{arXiv:1310.4820v4}}.

\bibitem[LaWi]{LaWi} \textsc{Lacey, Michael T., Wick, Brett D.,} \textit{Two
weight inequalities for Riesz transforms: uniformly full dimension weights}, 
\textit{\texttt{arXiv:1312.6163v1,v2,v3}}.

\bibitem[NTV1]{NTV1} \textsc{F. Nazarov, S. Treil and A. Volberg,} \textit{%
The Bellman function and two weight inequalities for Haar multipliers}, J.
Amer. Math. Soc. \textbf{12} (1999), 909-928, MR\{1685781 (2000k:42009)\}.

\bibitem[NTV2]{NTV2} \textsc{Nazarov, F., Treil, S. and Volberg, A.,} 
\textit{The }$Tb$\textit{-theorem on non-homogeneous spaces,} Acta Math. 
\textbf{190} (2003), no. 2, MR 1998349 (2005d:30053).

\bibitem[NTV3]{NTV} \textsc{Nazarov, F., Treil, S., and Volberg, A.,} 
\textit{Accretive system }$Tb$\textit{-theorems on nonhomogeneous spaces},
Duke Math. J. \textbf{113} (2) (2002), 259--312.

\bibitem[NTV4]{NTV3} \textsc{F. Nazarov, S. Treil and A. Volberg,} \textit{%
Two weight estimate for the Hilbert transform and corona decomposition for
non-doubling measures}, preprint (2004) \texttt{arxiv:1003.1596}

\bibitem[NaVo]{NaVo} \textsc{F. Nazarov and A. Volberg,} \textit{The Bellman
function, the two weight Hilbert transform, and the embeddings of the model
space }$K_{\theta }$, J. d'Analyse Math. \textbf{87} (2002), 385-414.

\bibitem[NiTr]{NiTr} \textsc{N. Nikolski and S. Treil,} \textit{Linear
resolvent growth of rank one perturbation of a unitary operator does not
imply its similarity to a normal operator}, J. Anal. Math. 87 (2002),
415--431. MR1945291.

\bibitem[PeVoYu]{PeVoYu} \textsc{F. Peherstorfer, A. Volberg and P.
Yuditskii,} \textit{CMV matrices with asymptotically constant coefficients,
Szeg\"{o}--Blaschke class, Scattering Theory,} J. Funct. Anal. \textbf{256}
(2009), no. 7, 2157--2210.

\bibitem[PeVoYu1]{PeVoYu1} \textsc{F. Peherstorfer, A. Volberg and P.
Yuditskii,} \textit{Two-weight Hilbert transform and Lipschitz property of
Jacobi matrices associated to Hyperbolic polynomials,} J. Funct. Anal. 
\textbf{246} (2007), 1--30.

\bibitem[Saw1]{Saw1} \textsc{E. Sawyer,} \textit{A characterization of a
two-weight norm inequality for maximal operators}, Studia Math. \textbf{75}
(1982), 1-11, MR\{676801 (84i:42032)\}.

\bibitem[Saw]{Saw3} \textsc{E. Sawyer,} \textit{A characterization of two
weight norm inequalities for fractional and Poisson integrals}, Trans.
A.M.S. \textbf{308} (1988), 533-545, MR\{930072 (89d:26009)\}.

\bibitem[SaShUr5]{SaShUr5} \textsc{Sawyer, Eric T., Shen, Chun-Yen,
Uriarte-Tuero, Ignacio,} A \textit{two weight theorem for }$\alpha $\textit{%
-fractional singular integrals with an energy side condition and quasicube
testing}, \texttt{arXiv:1302.5093v10.}

\bibitem[SaShUr6]{SaShUr6} \textsc{Sawyer, Eric T., Shen, Chun-Yen,
Uriarte-Tuero, Ignacio,} A \textit{two weight theorem for }$\alpha $\textit{%
-fractional singular integrals with an energy side condition, quasicube
testing and common point masses}, \texttt{arXiv:1505.07816v2,v3.}

\bibitem[SaShUr7]{SaShUr7} \textsc{Sawyer, Eric T., Shen, Chun-Yen,
Uriarte-Tuero, Ignacio,} A \textit{two weight theorem for }$\alpha $\textit{%
-fractional singular integrals with an energy side condition}, Revista Mat.
Iberoam. \textbf{32} (2016), no. 1, 79-174.

\bibitem[SaShUr8]{SaShUr8} \textsc{Sawyer, Eric T., Shen, Chun-Yen,
Uriarte-Tuero, Ignacio,} The \textit{two weight }$T1$ \textit{theorem for
fractional Riesz transforms when one measure is supported on a curve}, 
\texttt{arXiv:1505.07822v4}.

\bibitem[SaShUr9]{SaShUr9} \textsc{Sawyer, Eric T., Shen, Chun-Yen,
Uriarte-Tuero, Ignacio,} A \textit{two weight fractional singular integral
theorem with side conditions, energy and }$k$\textit{-energy dispersed,} 
\texttt{arXiv:1603.04332v2}.

\bibitem[SaShUr10]{SaShUr10} \textsc{Eric T. Sawyer, Chun-Yen Shen and
Ignacio Uriarte-Tuero,} \textit{Failure of necessity of the energy condition}%
, \texttt{arXiv:16072.06071v2.}

\bibitem[Ste]{Ste} \textsc{E. M. Stein,} \textit{Harmonic Analysis:
real-variable methods, orthogonality, and oscillatory integrals},\textit{\ }%
Princeton University Press, Princeton, N. J., 1993.

\bibitem[Vol]{Vol} \textsc{A. Volberg,} \textit{Calder\'{o}n-Zygmund
capacities and operators on nonhomogeneous spaces,} CBMS Regional Conference
Series in Mathematics (2003), MR\{2019058 (2005c:42015)\}.

\bibitem[VoYu]{VoYu} \textsc{A. Volberg and P. Yuditskii,} \textit{On the
inverse scattering problem for Jacobi matrices with the spectrum on an
interval, a finite system of intervals or a Cantor set of positive length,}
Comm. Math. Phys. \textbf{226} (2002), 567--605.
\end{thebibliography}
\end{document}